\documentclass[11pt,reqno]{amsart}

\renewcommand{\Im}{\operatorname{Im}}
\newcommand{\sech}{\operatorname{sech}}
\newcommand{\defeq}{\stackrel{\rm{def}}{=}}

\usepackage{amssymb}
\usepackage{amsthm}
\usepackage{amsxtra}
\usepackage[active]{srcltx}
\usepackage{graphicx}
\usepackage{amsmath}
\usepackage{listings} 
\usepackage{float}
\usepackage{cases} 
\usepackage{threeparttable}
\usepackage{dcolumn}
\usepackage{multirow}
\usepackage{booktabs}
\usepackage{color}

\newcommand{\R}{\mathbb R}

\newtheorem{theorem}{Theorem}

\newtheorem{proposition}{Proposition}[section]

\newtheorem{conjecture}{Conjecture}  

\theoremstyle{remark}

\numberwithin{equation}{section}

\title[Space-correlated SNLS]{Behavior of solutions to the 1D focusing\\ stochastic nonlinear Schr\"odinger equation \\
with spatially correlated noise}

\author[A. Millet]{Annie Millet}
\address{SAMM (EA 4543), Universit\'e Paris 1 Panth\'eon Sorbonne, Centre Pierre Mend\`es France, 
 90 Rue de Tolbiac, 75013 Paris Cedex, France (and LPSM, UMR 8001)}
\curraddr{}
\email{annie.millet@univ-paris1.fr}

\author[A.D. Rodriguez]{Alex D. Rodriguez}
\address{Department of Mathematics  \& Statistics\\Florida International University,  Miami, FL 33199, USA}
\curraddr{}
\email{arodr1128@fiu.edu}

\author[S. Roudenko]{Svetlana Roudenko}
\address{Department of Mathematics \& Statistics\\Florida International University,  Miami, FL 33199, USA}
\curraddr{}
\email{sroudenko@fiu.edu}

\author[K. Yang]{Kai Yang}
\address{Department of Mathematics  \& Statistics\\Florida International University,  Miami, FL 33199, USA}
\curraddr{}
\email{yangk@fiu.edu}

\subjclass[2010]{35R60, 35Q55, 60H15, 60H35, 65C30, 65M06} 

\keywords{stochastic NLS, spatially correlated noise, multiplicative noise, blow-up probability, blow-up dynamics, mass-conservative numerical schemes}

\begin{document}

\begin{abstract}
We study the focusing stochastic nonlinear Schr\"odinger equation in one spatial dimension with multiplicative noise, driven by a Wiener process white in time and colored in space, in the $L^2$-critical and supercritical cases.
The mass ($L^2$-norm) is conserved due to the multiplicative noise defined via the Stratonovich integral, the energy (Hamiltonian) is not preserved. 
We first investigate how the energy is affected by various spatially correlated random perturbations. We then study the influence of the noise on the global dynamics 
measuring the probability of blow-up versus scattering behavior depending on various parameters of correlation kernels. Finally, we study the effect of the 
spatially  correlated noise on the blow-up behavior, and conclude that such random perturbations 
do not influence the blow-up dynamics, except for shifting of the blow-up center location. This is  similar to what we observed in \cite{MRY2020} for a 
space-time white driving noise. 
\end{abstract}

\maketitle

\tableofcontents

\section{Introduction} \label{Introduction}
We consider the 1D stochastic focusing nonlinear Schr\"odinger (SNLS) equation  subject to a multiplicative random perturbation.
The stochastic perturbation is  driven by   
a Wiener process, which is white in time and colored in space and 
indexed by a  parameter ranging from the deterministic 
case to space-time white noise. Our aim is to investigate how a given type of space-colored noise influences the global behavior of solutions. 

We consider two types of driving noises. The first type is  a  real $L^2$-valued   ${\mathcal Q}$-Brownian motion, 
where the trace-class covariance operator ${\mathcal Q}$ has a prescribed set of eigenfunctions. 
The decay of the corresponding eigenvalues will be either of Gaussian type (Example 1) or polynomial (Example 2). 
The second type of noise is spatially homogenous, that is, defined through a convolution with a kernel creating long range interactions. 
In this setting the correlation kernel is either a renormalized 
Riesz kernel, which is singular at the origin (Example 3), or a more regular kernel with exponential decay (Example 4). 
 
More precisely, we study the 1D focusing stochastic NLS  equation  
\begin{align}\label{E:NLS}
\begin{cases} 
iu_t+u_{xx}+|u|^{2\sigma}u= \epsilon \, u \circ dW,  
\quad (x,t)\in {\mathbb R} \times  [0,\infty),\\
u(x,0)=u_0,
\end{cases}
\end{align}
where the initial condition $u_0\in H^1(\mathbb R)$ is deterministic and the stochastic perturbation is  driven by a Wiener  process $W$ as mentioned above. 
The notation $ u(x,t) \circ W(dt,dx)$ denotes the Stratonovich integral, which  can be related to the It\^o integral 
(using the Stratonovich-It\^o correction term).  
For more details we refer the reader to \cite[p.99-100]{dBD2003} or \cite[Section 2]{MRY2020}. The reason for the Stratonovich integral is the $L^2$ norm 
conservation, which is important in applications. We mention that focusing stochastic NLS with multiplicative noise appears in various physical models, for example, 
see \cite{RGBC1995}, \cite{RGBC1994}, also \cite{dBD2005} and references therein.  
  
In the deterministic case of \eqref{E:NLS}, $\epsilon=0$, 
the local wellposedness in $H^1$ is due to Ginibre and Velo \cite{GV1979}, \cite{GV1985}, see also \cite{K1987}, \cite{T1987}, \cite{CW1988}, 
and the book \cite{Ca2003} for further details. During their lifespans, solutions to the deterministic version of \eqref{E:NLS} conserve mass $M(u)$ and energy (or Hamiltonian) $H(u)$ (also momentum, though it is not considered in this work), which are defined as 
\begin{align}\label{E:M}
M(u(t))& \defeq \|u(t)\|_{L^2}^2 = M(u_0), \\ 
H(u(t))& \defeq \frac{1}{2} \| \nabla u(t)\|_{L^2}^2 
- \frac{1}{2\sigma +2} \|u(t)\|_{L^{2\sigma +2}}^{2\sigma +2} = H(u_0).
\label{E:H}
\end{align} 
The deterministic equation has scaling invariance: if $u(t,x)$ is a solution to \eqref{E:NLS} with $\epsilon=0$, then so is $u_\lambda(t,x) = \lambda^{1/\sigma} \, u(\lambda^2 t, \lambda x)$. Under this scaling, the Sobolev $H^{s}$ norm is invariant with 
\begin{equation}\label{E:s}
s= \frac12-\frac1{\sigma}.
\end{equation}
Thus, the 1D quintic ($\sigma=2$) NLS equation is called $L^2$-critical ($s=0$) and
the NLS equation with $\sigma >2$ is referred to as the $L^2$-supercritical ($s>0$). 
In this paper we study nonlinearities with $\sigma \geq 2$ (or $s \geq 0$). 
In this case 
it is known that $H^1$ solutions may blow up in finite time, by a standard convexity argument on a finite variance $V(t) \defeq \int |x|^2 |u(t,x)|^2 \, dx$ 
(called the virial argument). Otherwise, solutions exhibit scattering (i.e., approach a linear evolution as $t \to \pm\infty$) or non-scattering (soliton) behavior, 
global in time. For that  we recall the notion of standing waves $u(t,x) = e^{it} Q(x)$. 
Here, $Q$ is the smooth, positive, decaying at infinity solution to the following equation 
\begin{equation}\label{E:Q}
-Q+Q^{\prime\prime}+Q^{2\sigma+1} = 0.
\end{equation}
This solution is unique and called the ground state; it is explicit in 1D: $Q(x) = (1+\sigma)^{\frac1{2\sigma}} \, \sech^{\frac1{\sigma}} (\sigma x)$. 

In the $L^2$-critical case ($\sigma=2$) solutions exist globally in time 
if $M(u_0) < M(Q)$ by a result of Weinstein \cite{We1983} (these solutions also scatter in $L^2$, see \cite{D2015}). If $M(u_0) \geq M(Q)$, solutions may
 blow up in finite time. The minimal mass blow-up solutions, $M(u_0)=M(Q)$, were characterized by Merle \cite{M1993}. The known stable blow-up dynamics is 
 available for solutions with the initial mass larger than that of the ground state $Q$, and has a rich history; 
  see \cite{SS1999}, \cite{YRZ2018}, \cite{YRZ2020},   \cite{F2015} (and references therein). 

In the $L^2$-supercritical case ($s>0$) the known thresholds for globally existing vs. blow-up in finite time solutions depend on the scale-invariant quantities such as 
$M(u)^{1-s} H(u)^{s}$ and $\|u\|_{L^2}^{1-s} \|\nabla u(t)\|_{L^2}^s$ and their relative size to the similar quantities for the ground state $Q$. We  summarize it 
in the following statement 
(here, $X \defeq \{ H^1(\R) ~ \mbox{if} ~~ 0<s<1;~ L^2(\R) ~\mbox{if}~~ s=0 \}$;  also note that $s<\frac12$ in 1D).

\begin{theorem}[\cite{HR2007}, \cite{DHR2008}, \cite{HR2010b},\cite{Gu2014}, \cite{FXC2011}, \cite{D2015}]\label{T:1} 
Let $u_0\in X$ and $u(t)$ be the corresponding solution to the 1D deterministic NLS equation \eqref{E:NLS} ($\epsilon=0$)
 with the maximal existence interval $(T_*, T^*)$. Suppose that 
\begin{equation}\label{E:ME}
M(u_0)^{1-s} H(u_0)^s  < M(Q)^{1-s} H(Q)^s.
\end{equation} 


{\bf 1.} 
If $\|u_0\|_{L^2}^{1-s} \|\nabla u_0\|_{L^2}^s < \|Q\|_{L^2}^{1-s} \|\nabla Q\|_{L^2}^s$, then $u(t)$ exists for all $t\in \mathbb R$ with $\|u(t)\|_{L^2}^{1-s} \|\nabla u(t)\|_{L^2}^s < \|Q\|_{L^2}^{1-s} \|\nabla Q\|_{L^2}^s$ and $u$ scatters in $X$: there exist $u_{\pm}\in X$ such that $\lim\limits_{t\rightarrow\pm\infty}\|u(t)-e^{it\Delta}u_{\pm}\|_{X}=0$.

{\bf 2.} 
If $\|u_0\|_{L^2}^{1-s} \|\nabla u_0\|_{L^2}^s > \|Q\|_{L^2}^{1-s} \|\nabla Q\|_{L^2}^s$, then $\|u(t)\|_{L^2}^{1-s} \|\nabla u(t)\|_{L^2}^s > \|Q\|_{L^2}^{1-s} \|\nabla Q\|_{L^2}^s$ for $t\in(T_*, T^*)$. Moreover, if $|x|u_0\in L^2(\R)$ (finite variance), 
then the solution $u(t)$ blows up in finite time; if $u_0$ is of infinite variance, 
then either the solution blows up in finite time or there exits a sequence of times $t_n\rightarrow +\infty$ (or $t_n\rightarrow -\infty$) 
such that $\|\nabla u(t_n)\|_{L^2(\R)}\rightarrow \infty$.
\end{theorem}
For the extensions of Theorem \ref{T:1} with \eqref{E:ME} replaced by $M(u_0)^{1-s} H(u_0)^s=M(Q)^{1-s} H(Q)^s $ see \cite{DR2010}, and 
by $M(u_0)^{1-s} H(u_0)^s > M(Q)^{1-s} H(Q)^s $ refer to \cite{HPR2010} and \cite{DR2015}. 

The focusing NLS equation subject to a stochastic perturbation has been studied in \cite{dBD2003} in the $L^2$-subcritical case, showing a global
 well-posedness for any $u_0 \in H^1$. Blow-up for $0 \leq s <1$ has been studied in \cite{dBD2005}  for a multiplicative noise. 
 The results in \cite{dBD2005} state that for $s \geq 0$ initial data with finite variance ($V(0)<\infty$) and sufficiently negative energy blow up before
  some finite time $t>0$ with positive probability \cite[Thm 4.1]{dBD2002c}. 
In the $L^2$-supercritical  case, 
the stochastic perturbation can create blow up with strictly positive probability from {\it any} initial condition before any strictly positive time,  
see \cite[Thm 5.1]{dBD2005}. 
More precisely, if the noise is nondegenerate, i.e., ${\ker}\,  \phi^*=\{0\}$, and  regular enough, then for any non-trivial initial data $u_0 \in \Sigma^2$ (here, 
$\Sigma^2 = \{f \in H^2:  |x|^2 u \in L^2\}$) and any time $t>0$, blow-up occurs with strictly positive probability before time $t$. 
Indeed, the non-degeneracy of the noise implies that before any prescribed positive time (say $t/2$), the solution $u(s,\cdot)$ will be (with strictly positive probability) 
in any given neighborhood of a function $v$ such that the deterministic NLS equation starting from $v$ will blow up before time $t/2$.

One major difference (and difficulty) compared to the deterministic setting is that energy is not necessarily conserved in the stochastic perturbations.
 In the SNLS equation \eqref{E:NLS} with multiplicative noise (defined via the Stratonovich integral) the mass is conserved a.s., see \cite{dBD2003}, 
 which allows to prove global existence of solutions
in the $L^2$-critical setting with $M(u_0) < M(Q)$; see \cite{MR2020}. 
To further understand global behavior in the $L^2$-supercritical setting  one needs to control energy (as can be seen from Theorem \ref{T:1}). 
Due to the scaling invariance and mass conservation, in \cite{MR2020} an analog of Theorem \ref{T:1} is obtained to describe behavior of solutions on some 
(random) time interval in the stochastic setting in the $L^2$-critical and supercritical cases.
In particular, 
if $M(u_0)^{1-s} H(u_0)^s  < \gamma \, M(Q)^{1-s} H(Q)^s$
for some $\gamma \in (0,1)$, and $\| u_0\|^{1-s}_{L^2} \|\nabla u_0\|^s_{L^2} < \| \nabla Q \|^{1-s}_{L^2} \|\nabla Q\|^{s}_{L^2}$, then there is no blow-up until some 
random time $\tau^*(u_0)$ such that $P(\tau^*(u_0) >T)>0$ for $T<T^*$, where $T^* = T^*(\sigma, \gamma, M(u_0), M(Q), m_\phi)$, 
where the constant $m_\phi$ defined in \eqref{def_mphi} is related to the roughness and strength of the driving noise $W$. 

While it is possible to obtain some upper bounds on the energy on a (random) time interval, 
the exact behavior of energy is not clear.  This is one of the motivations for this work, namely, to investigate the time evolution of energy, and then, 
the global behavior of solutions. Another motivation is to understand how the considered noise (colored in space and white in time) affects the 
probability  for global or finite existence, and then finally, how noise affects the blow-up dynamics, compared with the deterministic case.
We consider the discretization of energy (as well as mass), then obtain theoretical bounds on that discrete analog, including  the dependence on various 
discretization and perturbation parameters. We track the evolution of energy numerically, and then behavior of solutions, followed by studying how the noise prevents
 blow-up, or vice versa, leads towards the blow-up. After that we investigate the blow-up dynamics of solutions in both $L^2$-critical and supercritical settings and 
 obtain the rates,  profiles and other features such as locations of blow-up. Before we discuss our findings, we review stable blow-up in the deterministic setting. 

A stable blow-up in deterministic setting exhibits a self-similar structure with  specific rates and profiles. Due to the scaling invariance, the following rescaling 
of the (deterministic) equation is introduced via the new space and time coordinates  $(\tau, \xi)$ and a scaling function $L(t)$ 
(for more details see \cite{LePSS1987}, \cite{SS1999}, \cite{YRZ2019})
\begin{align}\label{E:rescale}
u(t,r)=\frac{1}{L(t)^{\frac{1}{\sigma}}}\,v(\tau, \xi), \quad \mbox{where} \quad \xi=\frac{r}{L(t)},~~r=|x|, \quad \tau=\int_0^t\frac{ds}{L(s)^2}.
\end{align}
The equation \eqref{E:NLS} (with $\epsilon=0$) then becomes
\begin{align}\label{E:v}
iv_{\tau}+ia(\tau)\left(\xi v_{\xi}+\frac{v}{\sigma}\right)+\Delta v + |v|^{2\sigma}v=0, 
\end{align}
with
\begin{align}\label{E:a}
a(\tau)=-L\frac{dL}{dt} \equiv -\frac{d \ln L}{d\tau}.
\end{align}
The limiting behavior of the parameter $a(\tau)$ in \eqref{E:v} as $\tau \to \infty$
makes a significant difference in blow-up behavior between the $L^2$-critical and $L^2$-supercritical cases. 
As $a(\tau)$ is related to $L(t)$ via \eqref{E:a}, the behavior of the rate,  
$L(t)$, is typically studied to understand the blow-up behavior (we do so in Section \ref{Numerical results}). Separating variables
 $v(\tau, \xi)=e^{i\tau}Q(\xi)$ in \eqref{E:v} and assuming that $a(\tau)$ converges to a constant $a$,
the following system is used to obtain blow-up profiles
\begin{align}\label{E:profile}
\begin{cases}
Q_{\xi\xi} -Q + ia\left(\dfrac{Q}{\sigma} + \xi Q_{\xi} \right) + |Q|^{2\sigma}Q=0, \quad \xi \in [0,\infty),\\
Q_{\xi}(0)=0,\quad  Q(0)\in \R,  
\quad Q(\infty)=0.
\end{cases}
\end{align}
Besides the conditions above, it is also required to have $|Q(\xi)|$ decrease monotonically with $\xi$, without any oscillations as $\xi \to \infty$ 
(see more on that in \cite{YRZ2019}, \cite{SS1999}, \cite{BCR1999}). In the $L^2$-critical case the above equation is simplified (due to $a$ being zero)
to the ground state equation \eqref{E:Q}. Nevertheless, the equation \eqref{E:profile} with nonzero $a$ (but asymptotically approaching zero) is investigated 
(even in the $L^2$-critical context), since the correction in the blow-up rate $L(t)$ comes exactly from that. It should be emphasized that the decay of $a(\tau)$ 
to zero in the critical case is extremely slow, which makes it very difficult to pin down the exact blow-up rate, or more precisely, the correction term in the blow-up rate,
 and it was quite some time until rigorous analytical proofs appeared    (in 1D \cite{Pe2000a}, followed by a systematic work in \cite{MR2005}-\cite{FMR2006} 
 and references therein; see  \cite{SS1999} or \cite[Introduction]{YRZ2019}). In the $L^2$-supercritical case, the convergence of $a(\tau)$ to a non-zero 
 constant is rather fast, 
and the rescaled solution converges to the blow-up profile fast as well. 
The more difficult question in this case is the profile itself, 
   since it is no longer the ground state from \eqref{E:Q}, but exactly an admissible solution (without fast oscillating decay and with an asymptotic decay of 
   $|\xi|^{-\frac1\sigma}$ as $|\xi| \to \infty$) of \eqref{E:profile}. 

Among all admissible solutions to \eqref{E:profile} there is no uniqueness as it was shown in \cite{BCR1999}, \cite{KL1995}, \cite{YRZ2019}. 
These solutions generate branches of so-called multi-bump profiles, that are labeled $Q_{J,K}$, indicating that the $J$th branch converges to the $J$th excited state,
 and $K$ is the enumeration of solutions in a branch. The solution $Q_{1,0}$, the first solution in the branch $Q_{1,K}$ (this is the branch, which converges to the 
 $L^2$-critical ground state solution $Q$ in \eqref{E:Q} as the critical index $s \to 0$), is shown (numerically) to be the profile of a stable supercritical blow-up.
  The second and third authors have been able to obtain the profile $Q_{1,0}$ in various NLS cases (see \cite{YRZ2019}, also an adaptation for a nonlocal 
  Hartree-type NLS \cite{YRZ2020}), and thus, we are able to use that in this work and compare it with the stochastic case. 


In the focusing SNLS case, in \cite{DM2002b} and \cite{DM2002a} 
numerical simulations were done when the driving noise is rough, namely, it is an approximation of space-time white noise. The effect of the multiplicative
 (and also  additive) noise is described for the propagation of solitary waves. In particular, it was noted that the blow-up mechanism  transfers energy from the
  larger scales to smaller scales, thus, allowing the mesh size to affect the formation of the blow-up in the case of multiplicative noise (the coarse mesh allows 
  formation of blow-up and the finer mesh prevents it or delays it).  The authors investigated the probability of the blow-up time and they observed that in the multiplicative case the blow-up is delayed on average. 
Other parameters' dependence (such as the dependence on the strength $\epsilon$ of the noise) is also discussed.

In this paper we use three numerical schemes 
from \cite{MRY2020}, where we studied the SNLS with perturbation driven by the space-time white noise. We apply these schemes to track energy of the stochastic 
Schr\"odinger flow in each of the four examples of noises driving the
multiplicative perturbation. After that we investigate the influence of the noise on the global behavior, in particular, probability of blow-up depending 
on the strength of the noise and spatial correlation.  In particular, we confirm  
that the noise generally delays or prevents blow-up. The more regular the noise is, the less delay or preventing effect it will have on the blow-up solutions.
Finally, we study the influence of the spatially correlated noise on the blow-up dynamics. In particular, we investigate the following conjectures. 

\begin{conjecture}[$L^2$-critical case]\label{C:1}
Let $u_0 \in H^1(\mathbb R)$ and $u(t)$, $t>0$, be the solution to the SNLS equation  \eqref{E:NLS} with $\sigma=2$ and the multiplicative noise $\epsilon \, u \circ dW$ driven by a spatially-correlated Brownian motion $W$. 

Sufficiently localized initial data with $\|u_0\|_{L^2} > \|Q\|_{L^2}$ blows up
in finite positive (random) time with positive probability.

If a solution blows up at a random positive time $T(\omega)>0$ for a given $\omega \in \Omega$, then the blow-up is characterized by a
 self-similar profile (same ground state profile $Q$ from \eqref{E:Q} as in the deterministic NLS), and for $t$ close to $T(\omega)$ 
\begin{equation}\label{E:loglog}
\|u_x(t,\cdot) \|_{L^2_x} \sim \frac1{L(t)}, \quad \mbox{ where}\quad L(t) \sim \left( \frac{2\pi(T-t)}{\ln|\ln(T-t)|} \right)^{\frac{1}{2}} \quad 
\mbox{ as} \quad {t \to T(\omega)}, 
\end{equation}
known as the {\it log-log} rate due to the double logarithmic correction in $L(t)$. 

Thus, the solution blows up in a self-similar regime with profile converging to a rescaled ground state  profile $Q$, and the core part of the solution 
$u_c(x,t)$ behaves as 
$$ 
u_c(t,x) \sim \dfrac{1}{L(t)^{\frac{1}{2}}} Q\left(\frac{x-x(t)}{L(t)}\right) e^{i\gamma(t)} 
$$
with $L(t)$ converging as in \eqref{E:loglog}, $\gamma(t) \to \gamma_0$, and  $x(t) \to x_c$ (the blow-up center). 

Furthermore, conditionally on the existence of blow-up in finite time $T(\omega)> 0$,  $x_c$ is a Gaussian random variable. 
\end{conjecture}

\begin{conjecture}[$L^2$-supercritical case]\label{C:2}
Let $u_0 \in H^1(\mathbb R)$ and $u(t)$, $t>0$, be the solution to the SNLS equation \eqref{E:NLS} with $\sigma > 2$ and the multiplicative noise
 $\epsilon \, u \,\circ\, dW$ driven by a  spatially-correlated Brownian motion $W$. 

Sufficiently localized initial data
blows up in finite positive (random) time with positive probability.

If a solution blows up at a random positive time $T(\omega)>0$ for a given $\omega \in \Omega$, then the blow-up core dynamics $u_c(x,t)$ for $t$ close to $T(\omega)$ is characterized as
\begin{equation}\label{E:blowup-super}
u_c(t,x) \sim \dfrac{1}{L(t)^{\frac1{\sigma}}} Q\left(\frac{x-x(t)}{L(t)}\right) \exp \left({i \theta(t) + \frac{i}{2a(t)}\log \frac{T}{T-t}} \right),
\end{equation}
where the blow-up profile $Q$ is the $Q_{1,0}$ solution of the equation \eqref{E:profile}, $a(t) \to a$, the specific constant corresponding to the $Q_{1,0}$ profile, 
$\theta(t) \to \theta_0$, $x(t) \to x_c$ (the blow-up center), and $L(t)=(2a(T-t))^{\frac12}$.
Consequently, a direct computation yields that for $t$ close to $T(\omega)$
\begin{equation}\label{E:rate-super}
\|u_x( t, \cdot)\|_{L_x^2} \sim \frac1{L(t)^{1-s}} = {\left(2a(T-t) \right)^{-\frac12(\frac12+\frac1{\sigma})}}.
\end{equation}
Furthermore, conditionally on the existence of blow-up in finite time $T(\omega)> 0$,  $x_c$ is a Gaussian random variable. 

Thus, the blow-up happens with a polynomial rate \eqref{E:rate-super} without correction, and with profile converging to the same blow-up profile as in the 
deterministic supercritical NLS case.
\end{conjecture}

Previously, we confirmed the above conjectures in the case of a driving space-time white noise $W$ (for both additive and multiplicative perturbations) in 
\cite{MRY2020}. We are able to confirm the above conjectures in the setting of this paper - the four examples of spatially correlated Wiener processes, which are used to define the multiplicative random  perturbations. 

The paper is organized as follows. In Section \ref{S:P} we review the mass  conservation and energy bounds in the stochastic setting, then recall the three mass-conservative numerical schemes,  one of them being also energy-conservative in the deterministic setting. 
In Section \ref{Q-Wiener} we describe the first type of the driving noise $W$,  which is a  $\mathcal Q$-Brownian motion, via two examples. This is accompanied by the upper estimates for energy in both examples, and then numerical tracking of energy. In Section \ref{homogen} we study a spatially homogeneous noise $W$  via another two examples, observing first growth and then leveling off of the energy as in the case of  $\mathcal{Q}$-Brownian motions. After that we investigate the probability of blow-up
   in Section \ref{Numerical results} and how it is influenced by the strength of the noise and a  spatial correlation parameter. Our final investigations of profiles,
    rates and center location in the blow-up dynamics are in Section \ref{blow-up dynamics}. We give conclusions in Section \ref{conclusion} with an appendix 
    containing our computations of the normal distribution of the random variable representing the location shift of the blow-up center. 

{\bf Acknowledgments.} Part of this work was done when the first author visited Florida International University. She would like to thank FIU for the hospitality 
and the financial support. A.M.'s research has been conducted within the FP2M federation (CNRS FR 2036). 
S.R. was partially supported by the NSF grant DMS-1815873/1927258 as well as part of the K.Y.'s research and travel support to work on this project came from the above grant. A.D.R. was supported by REU program under DMS-1927258 (PI: Roudenko).

\section{Preliminaries}\label{S:P}
In this section we recall the time evolution of mass and energy when equation \eqref{E:NLS} is driven by a regular noise, then define the numerical schemes and the discretized versions of the mass and energy. 

\subsection{Time dependence of mass and energy}
Let the noise ${W}=\sum_{j\geq 0} \beta_j \phi e_j$  be  real-valued and regular in the space variable,
that is, colored in space by means of a Hilbert-Schmidt operator $\phi$ from $L^2_{\R}(\R) $ to $L^2_{\R}(\R)$, with the Hilbert-Schmidt norm denoted by $\|\phi\|_{L^{0,0}_{2,\R}}$.  
Since the process $W$ is real-valued and the noise is  multiplicative, as in the deterministic case, i.e., when $\epsilon =0$, 
the equation \eqref{E:NLS} conserves mass  almost surely (see \cite[Proposition 4.4]{dBD2003}), i.e.,
\begin{align}\label{D: mass}
M(u(t))=\int_{\R} |u(t,x)|^2 dx= 
M(u_0) \quad \mbox{\rm a.s.}
\end{align} 
This is a consequence  
of  rewriting \eqref{E:NLS} using the Stratonovich-It\^o correction term
\begin{equation}\label{Ito_form}
 i d_t u(t) - \big( \Delta u(t) + |u(t)|^{2\sigma} u(t) \big) dt 
=  \; u(t) dW(t) - \frac{i}{2} F_\phi u(t) dt, 
\end{equation}
where $F_\phi(x)=\sum_{k\geq 0} \big( \phi e_k(x) \big)^2$,  
and applying the  It\^o formula.

In the deterministic case, the energy (or Hamiltonian) $H(u)$  
of the solution, defined in  \eqref{E:H}, is conserved in time. This is no longer true in a stochastic setting. 

In order to study the time evolution of energy in the stochastic framework, we have to impose stronger assumptions on the operator $\phi$.
More precisely, we require that $\phi$ is Hilbert-Schmidt from $L^2_{\R}(\R)$  to   $H^1_{\R}(\R)$, and Radonifying from $L^2_{\R}(\R)$ to
$ W^{1,\kappa}_{\R}(\R)$ for some $\kappa >2$.
As proved in \cite[Proposition 4.5]{dBD2003}, the stochastic perturbation creates a time evolution  of  energy described by the It\^o formula for the It\^o 
formulation \eqref{Ito_form} of the stochastic NLS equation \eqref{E:NLS}
\begin{align*}
H(u(t)) = &\,  H(u_0) - \mbox{\rm Im } \epsilon \sum_{j\geq 0} \int_0^t  \int_{\R} \bar{u}(s,x)\nabla u(s,x) \cdot (\nabla \phi e_j)(x) dx d\beta_j(s)   \\
& +\frac{\epsilon^2}{2} \sum_{j\geq 0} \int_0^t \int_{\R} |u(s,x)|^2 \; |\nabla (\phi e_j)|^2 dx ds.  
\end{align*}

Taking expected values, we deduce that for any $\epsilon >0$ 
\begin{equation} \label{energy_mul}
{\mathbb E}(H\big( u(t)\big) = H(u_0) + \frac{\epsilon^2}{2} {\mathbb E}  \sum_{j\geq 0} \int_0^t \int_{\R} |u(s,x)|^2 \big| (\nabla \phi e_j)(x)\big|^2 dx ds
\leq H(u_0) + \frac{\epsilon^2}{2} m_\phi M(u_0)\, t,
\end{equation}
where 
\begin{equation}\label{def_mphi} 
m_\phi \defeq  \sup_{x\in \R} \sum_{j\geq 0} | \nabla(\phi e_j)(x)|^2<\infty,
\end{equation} 
since $\phi$ is Radonifying from $L^2_{\R}(\R)$ to $\dot{W}_{\R}^{1,\infty}(\R)$.

We next describe our discretizations and the numerical schemes that we use, which preserve the discrete mass; we use those to study the effect of various types 
of space-correlated driving noises on the global behavior of solutions, including the blow-up probability before a given time $T$ and the blow-up profiles. 
The time evolution of energy is a crucial first step in this study.

\subsection{Discretizations and numerical schemes}\label{S:scheme}

Let $[-L_c,L_c]$ to be a symmetric interval of computational domain, and let $\left\lbrace x_j \right\rbrace_{j=0}^N$ be  grid points  from $-L_c$ to $L_c$ (the points $x_j$ are not necessarily  equi-distributed); denote $\Delta x_j=x_{j+1}-x_{j}$. We also use the pseudo-points $x_{-1}$ satisfying $\Delta x_{-1}=\Delta x_0$, and $x_{N+1}$ 
satisfying $\Delta x_{N-1}=\Delta x_N$. Note that $x_0=-L_c$, $x_{N}=L_c$, and in the case of a constant space mesh $\Delta x$, and for $N$ even we have 
$x_{\frac{N}{2}}=0$ and $x_{\frac{N}{2}- k}=- x_{\frac{N}{2}+k}$ for $k=0, \cdots , \frac{N}{2}$.

We recall the second order discrete differential operators for a non-constant space mesh; it replaces $\partial_{xx}$ (see \cite{MRY2020} for more details). 
Given a function $f:[-L_c,L_c]\to {\mathbb C}$, set $f_j=f(x_j)$, and from the Taylor expansion of $f(x_{j-1})$ and $f(x_{j+1})$ around  $x_j$,
 one can define the second order difference operator, which is a second order approximation of $\partial_{xx}$, as
\begin{align}\label{D: D2}
\mathcal{D}_2 f_j \defeq \frac{2}{\Delta x_{j-1}(\Delta x_{j-1}
+\Delta x_{j})} f_{j-1}- \frac{2}{\Delta x_{j-1}\Delta x_{j}} f_j +\frac{2}{(\Delta x_{j-1}+\Delta x_{j})\Delta x_{j}}f_{j+1}.
\end{align}

Let $\Delta t_m=t_{m+1}-t_m$ be the time step size from $t=t_{m}$ to $t=t_{m+1}$, $m=0,1, \dots$, and 
$u_j^m$ denote the full discretization in space and time of $u$  at time $t_m$ and location $x_j$, that is, the approximation of $u(t_m,x_j)$. 
Set $V^m_j \defeq |u^m_j|^{2\sigma}$, 
and define the mid-point in time as $u^{m+\frac{1}{2}}_j=\frac{1}{2}(u^m_j+u^{m+1}_j)$. 
   
In Sections \ref{Q-Wiener} and \ref{homogen}, simulations are done on a uniform space mesh $\Delta x$ (that is, $\Delta x_j =\Delta x$ for all $j$). 
Later in the paper, where we investigate global behavior and track the blow-up dynamics in Sections \ref{Numerical results} and \ref{blow-up dynamics}, our mesh-refinement algorithm leads to a non-uniform mesh. Therefore, we give our schemes in terms of non-uniform meshes. 

We use the following discretization schemes from \cite{MRY2020}: 
the mass-energy conservative (MEC) scheme (which is a generalization of the scheme in \cite{DM2002a} to the non-uniform mesh)  
\begin{align}		
\label{mass-energy}
i \, \dfrac{u_j^{m+1}-u_j^m}{\Delta t_m}+\mathcal{D}_2 u_j^{m+\frac{1}{2}} + \frac{1}{\sigma +1} \frac{|u^{m+1}_j|^{2(\sigma +1)}- |u^m_j|^{2(\sigma +1)}}
{|u^{m+1}_j|^2-|u^m_j|^2} \; u^{m+\frac{1}{2}}_j  =\epsilon f_j^{m+\frac{1}{2}},
\end{align}
the Crank-Nicholson (CN) scheme 
\begin{align}\label{NS: crank-nicholson}
i \, \dfrac{u_j^{m+1}-u_j^m}{\Delta t_m}+\mathcal{D}_2 u_j^{m+\frac{1}{2}}+V_j^{m+\frac{1}{2}}u^{m+\frac{1}{2}}_j=\epsilon f_j^{m+\frac{1}{2}},
\end{align}
and our linear extrapolation (LE) scheme, which uses the extrapolation to approximate the potential term $V_j^{m+\frac12}$, namely,  
\begin{align}\label{NS:relaxation}
i \, \dfrac{u_j^{m+1}-u_j^m}{\Delta t_m}+\mathcal{D}_2 u_j^{m+\frac{1}{2}}+\frac{1}{2}\left(\frac{2\Delta t_{m-1}+
\Delta t_m}{\Delta t_{m-1}} V^{m}_j- \frac{\Delta t_m}{\Delta t_{m-1}} V^{m-1}_j \right) u^{m+\frac{1}{2}}_j=\epsilon f_j^{m+\frac{1}{2}}.
\end{align}

The Neumann boundary conditions on both sides of the space interval are imposed by setting $u_{-1}=u_0$ and $u_N=u_{N+1}$ on the pseudo-points 
$x_{-1}$ and $x_{N+1}$.

We set the stochastic perturbation  as 
\begin{equation}\label{E:noise-discrete}
f^{m+\frac{1}{2}}_j \defeq \frac{1}{2} (u^m_j + u^{m+1}_j )\tilde{f}^{\, m+\frac{1}{2}}_j,
\end{equation}
where $\tilde{f}^{\, m+\frac{1}{2}}_j$ depends on the type of driving noise (four different example), which we describe next.

\section{Stochastic perturbation driven by a ${\mathcal Q}$-Wiener process}\label{Q-Wiener}

\subsection{Description of the driving noise}
 Let ${\mathcal Q}$ be a trace-class positive operator from $L^2_{\R}(\R)$ to itself. Recall that a ${\mathcal Q}$-Wiener process $W=\big\{ W(t)\}_{t\geq 0}$  is an 
$L^2_{\R}(\R)$-valued process with continuous trajectories, independent time increments,  with $W(0)=0$, and such that the distribution of $W(t)-W(s)$ is Gaussian
with mean zero and covariance operator $(t-s){\mathcal Q}$ on $L^2_{\R}(\R)$ for $0\leq s\leq t$.
 This implies that given instants $s,t \in [0,+\infty)$ and functions $u,v\in L^2_{\R}(\R)$, 
$$
{\mathbb E}\big[ (W(s),u) \; (W(t),v) \big] = \big( s\wedge t)\; ({\mathcal Q}u,v).
$$
Let $\{e_j\}_{j\geq 0}$ be an orthonormal basis of $L^2_{\R}(\R)$ such that ${\mathcal Q}\, e_j=\lambda_j e_j$ for $j\geq 1$. Then $\lambda_j > 0$ and 
$\sum_{j\geq 0} \lambda_j < \infty$. Note that the processes
$$
\beta_j(t) \defeq \frac{1}{\sqrt{\lambda_j}} \big( W(t),e_j\big), \quad t\geq 0, \quad j=1,2,\dots,
$$
are independent one-dimensional standard Brownian motions. Let $\phi:L^2_{\R}(\R) \to L^2_{\R}(\R)$ be the Hilbert-Schmidt operator defined by 
$\phi \, e_j= \sqrt{\lambda_j} e_j$. 
Then the Wiener process $W$ can be expanded as follows
\begin{equation}\label{E:W-trunc}
W(t)=\sum_{j\geq 0} \sqrt{\lambda_j}  \, \beta_j(t)\, e_j =  \sum_{j\geq 0} \beta_j(t)  \phi \, e_j. 
\end{equation}
We send the reader to \cite{daPZab} for further details. 

For practical reasons we only consider finitely many orthonormal functions $\{e_j\}_{0\leq j\leq N}$, thus, truncate the series in \eqref{E:W-trunc} accordingly. This
defines an approximation $W_N$ of $W$, namely, $W_N(t)=\sum_{j=0}^N \beta_j(t)  \phi \, e_j$. 
In order to study the energy, we need the operator  $\phi$  to be Hilbert-Schmidt from $L^2_{\R}(\R) $ to $H^1_{\R}(\R)$, and thus, require the functions $\{e_j\}$
to belong to $H^1_{\R}(\R)$. In the same spirit as in \cite{MRY2020}, we consider ``hat" functions $\{g_j\}_{j\geq 0}$ defined on the space interval $[x_j, x_{j+1}]$ as follows. Let $x_{j+\frac{1}{2}}=\frac{1}{2}\big[ x_j+x_{j+\frac{1}{2}}]$, $\Delta x_j= x_{j+1}-x_j$, 
and for $j=0, \cdots, N-1$, set
$$
g_j(x)\defeq \begin{cases}
c_j (x-x_j) \quad \mbox{\rm for} \quad x \in [x_j,x_{j+\frac{1}{2}}], \\
c_j (x_{j+1} -x) \quad \mbox{\rm for} \quad x \in [x_{j+\frac{1}{2}}, x_{j+1}], 
\end{cases}
$$ 
where $c_j \defeq \frac{2 \sqrt{3}}{(\Delta x_j)^{3/2}}$ 
is chosen to ensure  $\|g_j\|_{L^2}=1$. 
 
Given points $x_0< x_1< \cdots <x_N$, define the functions $e_j$'s, $j=0, \cdots, N$, by
\begin{align} \label{def_ej}
 \begin{cases}  e_j = &g_{j-1} 1_{[x_{j-\frac{1}{2}}, x_j]} + g_j 1_{[x_j, x_{j+\frac{1}{2}}]}, \;  1\leq j\leq N-1,  \\
  e_0=& \sqrt{2}\,  g_0 1_{[x_0, x_{\frac{1}{2}}]} , \quad e_N= \sqrt{2}\,  g_{N-1} 1_{[x_{N-\frac{1}{2}}, x_N]}.
  \end{cases}
\end{align} 
Since the functions $\{e_j\}_{j=0}^N$  have disjoint supports, they are orthogonal.
By symmetry of the functions $g_j$, we have $\|e_j\|_{L^2}=1$ for $j=0, \cdots, N$. We can now construct an orthonormal basis $\{e_k\}_{k\geq 0}$ 
of $L^2_{\R}(\R)$ containing  the above $\{e_j\}_{0\leq j\leq N}$. 
For our purposes we assume that $N$ is an even integer. For the first type of noise,  we suppose that ${\mathcal Q} e_j=\lambda_j e_j$, $j=0, ...,N$,
for some specific choice of eigenvalues $\lambda_j$. 

We then define the random variables $\tilde{f}^{\,m+\frac{1}{2}}_j$, describing the driving noise, as
$$
\tilde{f}^{\, m +\frac{1}{2}}_j \defeq \sqrt{\lambda_j} \; \frac{\sqrt{3}}{2}  
  \frac{\big[ \sqrt{\Delta x_{j-1}}+\sqrt{\Delta x_j} \big] } {\sqrt{\Delta t_m} 
   \big[  \Delta x_{j-1} + \Delta x_j\big] } \chi^{m+\frac{1}{2}}_j,
$$
where the random variables $\{ \chi^{m+\frac{1}{2}}_j : m=0,\dots, M-1, j=0,\dots, N\}$  are independent Gaussian random variables ${\mathcal N}(0,1)$.
This is consistent with \cite{MRY2020}, since  for the space-time white noise, all eigenvalues $\lambda_j$ are equal to 1.
The difference with the scheme used in \cite{MRY2020} is  that,  when moving away from the origin, the effect of the noise is reduced  by
 the factor  $\sqrt{\lambda_j}$, 
which in the following examples will depend on the distance between  $x_j$ and  0. 

We  consider two types of eigenvalues $\lambda_j=\Phi_\beta(|x_j|)$, 
defined in terms of  a function $\Phi_\beta(|x|)$,  which has either an
{\it exponential} (Gaussian-type) 
or a {\it polynomial} decay as $|x|$ grows. The positive parameter $\beta$ enables us to tune the decay.

\subsubsection{Example 1: {\bf Gaussian-type decay }} 

We set 
$$
\quad \Phi^{(1)}_\beta(x)= e^{ -(1-\beta) x^2 } ~~\mbox{\rm for}~~ \beta \in [0,1].
$$
First, observe that when $\beta \in [0,1)$, up to some normalizing constant,  $\Phi^{(1)}_\beta$ is a centered Gaussian kernel with variance $\frac{1}{2(1-\beta)}$. 
Thus, when $\beta$ approaches 1, it becomes more spread out. Hence, when $\beta=1$, the kernel is a constant function $\Phi^{(1)}_1=1$ and  our noise $W$ becomes an approximation $W_N$ of the space-time white noise, studied in \cite{DM2002a} and  \cite{MRY2020}. 

We define the operator $\phi^{(1)}_\beta$ as 
$$
\phi^{(1)}_\beta  e_j = 
\begin{cases}
\left(\Phi^{(1)}_\beta(|x_j|) \right)^{\frac{1}{2}}\, e_j, \quad & \mbox{for} ~~ j=0, \cdots, N, \\   
0, &\mbox{otherwise}.
\end{cases} 
$$ 
For $N$ even, a constant space mesh $\Delta x_j$, equal to $\Delta x  $ and  $\beta \in [0,1)$, 
$\phi^{(1)}_\beta$  is Hilbert-Schmidt from $L^2_{\R}(\R)$ to $H^1_{\R}(\R)$ and Radonifying from $L^2_{\R}(\R)$ to $\dot{W}^{1,\infty}_{\R}(\R)$. 
We then have  $L_c=\frac{N}{2} \Delta x$, and 
since $\{\Phi^{(1)}_\beta(j\Delta x)\}_{j>0}$ is decreasing, we deduce 
$$  
1+\frac{2}{\Delta x} \int_{\Delta x}^{L_c+\Delta x} e^{-(1-\beta) x^2} dx \leq \| \phi^{(1)}_\beta\|_{L^{0,0}_{2, \R}}^2 \leq 
1+\frac{2}{\Delta x} \int_0^{L_c} e^{-(1-\beta) x^2} dx.
$$
As $\Delta x\to 0$ and  $L_c\to \infty$, we get 
$$  
\| \phi^{(1)}_\beta\|_{L^{0,0}_{2, \R}}^2 
\sim  \frac{1}{\Delta x} \Big( \frac{\pi}{1-\beta}\Big)^{\frac{1}{2}},
\qquad 
\| \phi^{(1)}_\beta\|_{L^{0,1}_{2,\R}}^2 
\sim  \frac{12}{(\Delta x)^3} \Big( \frac{\pi}{1-\beta}\Big)^{\frac{1}{2}},
$$ 
and 
$$
m_{\phi^{(1)}_\beta} \sim \frac{12}{(\Delta x)^4} \Big( \frac{\pi}{1-\beta}\Big)^{\frac{1}{2}},
$$
which appears in the upper estimate  \eqref{energy_mul}. 

\subsubsection{  Example 2: {\bf Polynomial decay}} 
Fix a real number  $n\geq 1$, and  set
$$
\Phi^{(2)}_\beta(x)= \frac{1}{\big( 1+|x|\big)^{n(1-\beta)}} ~~~ \mbox{\rm for } ~~ \beta \in [0,1].
$$
(To ease notations, $n$ is omitted on the left-hand side.)
Note that when $\beta=0$, the decay is of the order $|x|^{-n}$ for large values of $|x|$,
the fastest in this setting, 
and as $\beta$ decreases, the noise becomes more regular. The parameter $n$ enables us to tune this decay. 
 
Let $\phi^{(2)}_\beta$ be the operator from $L^2_{\R}(\R)$ to $H^1_{\R}(\R)\cap L^\infty_{\R}(\R)$ defined by 
$$
\phi^{(2)}_\beta(e_j)= 
\begin{cases}
\left( \Phi^{(2)}_\beta(x_j)\right)^{\frac{1}{2}} e_j,
\quad & \mbox{for} ~~ j=0, \cdots, N, \\   
0, &\mbox{otherwise}.
\end{cases} 
$$ 
Note that if $\beta=1$, the operator $\phi^{(2)}_1$ is the identity  when restricted to ${\rm span}\, (e_0, ..., e_N)$.
This is the covariance  of the projection $W_N$  of the space-time white noise on that subspace (which was used in \cite{DM2002a} and \cite{MRY2020}).
As in the previous example, we suppose that  $N$ is even, and the space mesh is uniform (thus, equal to $\Delta x$) to obtain estimates
of various operator norms of $\phi^{(2)}_\beta$. 
We have 
$$ 
\|\phi^{(2)}_\beta\|_{L^{0,0}_{2,\R}}^2 = \sum_{j=0}^N \Phi^{(2)}_\beta(x_j) = 1+ 2\, \sum_{j=1}^{{N}/{2}} (1+j\Delta x)^{-n(1-\beta)}.
$$
We bound the last term (noting that $\{\Phi^{(2)}_\beta(j\Delta x)\}_j$ is decreasing) as 
$$ 
\int_{\Delta x}^{L_c+\Delta x} (1+x)^{-n(1-\beta)} dx \leq \Delta x\sum_{j=1}^{\frac{N}{2}} (1+j \Delta x)^{-n(1-\beta)} \leq \int_0^{L_c} (1+x)^{-n(1-\beta)} dx.
$$
Hence, for a fixed $L_c$, as $\Delta x\to 0$, we deduce
$$
\|\phi^{(2)}_\beta \|_{L^{0,0}_{2,\R}}^2 \sim \frac{1}{\Delta x} \int_{-L_c}^{L_c} (1+|x|)^{-n(1-\beta)} dx.
$$
Recalling a basic fact that the indefinite integral $I(a)\defeq\int_{-\infty}^\infty (1+|x|)^{-a} dx$ 
converges if and only if $a>1$, to the value $I(a)=\frac{2}{a-1}$, we obtain that as $\Delta x\to 0$ and $L_c\to \infty$ 
$$ 
\|\phi^{(2)}_\beta \|_{L^{0,0}_{2,\R}}^2 \sim \frac{2}{\Delta x \, (n-1-n\beta)} \quad \mbox{\rm if and only if } \beta \in \Big[0,\frac{n-1}{n}\Big) .
$$
A similar computation  for the same range $0 \leq \beta < \frac{n-1}{n}$ yields
$$
\|\phi^{(2)}_\beta \|_{L^{0,1}_{2,\R}}^2 \sim \frac{24 }{(\Delta x)^3\, (n-1-n\beta)},
$$
and 
$$
\quad m_{\phi^{(2)}_\beta} \sim   \frac{24}{(\Delta x)^4\,  (n-1-n\beta)}.
$$
Note that for $\beta \in \big[ 0, 1-\frac{1}{n}\big)$, the above upper estimates for  Hilbert-Schmidt and Radonifying norms are insensitive to the length $L_c$, however, depend on the space mesh $\Delta x$. We remark that in this range of $\beta$ we have a discretization of a ${\mathcal Q}$-Brownian motion. 

For $\beta(n) = \frac{n-1}{n}$ we have  as $\Delta x\to 0$
$$
\|\phi^{(2)}_{\beta(n)} \|_{L^{0,0}_{2,\R}}^2 \sim \frac{2}{\Delta x} \ln(L_c+1), \quad \|\phi^{(2)}_{\beta(n)} \|_{L^{0,1}_{2,\R}}^2 \sim \frac{24}{(\Delta x)^3} \ln( L_c+1), ~~ \mbox{\rm and } \; 
m_{\phi^{(2)}_{\beta(n)}} \sim   \frac{24}{(\Delta x)^4}\ln(L_c+1).
$$

Finally, for $\beta \in \big( \frac{n-1}{n} ,1]$ we have  
$$
\int_{-L_c}^{L_c} (1+|x|)^{-n(1-\beta)} dx =  \frac{2}{n\beta -n+1} \big[ (L_c+1)^{n\beta -n+1} -1\big].
$$
Hence, as $\Delta x \to 0$ and $L_c\to \infty$, when $\frac{n-1}{n}<\beta\leq 1$, we obtain $$
\|\phi^{(2)}_\beta \|_{L^{0,0}_{2,\R}}^2 \sim \frac{2}{\Delta x (n\beta -n+1)} \big[ (L_c+1)^{n\beta -n+1} -1\big];
$$
by a similar computation when $\Delta x\to 0$ and $L_c\to \infty$, we get 
$$
\|\phi^{(2)}_\beta \|_{L^{0,1}_{2,\R}}^2 \sim \frac{24 \, \big[ (L_c+1)^{n\beta -n+1} -1\big] }{(\Delta x)^3\, (n\beta -n+1)},
$$
and 
$$
m_{\phi^{(2)}_\beta} \sim  \frac{24\, \big[ (L_c+1)^{n\beta -n+1} -1\big]}{(\Delta x)^4\,( n\beta -n+1)}.
$$
We note that when $\beta \in \big[ 1-\frac{1}{n},1]$, we no longer have the discretization of a ${\mathcal Q}$-Brownian motion taking values in $L^2_{\R}(\R)$.

\subsection{Discrete mass and energy; upper bounds on energy}

Consider the discrete mass
\begin{align}\label{D: Dmass}
M_{\mathrm{dis}}[u^m] \defeq \frac{1}{2} \sum_{j=0}^{N} |u^m_j|^2 \left( \Delta x_j + \Delta x_{j-1} \right),
\end{align}
which is conserved in our stochastic setting. Indeed, the proof of \cite[Lemma 2.1]{MRY2020} shows that the above three schemes \eqref{mass-energy}, \eqref{NS: crank-nicholson} and \eqref{NS:relaxation} conserve the discrete mass \eqref{D: Dmass}  at each time step: 
$M_{\mathrm{dis}}[u^m]=M_{\mathrm{dis}}[u^{m+1}]$, $m=0, ..., M-1$. This proof relies only on the fact that the noise is real-valued, multiplicative, and that we use the Stratonovich integral, which gives rise to  $\frac{1}{2} (u^m_j + u^{m+1}_j ) $ in the scheme. 

We next define the discrete energy adapted to the non-uniform mesh case
\begin{equation}\label{dis-energy}
H_{\rm dis}[u^m] \defeq  \frac{1}{2} \sum_{j=0}^N \Big| \frac{u^m_{j+1}-u^m_j}{\Delta x_j}\Big|^2 \Delta x_j - \frac{1}{4(\sigma+1)} \sum_{j=0}^N
|u^m_j|^{2(\sigma +1)} (\Delta x_{j-1}+ \Delta x_j).  
\end{equation} 
In the deterministic case ($\epsilon=0$), the MEC scheme \eqref{mass-energy} conserves the discrete energy, i.e., $H_{\rm dis}[u^{m+1}] \equiv H_{\rm dis}[u^{m}]$, which is proved  by multiplying $\frac{1}{2}(\bar{u}^{m+1}-\bar{u}^{m})(\Delta x_j+\Delta x_{j-1})$, summing from $j=0$ to $j=N$ and taking the real part.

In the stochastic setting, energy is not conserved, and the following proposition provides upper estimates on the time evolution of the average of an instantaneous and a maximal discrete energy. For simplicity we consider the scheme \eqref{mass-energy} with a constant space and time mesh. In that case the discrete energy \eqref{dis-energy} simplifies to
$$
H_{\rm dis}[u^m] =  \frac{1}{2} \sum_{j=0}^N \Big| \frac{u^m_{j+1}-u^m_j}{\Delta x}\Big|^2 \Delta x - \frac{1}{2(\sigma+1)} \sum_{j=0}^N
|u^m_j|^{2(\sigma +1)} \, \Delta x.
$$
Let $\tau^*_{\rm dis}$ denote the existence time of the discrete MEC scheme. 
\begin{proposition}\label{max_H_multi}
Let $u_0\in H^1$, $\phi e_j= \sqrt{\Phi_\beta(x_j)}\,  e_j$  be the covariance described in terms of  a function $\Phi_\beta$,
and $t_M<\tau^*_{\rm dis}$ be a point of the time grid for $N$ even and  constant space and time meshes. 
 Set $C=\frac{\sqrt{3}}{2} \big( 1+\frac{\sqrt{2e}}{\sqrt{\pi}}\big)$.  Then 
\begin{align}	
{\mathbb E}\big(  H_{\rm dis} [u^M] \big) &\leq H_{\rm dis}[u^0] + \frac{\epsilon}{2} \,   M_{\rm dis}[u^0]  \frac{C}{\sqrt{\Delta x} \big( \Delta t\big)^{\frac{3}{2}}} \, 
\Big[ \sqrt{\Phi_\beta(0) }+2 \sum_{j=1}^{\frac{N}{2}}  \sqrt{\Phi_\beta(j \Delta x)}\Big]\; t_M,  
\label{E_H_mult}
\\
{\mathbb E}\Big( \max_{0\leq m\leq M} H_{\rm dis}[u^m] \Big) &\leq \, H_{\rm dis}[u^0] + \epsilon \,  M_{\rm dis}[u^0]  \frac{C}{\sqrt{\Delta x} \big( \Delta t\big)^{\frac{3}{2}}} \, 
\Big[ \sqrt{\Phi_\beta(0) }+2 \sum_{j=1}^{\frac{N}{2}}  \sqrt{\Phi_\beta(j \Delta x)}\Big]\; t_M. 
\label{E_max_H_mult}
\end{align}
\end{proposition}

\begin{proof}
The approach is similar to that of  \cite[Prop. 3.2]{MRY2020}, though we include it for the sake of completeness. Multiplying the equation \eqref{mass-energy} by $- \Delta x\, (\bar{u}_j^{m+1}-\bar{u}^m_j)$, summing over $m=0,\dots, M-1$ and $j=0,\dots, N$, and using the conservation of the discrete energy in the deterministic case, we deduce that for some real-valued random variable $R(M,N)$, which changes from one line to the next,
\begin{align} 	\label{upp_H_multi}
&H_{\rm dis}[u^M]=\, H_{\rm dis}[u^0] +i R(M,N) + \epsilon \, \Delta x \,   \sum_{m=0}^{M-1} \sum_{j=0}^N  (\bar{u}^{m+1}_j-\bar{u}^m_j) 
 \, \frac{1}{2}\big( u^{m+1}_j+u^m_j\big)
 {f}^{m+\frac{1}{2}}_j  \nonumber \\
 =& \quad H_{\rm dis }[u^0] + i R(M,N)   + \frac{ \epsilon\, \Delta x}{2}   \sum_{m=0}^{M-1} \sum_{j=0}^N   \, \big(  |u^{m+1}_j|^2 - |u^{m}_j|^2\big) \, 
   {f}^{m+\frac{1}{2}}_j, \\
 = &\quad H_{\rm dis }[u^0] + i R(M,N)   - \frac{\epsilon}{2 \Delta t }  \!\!   \, \int_0^{t_M}\!\!   \int_{\R} 
  |{U}(s,x)|^2 W_N(ds,dx) 
+  \frac{ \epsilon\, \Delta x}{2}   \sum_{m=0}^{M-1} \sum_{j=0}^N   \,  |u^{m+1}_j|^2 \,   {f}^{m+\frac{1}{2}}_j,  \label{upp_H_multi_1}
\end{align} 
where $U(s,x)$ is the step process defined by $U(s,x)=u^m_j$ on the rectangle $[t_m,t_{m+1})\times [x_{j-\frac{1}{2}} , x_{j+\frac{1}{2}})$. 
Since the discrete mass is preserved by the scheme, we have 
$$ 
\frac{ \epsilon\, \Delta x}{2}   \sum_{m=0}^{M-1} \sum_{j=0}^N   \,  |u^{m+1}_j|^2 \,   {f}^{m+\frac{1}{2}}_j  
\leq \frac{\epsilon\, M_{\rm dis}[u^0]}{2} \sum_{m=0}^{M-1} \max_{0\leq j\leq N} | {f}^{m+\frac{1}{2}}_j|. 
$$ 
Using the definition of ${f}^{m+\frac{1}{2}}_j$, we deduce  
$$ 
E\Big( \max_{0\leq j\leq N} | {f}^{m+\frac{1}{2}}_j| \Big) =\frac{\sqrt{3}}{2}\;  \frac{1}{\sqrt{\Delta t} \sqrt{\Delta x}}\;  E\Big( \max_{0\leq j\leq N}  \alpha_j 
| \chi^{m+\frac{1}{2}}_j| \Big),
$$
where the random variables $\chi^{m+\frac{1}{2}}_j$ are independent standard Gaussians and $\alpha_j=\sqrt{\Phi_\beta(x_j)}$. 
 
Next, we note the fact that if $\{G_k, \, k=1,\dots,n\}$ are independent standard Gaussians and $B_n=\max_{1\leq k\leq n} \{ \gamma_k\, |G_k| \}$ 
for positive constants $\{\gamma_k\}$, then for $n\geq 2$ 
\begin{equation}\label{E_max_G}
{\mathbb E}(B_n) \leq \Big( 1+\frac{\sqrt{2e}}{\sqrt{\pi}}\Big) \sum_{k=1}^n \gamma_k. 
\end{equation}
Observe that this upper estimate is relevant in the case when the infinite series $\sum_{k\geq 0} \gamma_k$ is convergent. 
When $\{\gamma_k\}_k$  is a constant sequence, 
the upper estimate (3.19), used in the proof of \cite[Prop 3.2]{MRY2020}, gives a sharper upper bound. We next prove \eqref{E_max_G}, noting that the proof differs from the one done in \cite[Prop 3.2]{MRY2020}. For every $t>0$ we have 
$$ 
P(B_n \leq  t) =  \prod_{k=1}^n P(\gamma_k |G_k|\leq t) 
=  \prod_{k=1}^n \Big[ 1- P\Big( |G_1| > \frac{t}{\gamma_k} \Big)\Big] 
 \geq 1-\sum_{k=1}^n P\Big(|G_1| > \frac{t}{\gamma_k}\Big),
$$ 
and we deduce  that  
$$
{\mathbb E}(B_n)  = \int_0^\infty\! P(B_n > t)\,  dt  
 \leq \sum_{k=1}^n  \Big[ c_k +\int_{c_k}^\infty P\Big(|G_1|> \frac{t}{\gamma_k} \Big) dt \Big] 
$$
for any choice of positive constants  $\{c_k\}_k$. Using the tail estimate  
$$
P(|G_1| > t) = \frac{2}{\sqrt{2\pi}} \int_t^\infty e^{-\frac{x^2}{2}} dx \leq \frac{2}{\sqrt{2\pi}} \frac{1}{t} \int_t^\infty x e^{-\frac{x^2}{2}} dx 
= \sqrt{\frac{2}{\pi}}  \frac{1}{t}  e^{-\frac{t^2}{2}},
$$ 
we obtain
\begin{align*}
{\mathbb E}(B_n)  \leq &\sum_{k=1}^n \Big[ c_k  +  \sqrt{\frac{2}{\pi}}\, \int_{c_k} ^\infty \frac{\gamma_k}{t} e^{-\frac{1}{2}\, \big(\frac{t}{ \gamma_k}\big)^2}\;  dt\Big]  
 \leq \sum_{k=1}^n \Big[  c_k+  \sqrt{\frac{2}{\pi}}\, \gamma_k  \int_{\frac{c_k}{\gamma_k}}^\infty \frac{e^{-\frac{t^2}{2}}}{t}  dt\Big] \\
 \leq&\,  \sum_{k=1}^n \Big[ c_k + \gamma_k \Big( \frac{\gamma_k}{c_k} \Big)^2  \sqrt{\frac{2}{\pi}} e^{-\frac{1}{2}\big( \frac{c_k}{\gamma_k}\big)^2} \Big].  
\end{align*} 
Choosing $c_k \defeq \gamma_k$, we deduce \eqref{E_max_G}. 
 
Keeping the real part of  \eqref{upp_H_multi_1}, we get that  for $C= \frac{\sqrt{3}}{2}\;  \big(1+\frac{\sqrt{2e}}{\sqrt{\pi}}\big)$, 
\begin{align*}
{\mathbb E} \big( H_{\rm dis}[u^M]\big) \leq &\; H_{\rm dis}[u^0] +  \frac{\epsilon}{2} M_{\rm dis}[u^0]   \, M\, \frac{C}{\sqrt{\Delta t}\, \sqrt{\Delta x}}\, 
\Big[ \sqrt{\Phi_\beta(0) }+2 \sum_{j=1}^{\frac{N}{2}}  \sqrt{\Phi_\beta(j \Delta x)}\Big] \\
\leq & \; 
H_{\rm dis}[u^0] + \epsilon  M_{\rm dis}[u^0]  \frac{C}{\sqrt{\Delta x} \big( \Delta t\big)^{\frac{3}{2}}} \, 
\Big[ \sqrt{\Phi_\beta(0) }+2 \sum_{j=1}^{\frac{N}{2}}  \sqrt{\Phi_\beta(j \Delta x)}\Big]\; t_M.    \
\end{align*}
This completes the proof of \eqref{E_H_mult}.

To prove \eqref{E_max_H_mult},  keeping the real part of  \eqref{upp_H_multi} and upper estimating $| u^{m+1}_j|^2 - |u^m_j|^2$ by
$| u^{m+1}_j|^2 + |u^m_j|^2$, we obtain 
$$ 
\max_{0\leq m\leq M} H_{\rm dis}[u^M]= H_{\rm dis}[u^0]  +  \frac{ \epsilon\, \Delta x}{2}   \sum_{m=0}^{M-1} \sum_{j=0}^N   \,  \big( |u^{m+1}_j|^2 + |u^m_j|^2\big)
\,   |{f}^{m+\frac{1}{2}}_j|.  
$$  
The same argument as for \eqref{E_H_mult} concludes the proof. 
\end{proof}

We next give explicit bounds \eqref{E_H_mult}-\eqref{E_max_H_mult} for the two examples described above.
 
\underline{\bf Example 1:} From  $\Phi^{(1)}_\beta(x)=e^{ - (1-\beta) x^2}$, we have $\Phi^{(1)}_\beta(0)=1$, and for  $\beta \in [0,1)$ 
$$
\sum_{j=1}^{\frac{N}{2}} \sqrt{\Phi^{(1)}_\beta (j \Delta x)} \leq \frac{1}{\Delta x} \int_0^{L_c} e^{-\frac{(1-\beta)}{2} x^2}\, dx  
\leq \frac{1}{\Delta x}\,  \frac{\sqrt{2\pi}}{2} \frac{1}{\sqrt{1-\beta}}. 
$$ 
Thus, 
$$
{\mathbb E} \big(  H_{\rm dis} [u^M] \big) 
\leq H_{\rm dis}[u^0] + \frac{\epsilon}{2} \, M_{\rm dis}[u^0] \, \frac{C}{\sqrt{\Delta x} \big( \Delta t\big)^{\frac{3}{2}}} \, 
\Big[ 1 +   \frac{\sqrt{2\pi}}{\Delta x \, \sqrt{1-\beta}} \Big] \; t_M; 
$$
 a similar bound holds for \eqref{E_max_H_mult}. 
 
\underline{\bf Example 2:} From $\Phi^{(2)}_\beta(x)= {(1+|x|)^{-n(1-\beta)}}$, we have $\Phi^{(2)}_\beta(0)=1$ and
$$
\sum_{j=1}^{\frac{N}{2}} \sqrt{\Phi^{(1)}_\beta (j \Delta x)} \leq \frac{1}{\Delta x} \int_0^{L_c} (1+x)^{-\frac{n(1-\beta)}{2}}\, dx.
$$ 
If $\beta \in \big[0, 1-\frac{2}{n}\big)$, we deduce that 
$$
\sum_{j=1}^{\frac{N}{2}} \sqrt{\Phi^{(2)}_\beta (j \Delta x)} \leq   \frac{2}{\big[ n(1-\beta) -2\big] \Delta x}.
$$
If $\beta \in \big[  1-\frac{2}{n}, 1\big)$, then the above sum also depends on $L_c$, more precisely, 
\begin{align*} 
\sum_{j=1}^{\frac{N}{2}} \sqrt{\Phi^{(2)}_\beta (j \Delta x)} &\leq \frac{ \ln(L_c)}{\sqrt{\Delta x}} \quad \mbox{\rm for}\quad   \beta = 1-\frac{2}{n}, \\
\sum_{j=1}^{\frac{N}{2}} \sqrt{\Phi^{(2)}_\beta (j \Delta x)} &\leq \frac{2}{2-n(1-\beta)} \Big[ L_c^{1-\frac{n(1-\beta)}{2}} -1\Big]\, \frac{1}{\Delta x}
 \quad \mbox{\rm for}\quad  \beta \in \big( 1-\frac{2}{n}, 1\big).
\end{align*}
Substituting the above into \eqref{E_H_mult} or \eqref{E_max_H_mult}, we obtain the bounds in Example 2. 

From the above analysis, we find that the upper bounds for the discrete energy can depend on parameters $\Delta x$, $L_c$, $\Delta t$ and $\epsilon$.
We will next investigate this dependence numerically. 

\subsection{Numerical tracking of discrete energy}
We first show the accuracy of all three schemes in discrete mass and energy computations for the $\mathcal{Q}$-Brownian  driving noise. 
We take initial data of type $u_0=A\, Q$, where $Q$ is the ground state from \eqref{E:Q}, and obtain the error in computing the discrete mass 
(since it is supposed to be conserved) and then track the growth of the discrete energy (both instantaneous and maximum up to some given time t). 

\begin{figure}[ht]
\includegraphics[width=0.32\textwidth]{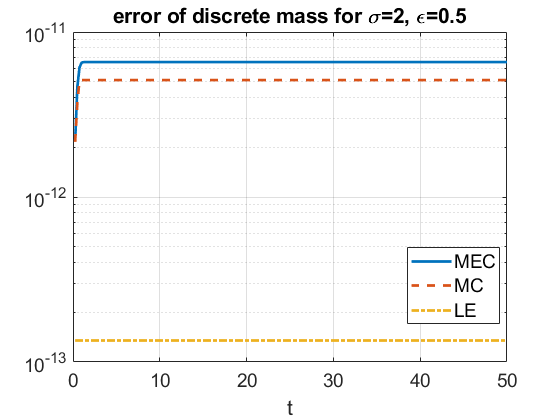} 
\includegraphics[width=0.32\textwidth]{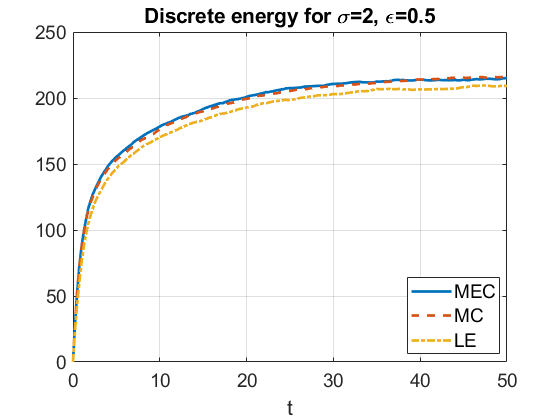} 
\includegraphics[width=0.32\textwidth]{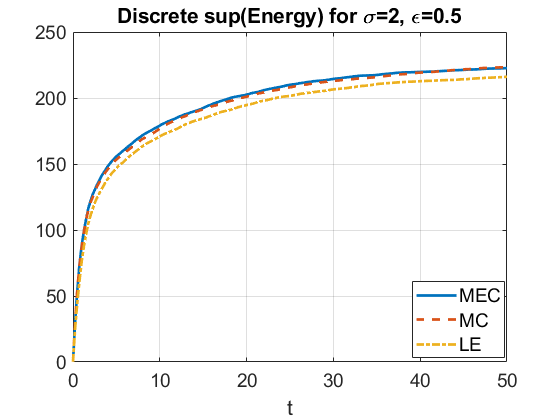}  
\caption{Accuracy of three schemes with noise in Example 2, $n=2$. The $L^2$-critical case ($\sigma =2$) with  $u_0=0.9Q$, $\beta=0.5$, $\epsilon=0.5$, 
$L_c=20$, $\Delta x=0.1$ and $\Delta t=0.01$.
The left plot is the error $\mathcal{E}^m[M]$ defined in \eqref{E:M-error} in computation of the discrete mass for all three schemes. The growth of average
instantaneous energy (middle) and maximum energy (right) for different numerical schemes.}
\label{F:scheme Ex2}
\end{figure} 

In Figure \ref{F:scheme Ex2} we show the accuracy of our computations in the $L^2$-critical case ($\sigma=2$) for the initial data $u_0 = 0.9 \, Q$. The left graph shows the accuracy of all three schemes in computing the discrete mass. The error 
is defined as 
\begin{equation}\label{E:M-error}
\mathcal{E}^m[M] \defeq \max_m \left\lbrace M_{\mathrm{dis}}[u^m] \right\rbrace-\min_m \left\lbrace M_{\mathrm{dis}}[u^m] \right\rbrace,
\end{equation}
and is on the order of $10^{-13}, \dots, 10^{-11}$, with the linear extrapolation (LE) scheme outperforming slightly the other two schemes (it does not accumulate any error from solving a nonlinear system in the fixed point iteration as the other two schemes). The middle and right subplots show the growth and leveling off of the expected value of energy in Example 2 (we omit Example 1 as it is similar and has faster decay), the instantaneous energy (in the middle) and the average of sup energy (on the right). The average here was computed out of 100 runs. 
For the purposes of (a large number of) multiple runs, it is significantly faster to use the LE scheme. 

We next investigate the time evolution of energy. We consider both $L^2$-critical and supercritical cases, and study solutions on the time interval $0\leq t \leq 100$. 
For that we take $u_0=AQ$ with $A=0.9$ in the $L^2$-critical ($\sigma =2$) case, and $A=0.8$ in the $L^2$-supercritical ($\sigma =3$) case. 
The reason for a smaller coefficient in the supercritical case is to ensure that solutions exist on this time interval (see more about that at the end of 
Section \ref{Numerical results}).   

\begin{figure}[ht]
\includegraphics[width=0.32\textwidth]{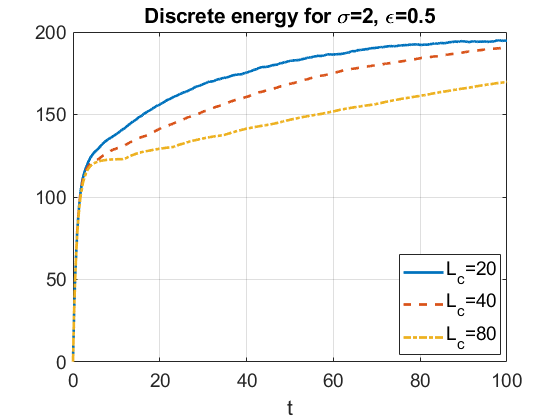} 
\includegraphics[width=0.32\textwidth]{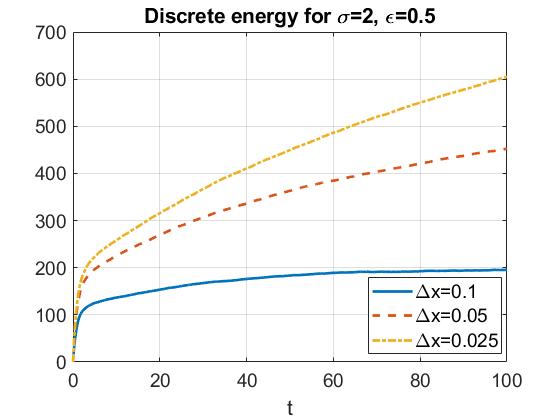} 
\includegraphics[width=0.32\textwidth]{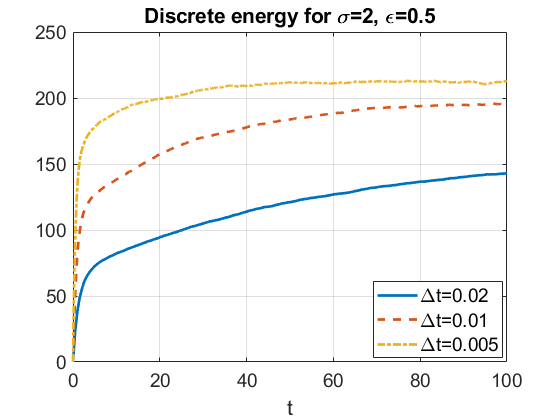}  
\caption{Time evolution of energy and its dependence on parameters $L_c$ (left), $\Delta x$ (middle) and $\Delta t$ (right), the noise is Gaussian-type decay kernel (Example 1) with 
$\beta=0.5$ and $\epsilon=0.5$.} 
\label{F:dx dt Lc Ex1}
\end{figure}

Figure \ref{F:dx dt Lc Ex1} tracks the time evolution of the discrete energy in
Example 1 (Gaussian-type decay of eigenvalues) and its dependence on $L_c$, 
$\Delta x$ and $\Delta t$. We note that there is leveling off in the dependence on $L_c$ and $\Delta t$, and there is an inverse dependence on $\Delta x$. 
In Figure \ref{Fig:beta eps Ex1} we track the dependence of energy on  correlation $\beta$ and noise strength $\epsilon$ in this example. 

\begin{figure}[ht]
\includegraphics[width=0.35\textwidth]{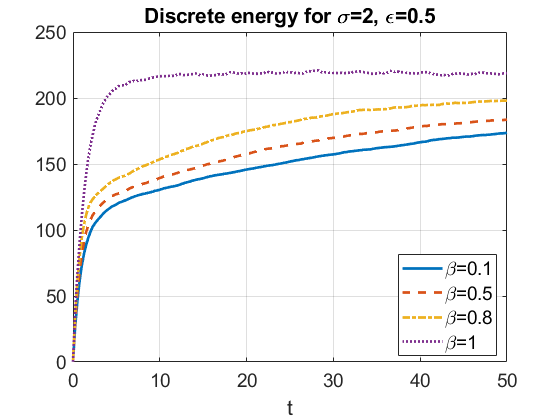} 
\includegraphics[width=0.35\textwidth]{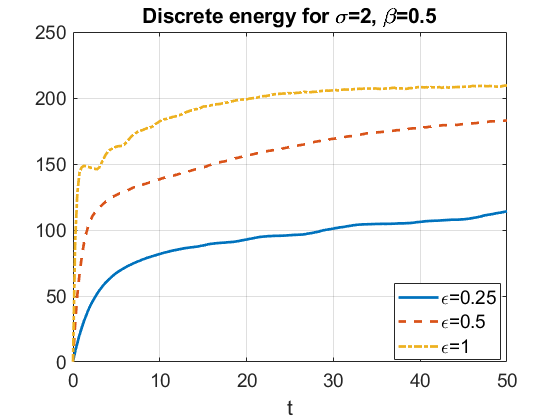}  
\includegraphics[width=0.35\textwidth]{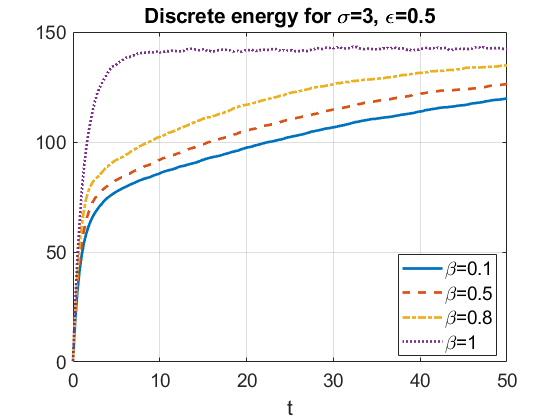} 
\includegraphics[width=0.35\textwidth]{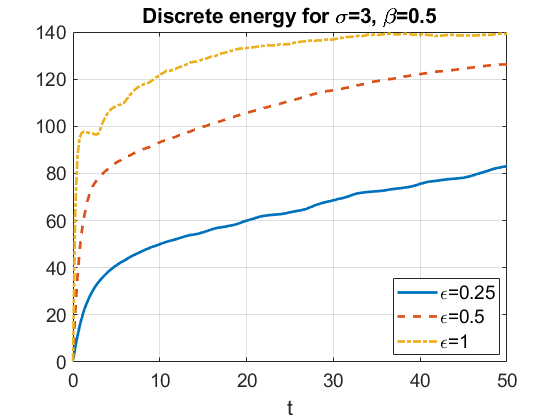}
\caption{The growth of energy for different values of $\beta$ (left) 
and $\epsilon$ (right) in Example 1 (Gaussian decay) with $L_c=20$, $\Delta x=0.1$ and $\Delta t=0.01$. 
Comparison is given in both $L^2$-critical and supercritical cases
for the same $\epsilon=0.5$ on the left two plots, and for the same $\beta=0.5$ on the right two plots.} 
\label{Fig:beta eps Ex1}
\end{figure}  

In Figures \ref{F:dx dt Lc Ex2}-\ref{F:dx dt Lc Ex2 n4} we study the time evolution of energy when the covariance of the driving  noise has a polynomial decay  
(Example 2).  In Figure \ref{F:dx dt Lc Ex2} we show how energy depends on $L_c$, $\Delta x$ and $\Delta t$ (note the dependence on $\Delta x$).

\begin{figure}[ht]
\includegraphics[width=0.32\textwidth]{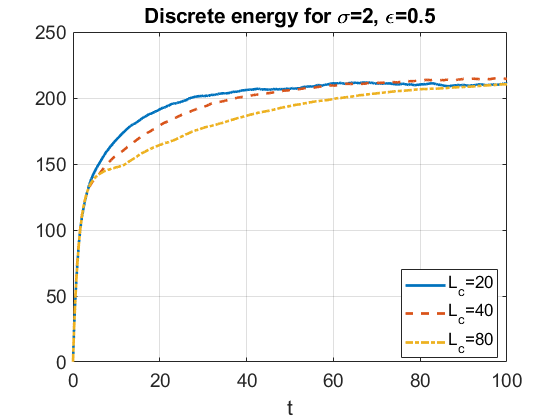}
\includegraphics[width=0.32\textwidth]{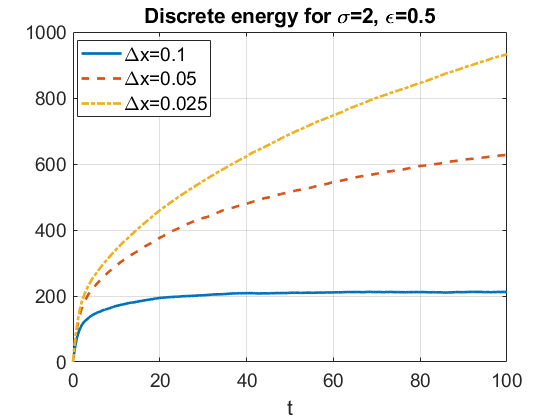} 
\includegraphics[width=0.32\textwidth]{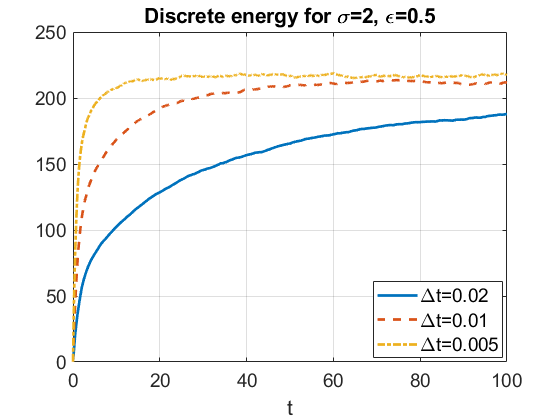}  
\caption{Time evolution of energy and its dependence on parameters $L_c$ (left), $\Delta x$ (middle) and $\Delta t$ (right) in Example 2 (polynomial decay) with $n=2$, 
$\beta=0.5$ and $\epsilon=0.5$.}
\label{F:dx dt Lc Ex2}
\end{figure}

In Figure \ref{F:beta eps Ex2} the dependence on the correlation parameter $\beta$ and the strength of the noise $\epsilon$ is shown, for $n=2$ 
(energy levels off, in some cases eventually; reaching the horizontal asymptote  faster for larger $\beta$, when the kernel is less spread out, or for larger 
$\epsilon$, when the strength of the noise is higher).  

\begin{figure}[ht]
\includegraphics[width=0.36\textwidth]{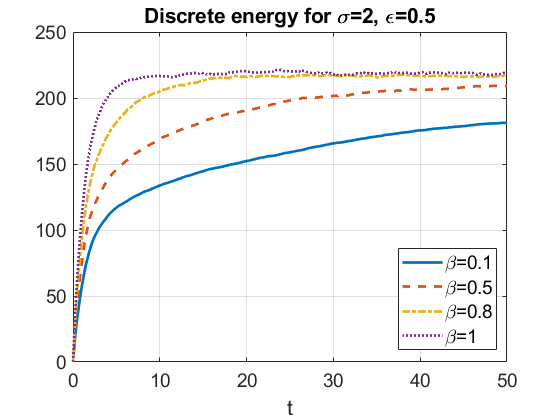}
\includegraphics[width=0.36\textwidth]{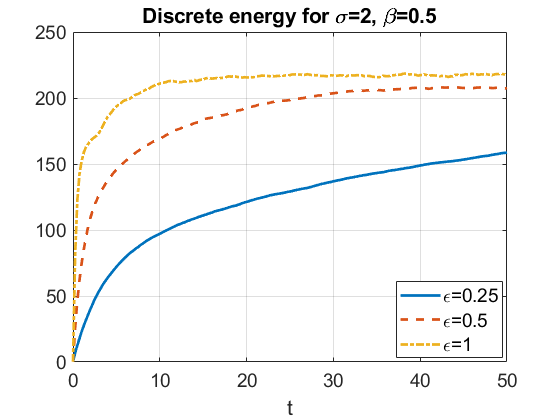}  
\includegraphics[width=0.36\textwidth]{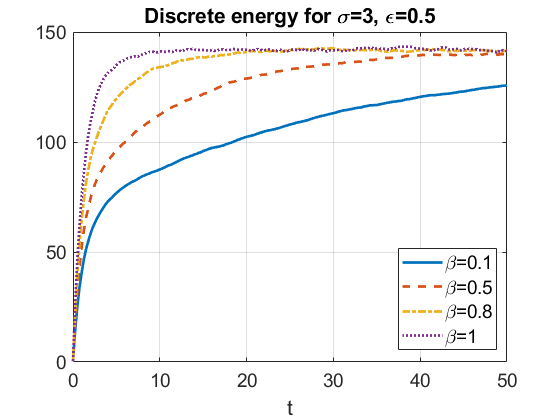}  
\includegraphics[width=0.36\textwidth]{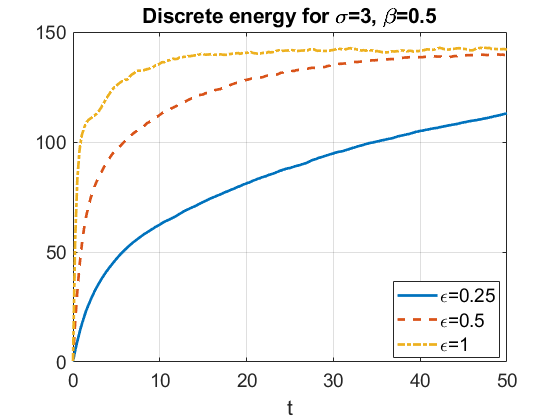}
\caption{The growth of energy for different values of $\beta$ (left) 
and $\epsilon$ (right) in Example 2 (polynomial decay) with $n=2$, $L_c=20$, $\Delta x=0.1$ and $\Delta t=0.01$. 
Comparison is given in both $L^2$-critical and supercritical cases
for the same $\epsilon=0.5$ on the left two plots, and for the same $\beta=0.5$ on the right two plots.} 
\label{F:beta eps Ex2}
\end{figure} 

In Figure \ref{F:dx dt Lc Ex2 n4} we take $n=4$ and vary the correlation parameter $\beta$, noting that for larger $\beta$ the energy gets slightly larger and stabilizes faster, and that there is almost no dependence on $L_c$.

\begin{figure}[ht]
\includegraphics[width=0.32\textwidth]{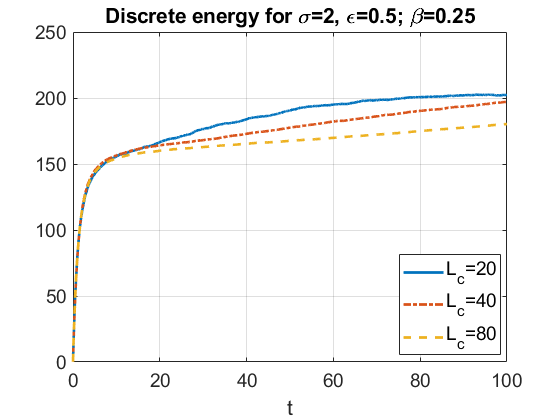}
\includegraphics[width=0.32\textwidth]{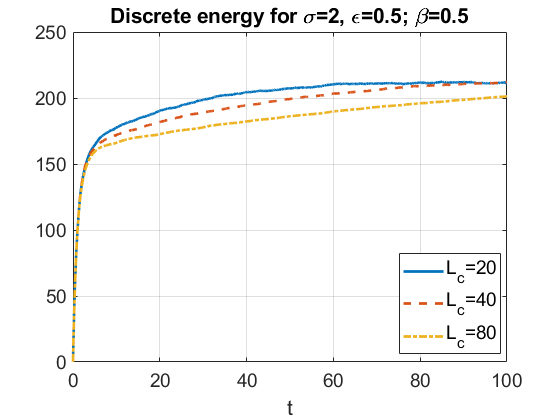} 
\includegraphics[width=0.32\textwidth]{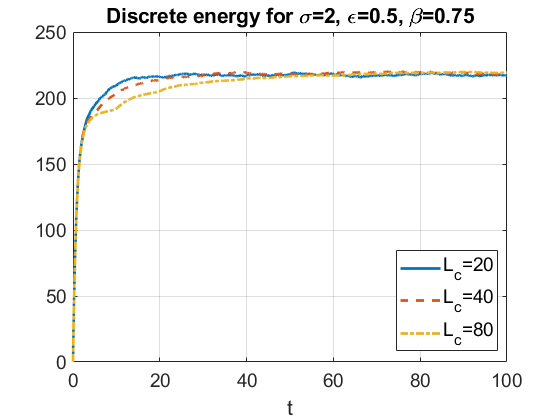}  
\caption{The growth of energy in Example 2 (polynomial decay) with  $n=4$, 
for different values of $L_c$ with $\beta=0.25$ (left), $\beta=0.5$ (middle) and $\beta=0.75$ (right).} 
\label{F:dx dt Lc Ex2 n4}
\end{figure}

We summarize that in both Example 1 and 2, the energy grows sharply in the beginning, then slows down in its growth and levels off: the larger the strength of the noise 
$\epsilon$ is, or the closer it is to the space-time white noise (in other words, the more irregular the noise becomes), or the smaller the time step is, then the faster the discrete energy levels off. It seems to be very sensitive to the space mesh size $\Delta x$,  but not sensitive to the length of the computational interval $L_c$.

\section{Stochastic perturbation driven by a homogeneous Wiener process}\label{homogen}
In this section we discuss another classical way to smooth the space-time white noise in the space variable. As in the previous section, 
the noise will be white in time and correlated in space. However, it will not be an $L^2_{\R}(\R)$-valued process and, as in \cite{DM2002a}  and \cite{MRY2020}, we will
have to consider a partial sum of an infinite diverging series. 

\subsection{Description of the driving noise}
Let ${\mathcal D}(\R^2)$ be the set of $C^\infty$-functions with compact support. Let $\tilde{W}= \{ W(\phi); \phi \in {\mathcal D}(\R^2)\}$ be an $L^2(\Omega)$-valued
centered  Gaussian process with covariance defined by
$$ 
{\mathbb E}\big(\tilde{W}(\phi)\, \tilde{W}(\psi)\big)  = J(\phi,\psi) \defeq \int_0^\infty \!\!  ds \int_{\R} dx \int_{\R}  \phi(s,x) \tilde{\Phi}(|x-y|) \psi(s,y) \, dy  ~~ \mbox{ for} ~~
\phi, \psi \in {\mathcal D}(\R^2).
$$
We assume that the function $\tilde{\Phi}$ (which may be defined almost everywhere)  is the density of a measure, which is the Fourier transform of a 
tempered symmetric measure $\mu$ (referred to as its spectral measure). 
Indeed, this requirement is a necessary and sufficient condition for 
$J(\cdot,\cdot)$ to be non-negative definite and define a covariance structure 
(for more details see \cite[p.5-6]{Da1999}). 
To stress the difference with examples in the previous Section \ref{Q-Wiener} we denote the covariance kernel by $\tilde{\Phi}$.

We are interested in two cases, which we call Examples 3 and 4. 
\subsubsection{Example 3: {\bf Riesz kernel}} 
Let $\beta\in (0,1)$, and recall the Riesz kernel, defined by
$$
R_\beta(x)=  |x|^{-\beta}\;  \mbox{\rm  for }\; x\neq 0, \quad \mbox{\rm and} \; R_\beta(0)= +\infty.
$$
In order to make sure that when $\beta \rightarrow 1$, the corresponding homogeneous noise   approaches  the space time white noise,  
as in  Examples 1 and 2  considered in Section \ref{Q-Wiener},   we 
modify the Riesz kernel, multiplying it by the constant $\frac{(1-\beta)(2-\beta)}{2}$, to read
$$
\tilde{\Phi}^{(3)}_\beta(x)=\frac{(1-\beta)(2-\beta)}{2} \,  |x|^{-\beta}\;  \mbox{\rm  for }\; x\neq 0, \quad \mbox{\rm and} \; \tilde{\Phi}^{(3)}_\beta (0)=+\infty.
$$
This  comes down to changing the coefficient $\epsilon$ by another one depending on $\beta$.  The Fourier transform of $R_\beta$ is the function 
$\nu_\beta (x) = C(\beta) R_{1-\beta}(x) $ for some positive constant $C(\beta)$. 
Note that it is symmetric and is the density of a measure, which is  a tempered distribution. 
The  Riesz kernel $|x|^{-\beta}$ has a singularity at the origin. 

 As $\beta \to 1$, we get the limiting case
$\beta=1$ of the modified kernel $\tilde{\Phi}^{(3)}_1$, which corresponds to the space-time white noise (see \eqref{eq_cov_Riesz} and discussion afterwards). 

\subsubsection{Example 4: {\bf Exponential kernel}} 
For $\beta >0$ we define $\tilde{\Phi}^{(4)}_\beta$ by
$$
\tilde {\Phi}^{(4)}_\beta(x) = e^{ - {\beta |x|}}.
$$
For large values of $|x|$, the decay of $\tilde {\Phi}^{(4)}_\beta(x)$ is exponential, hence, 
faster than that of the Riesz kernel $|x|^{-\beta}$ from Example 3, which is polynomial. The Fourier transform of $\tilde{\Phi}^{(4)}_\beta$ 
is the function $G(x)=\frac{2\,   \beta}{\beta^2 + 4 \pi^2  x^2}$; note that  the symmetric measure $G(x) dx$ is a tempered distribution. 

\subsubsection{Covariance matrices}
We do not deal with a diagonal matrix anymore as in Examples 1 and 2;  instead  we consider the covariance matrix
\begin{equation}\label{def_Gamma}
\Gamma(j,k) \defeq \int_{\R} dx \int_{\R}  \tilde{e}_j(x)\,  \tilde{\Phi}(|x-y|)  \,  \tilde{e}_k(y) \, dy, \quad j,k=0,\dots,N,
\end{equation}
for some choice of orthonormal vectors $\{\tilde{e}_j\}_{0\leq j\leq N}$. 
The assumptions made on the existence of the spectral measure
 of $\tilde{\Phi}(x) dx$ ensure that the symmetric $(N+1)\times (N+1)$-matrix  $\Gamma$ is positive definite. 
Let  $\phi$ be the operator defined by $\phi \, \tilde{e}_j(x) =\int_{\R} \tilde{e}_j(y) f(|x-y|) dy$, where 
$\tilde{\Phi}(x-y) = \int_{\R} f(|x-z|)  f(|z-y|) dz = \int_{\R} f(|x-y-z|) f(|z|) dz$.  
To make numerical computations easier, for this type of noise in Examples 3 and 4 we  use indicator functions $\tilde{e}_j$. Indeed, thanks to the regularization effect
 of the convolution used in the definition of $\phi$, the regularity of the function $\tilde{\Phi}$ makes it possible to have an $H^1_{\R}(\R)$-valued function 
 $\phi\, \tilde{e}_j$ when $\tilde{e}_j$ is an indicator function. 
 
Recalling that $x_{j+\frac{1}{2}}=\frac{1}{2} (x_j+x_{j+1})$, we define the functions $\{ \tilde{e}_j\}_{0\leq j\leq N}$ as 
\begin{align*}
\tilde{e}_j = & \frac{\sqrt{2}}{\sqrt{\Delta x_{j-1} + \Delta x_j}} 1_{[x_{j-\frac{1}{2}}, x_{j+\frac{1}{2}}]}, \quad j=1,\dots, N-1,\\
 \tilde{e}_0= & \frac{\sqrt{2}}{\sqrt{\Delta x_0}} 1_{[x_0, x_{\frac{1}{2}}]}, \quad \tilde{e}_N=\frac{\sqrt{2}}{\sqrt{ \Delta x_{N-1}}} 1_{[x_{N-\frac{1}{2}}, x_N]}.
\end{align*} 
Note that $\{\tilde{e}_j\}_{0\leq j\leq N}$ are orthonormal functions in $L^2_{\R}(\R)$. 
We will now write the covariance matrices explicitly for each of the above two examples. In order to produce the covariance matrix that will be
used to define the driving perturbation in our simulations, we renormalize  $\Gamma$ to $\tilde{\Gamma}$ defined by
\begin{equation}\label{tilde_Gamma}  
\tilde{\Gamma}(j,k)= \frac{2}{\sqrt{\Delta x_{j-1}+ \Delta x_j}\, \sqrt{\Delta x_{k-1}+\Delta x_k}} \, \Gamma(j,k), \quad j,k=0,\dots,N. 
\end{equation}

\underline{\bf Example 3:} For the  Riesz kernel the renormalized covariance matrix $\tilde{\Gamma}^{(3)}_\beta$ is defined for $j,k=0, ..., N$ by
\begin{align}		\label{eq_cov_Riesz} 
\tilde{\Gamma}^{(3)}_\beta(j,k)= 2 \left( \frac{  \big| x_{k+\frac{1}{2}} - x_{j-\frac{1}{2}} \big|^{2-\beta}  +
\big| x_{k-\frac{1}{2}} - x_{j+\frac{1}{2}} \big|^{2-\beta}\!  -  \big|x_{k+\frac{1}{2}}-x_{j+\frac{1}{2}}\big|^{2-\beta} \! 
- \big| x_{k-\frac{1}{2}} \! - x_{j-\frac{1}{2}}\big|^{2-\beta}}
{ (\Delta x_{j-1}+\Delta x_j) (\Delta x_{k-1}+\Delta x_k)} \right). 
\end{align}
Note that $\tilde{\Gamma}^{(3)}_\beta$ is positive definite. Furthermore, if 
$\beta=1$, it is easy to see that 
$\tilde{\Gamma}^{(3)}_1(k,k) = \frac{2}{\Delta x_{k-1}+\Delta x_k}$ and $\tilde{\Gamma}^{(3)}_1(j,k) = 0$  for $j,k=0, ..., N$, $j\neq k$. 
This is the renormalized version of the covariance matrix of the space-time white noise used in \cite{MRY2020}. 

\underline{\bf Example 4:} For the exponential kernel, 
the renormalized covariance matrix $\tilde{\Gamma}^{(4)}_\beta$ is defined by 
$$ 
\tilde{\Gamma}^{(4)}_\beta(k,k)= \frac{4}{(\Delta x_{k-1}+\Delta x_k)^2} \, 
\Big[ \frac{\Delta x_{k-1}+\Delta x_k}{\beta}  - \frac{2}{\beta^2} \big( 1- e^{-\frac{\beta}{2} (\Delta x_{k-1}+\Delta x_k}\big)  \Big]
$$
for $k=0,\dots, N$, and
\begin{align*}
\tilde{\Gamma}^{(4)}_\beta(j,k)= 4 \; \frac{ e^{-\beta\big| x_{k-\frac{1}{2}} - x_{j+\frac{1}{2}}\big|} + e^{-\beta\big| x_{k+\frac{1}{2}} - x_{j-\frac{1}{2}}\big|}
 - e^{-\beta\big| x_{k-\frac{1}{2}} - x_{j-\frac{1}{2}}\big|} - e^{-\beta\big| x_{k+\frac{1}{2}} - x_{j+\frac{1}{2}}\big|}}{
 \beta^2\,  \big( \Delta x_{j-1}+\Delta x_j\big) \,  \big(\Delta x_{k-1}+  \Delta x_k\big)}
\end{align*}
for $j\neq k$, $j,k=0,\dots, N$. 

\subsection{Covariance matrix computation, bounds on discrete energy}  As in Section \ref{Q-Wiener}, we use mass-conservative schemes. 
However, the functions $\{\tilde{e}_j\}_{0\leq j\leq N}$
do not give rise to a diagonal covariance matrix. This is due to the fact that the noise correlation involves a convolution, which has long-range effects and
gives rise to a full matrix.

In order to simulate a centered Gaussian $(N+1)$-dimensional vector with covariance matrix $\tilde{\Gamma}$, we use the Cholesky decomposition of 
$\tilde{\Gamma}$.  This is possible in Examples 3 and 4, since the covariance matrices are positive-definite.
 More precisely, we find a lower triangular matrix $A$ such that 
$\tilde{\Gamma} = A A^*$, where $A^*$ denotes the transposed matrix of $A$.
Therefore, if $Y=(Y_0,\dots, Y_N)$ denotes an $(N+1)$-dimensional Gaussian vector with independent components, which are standard Gaussian 
(i.e., ${\mathcal N}(0,1)$) random variables, the covariance matrix of $Y$ is ${\rm Id}$.   Let ${\mathcal A}$ be the linear operator on $\R^{N+1}$, whose matrix in the canonical basis is $A$.
Then    ${X \defeq \mathcal A}Y$ is a centered Gaussian random vector with covariance  matrix $\tilde{\Gamma}$. 
We then have to produce a vector $u^{m+1} \defeq (u^{m+1}_0, \dots, u^{m+1}_{N})$ 
in terms of the vector  $u^m=(u^{m}_0,\dots,u^{m}_{N})$ 
(and can no longer define   an isolated component 
$u^{m+1}_j$ in terms of $u^m_j$ for  some $j=0,\dots,N$). 

Let ${\chi}^{m+\frac{1}{2}} \defeq ({\chi}^{m+\frac{1}{2}}_0,\dots, {\chi}^{m+\frac{1}{2}}_N)$ denote a Gaussian vector whose components are independent ${\mathcal N}(0,1)$ random variables. Then set for $j=0,\dots,N$  
\begin{equation}\label{E:homog}
\tilde{f}^{\,m+\frac{1}{2}}   =  \frac{1}{\sqrt{\Delta t_m} }\; 
{\mathcal A}  {\chi}^{m+\frac{1}{2}}  
\quad \mbox{\rm and } \quad   
{f}^{m+\frac{1}{2}}_j = \frac{u^m_j + u^{m+1}_j}{2} \,\tilde{f}^{\, m+\frac{1}{2}}_j . 
\end{equation}
With this definition of the vector $\{f_j^{m+\frac{1}{2}}\}_j$, we define analogs of the three schemes in \eqref{mass-energy}--\eqref{NS:relaxation}.

\subsubsection{Upper bounds on discrete energy}
We remind that the discrete mass  $M_{\rm dis}[u]$ defined in \eqref{D: Dmass} is conserved, due to the fact that the noise is real-valued and the factor
$\frac{1}{2}(u^m_j+u^{m+1}_j)$ used to define $f^{m+\frac{1}{2}}_j$ corresponds to the discretization of the Stratonovich integral.
We next prove an upper bound of the average of the instantaneous and maximal discrete energy.
\begin{proposition}\label{prop-Hdis-hom}
Let $u^m$ be the solution of the MEC scheme \eqref{mass-energy} with a constant time and space mesh, for a random perturbation defined by
\eqref{tilde_Gamma}. Suppose that $\sup_{0\leq k\leq N}\tilde{\Gamma}(k,k)\leq \delta^2$ for some positive constant $\delta^2$. 
Let $\tau^*$ denote the random existence time of this scheme. 
Then for every time $t_M < \tau^*$ on the time grid, we have that for $\Delta x\in (0,1)$ 
\begin{align}
{\mathbb E} \big( H_{\rm dis}[u^M]\big)  \leq & \; H_{\rm dis}[u^0] + \frac{\epsilon \, M_{\rm dis}[u^0]}{\sqrt{2} 
  \, \big( \Delta t\big)^{\frac{3}{2}} }\,
\delta \, \sqrt { \ln(4\, L_c) + | \ln(\Delta x)| }\; t_M, 		
\label{H-hom-inst}
\\
{\mathbb E}\Big( \max_{0\leq m\leq M}  H_{\rm dis}[u^m] \Big) \leq &\;  H_{\rm dis}[u^0] + \frac{\sqrt{2} \, \epsilon\,  M_{\rm dis}[u^0]}{
\big( \Delta t\big)^{\frac{3}{2}} }\,
\delta \, \sqrt{ \ln(4\, L_c) + | \ln(\Delta x)|}\; t_M. 		
\label{H-hom-max} 
\end{align}
\end{proposition}
 
\begin{proof}
We proceed as in the proof of Proposition \ref{max_H_multi}. Using \eqref{upp_H_multi_1}, we see that we need to find an  upper estimate of 
${\mathbb E} \big( \max_{0\leq k\leq N} |X_k^{m+\frac{1}{2}}|\big)$, where $\{ X^{m+\frac{1}{2}}_j\}_j$ is a centered Gaussian vector with a covariance matrix 
$\Gamma$. Note that for every $k=0,\dots, N$, $\mbox{\rm Var}(X^{m+\frac{1}{2}}_k) \leq \delta^2$, so that for every $\lambda >0$, 
$$
{\mathbb E}\Big( e^{\lambda \, |X^{m+\frac{1}{2}}_k|}\Big)  
= \frac{2}{\delta \sqrt{2\pi}} \int_0^\infty e^{-\frac{x^2}{2\delta^2} + \lambda x}\, dx \leq 2  e^{\frac{\lambda^2  \delta^2}{2}}.
$$
The next argument is a slight extension of the one used in the proof of \cite[Prop 3.2]{MRY2020}, where the random variables $X^{m+\frac{1}{2}}_j$ were standard Gaussians; it is based on the Pisier lemma (see, e.g., \cite[Lemma 10.1]{Lif2012}).
We have
\begin{equation}\label{Pisier}
{\mathbb E}\Big( \max_{0\leq j\leq N} \big|X^{m+\frac{1}{2}}_j\big|\Big) \leq \delta\; \sqrt{2 \ln\big[ 2\, (N+1)\big]}.
\end{equation}
We include its short proof for completeness. For any $\lambda >0$, using the Jensen inequality and the fact that $x\mapsto e^{\lambda x}$ is increasing, we obtain
\begin{align*}
\exp\Big( \lambda   {\mathbb E}\big[ \max_{0\leq k\leq N} |X^{m+\frac{1}{2}}_k|\big]\Big) \leq &\; {\mathbb E} \Big( \exp\Big[ \lambda \max_{0\leq k\leq N}
\big|X^{m+\frac{1}{2}}_k\big| \Big] \Big) \leq {\mathbb E} \Big( \max_{0\leq k\leq N} \exp\big( \lambda |X^{m+\frac{1}{2}}_k|\big) \Big) \\
\leq & \; \sum_{k=0}^N {\mathbb E} \Big(e^{\lambda \big| X^{m+\frac{1}{2}}_k\big|}\Big) \leq 2\, (N+1) e^{\frac{\lambda^2 \delta^2}{2}}.
\end{align*} 
Taking logarithms, we deduce
$$
{\mathbb E}\Big(  \max_{0\leq k\leq N} |X^{m+\frac{1}{2}}_k| \Big) \leq  \frac{1}{\lambda} \ln\Big(2\,  (N+1) e^{\frac{\lambda^2\, \delta^2}{2}}\Big) 
= \frac{\ln\big[ 2\, (N+1)\big] }{\lambda}  + \frac{\lambda\, \delta^2}{2}, 
$$
for every $\lambda >0$.  Choosing $\lambda = \frac{\sqrt{2 \ln [ 2\, (N+1) ]}}{\delta}$ for $N\geq 1$, concludes the proof of \eqref{Pisier}. 
 
Therefore, as in the proof of Proposition \ref{max_H_multi}, we have
$$
{\mathbb E} \big( H_{\rm dis}[u^M]\big) \leq H_{\rm dis}[u^0] + \sum_{m=0}^{M-1} \frac{\epsilon}{2} M_{\rm dis}[u^0] 
\frac{1}{\sqrt{\Delta t} \sqrt{\Delta x}} \delta\; \sqrt{2 \ln\big[ 2\, (N+1)\big]}.
$$ 
This concludes the proof of \eqref{H-hom-inst}. The inequality \eqref{H-hom-max} is obtained in a similar manner.
\end{proof} 
 
We next compute the upper bounds of the average discrete energy in the two examples of homogeneous noise.
 
\underline{Example 3:} The definition of $\tilde{\Gamma}^{(3)}_\beta$ for a constant space mesh $\Delta x$ implies that 
$\tilde{\Gamma}^{(3)}_\beta(k,k)= \big( \Delta x\big)^{-\beta}$ 
for $k=0,\dots, N$. Therefore, the assumptions of Proposition \ref{prop-Hdis-hom} are satisfied with $\delta^{(3)}_\beta = \big( \Delta x\big)^{- \frac{\beta}{2}}$, 
 and for $\Delta x \in (0,1)$ the estimate \eqref{H-hom-inst} becomes 
$$ 
{\mathbb E} \big( H_{\rm dis}[u^M]\big)  \leq  \; H_{\rm dis}[u^0] + \frac{\epsilon M_{\rm dis}[u^0]}{\sqrt 2 \, \big( \Delta t\big)^{\frac{3}{2}} 
\big(\Delta x)^{\frac{\beta}{2}}}\,\sqrt{ \ln( 2\, L_c) + | \ln(\Delta x)|}\; t_M.
$$ 	
 
\underline{Example 4:} The definition of  $\tilde{\Gamma}^{(4)}_\beta$ for a constant space mesh $\Delta x$  implies that for $k=0,\dots, N$
$$
\tilde{\Gamma}^{(4)}_\beta(k,k) = 
\frac{1}{(\Delta x)^2}\, \left[ \frac{2\, \Delta x}{\beta}  - \frac{2}{\beta^2} \big( 1-e^{-\beta \, \Delta x} 
\big) \right], 
$$
which, from Taylor expansion, is lower and upper bounded as 
$$ 
1-\frac{\beta \Delta x}{3}  \leq \tilde{\Gamma}^{(4)}_\beta (k,k) \leq  1-\frac{\beta \Delta x}{3} + \frac{ (\beta\, \Delta x)^2}{12} \leq 1. 
$$ 
Therefore,  the assumptions of  Proposition \ref{prop-Hdis-hom} are satisfied with
$\delta^{(4)}_\beta = \Big( 1-\frac{\beta \Delta x}{3}   + \frac{ (\beta\, \Delta x)^2}{12}  \Big)^{\frac{1}{2}} $,
and the upper estimate \eqref{H-hom-inst} becomes
\begin{align*}
{\mathbb E} \big( H_{\rm dis}[u^M]\big) & \leq  \; H_{\rm dis}[u^0] + \frac{\epsilon M_{\rm dis}[u^0]}{2   \, \big( \Delta t\big)^{\frac{3}{2}} } \, 
\Big( 1-\frac{\beta \Delta x}{3}  + \frac{ (\beta\, \Delta x)^2}{12} \Big)^{\frac{1}{2}}  \, \sqrt{2 \big[ \ln(2\, L_c) + | \ln(\Delta x)| \big]} \, t_M\\
&\leq \;  H_{\rm dis}[u^0] + \frac{\epsilon M_{\rm dis}[u^0]}{\sqrt 2   \, \big( \Delta t\big)^{\frac{3}{2}} } \, 
\sqrt{ \ln(2\, L_c) + | \ln(\Delta x)| } \, t_M. 
\end{align*}

\subsection{Numerical tracking of discrete energy}
To check the accuracy of our three schemes for this homogeneous type of noise, we show  the error $\mathcal{E}^m[M]$ defined in  \eqref{E:M-error}
in the computation of the discrete mass in the left subplot of Figure \ref{F:scheme3}, and the growth of the discrete energy  in the middle and right subplots of Figure \ref{F:scheme3}. 
There we consider the $L^2$-critical case and take $u_0=0.9Q$ as the initial condition with the noise from Example 3 (Riesz kernel) with $\beta=0.5$
and noise strength $\epsilon = 0.5$. Our other computational parameters are the same as in Figure \ref{F:scheme Ex2}: $L_c=20, \Delta x=0.1, \Delta t = 0.01$.
One can see that the error in mass is similar to the Example 2 in Figure \ref{F:scheme Ex2},  as well as the average energy (both instantaneous and maximal
as defined by  the left-hand sides of \eqref{H-hom-inst} and \eqref{H-hom-max}, respectively) grow and level off similarly. The LE scheme is  slightly 
underperforming  (probably due to slower catching up, since there is no nonlinear correction used in the LE scheme). Changing various parameters, 
we find similar behavior in accuracy, concluding that for all types of  driving noises considered, our numerical simulations are sufficiently accurate. 

\begin{figure}[ht]
\includegraphics[width=0.32\textwidth]{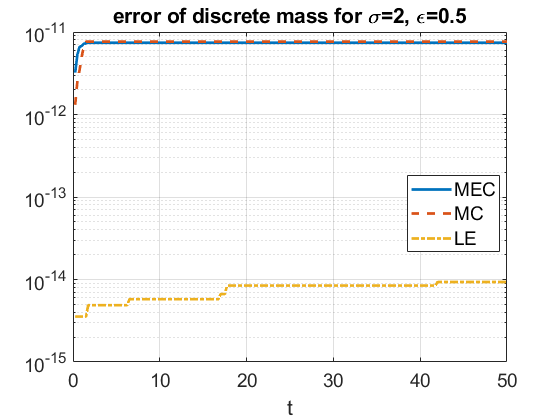} 
\includegraphics[width=0.32\textwidth]{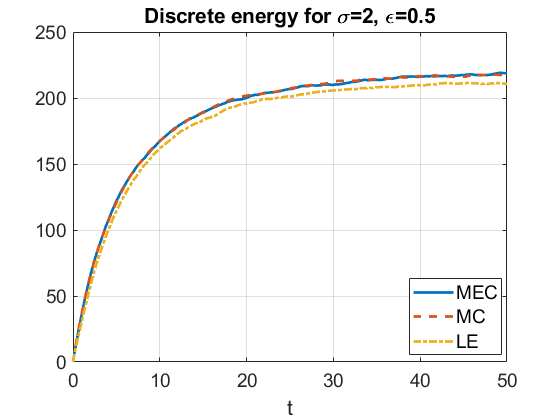} 
\includegraphics[width=0.32\textwidth]{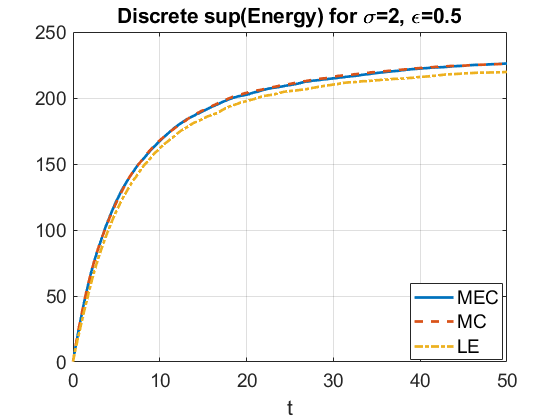}  
\caption{Accuracy of three schemes with noise in Example 3 (Riesz kernel). The $L^2$-critical case ($\sigma =2$) with  $u_0=0.9Q$, $\beta=0.5$, 
$\epsilon=0.5$, $L_c=20$, $\Delta x=0.1$ and $\Delta t=0.01$.
The left plot is the error $\mathcal{E}^m[M]$ \eqref{E:M-error} in computation of the discrete mass for three schemes. 
The growth of average (instantaneous) energy (middle) and max energy (right) from different numerical schemes.} 
\label{F:scheme3}
\end{figure} 

\begin{figure}[ht]
\includegraphics[width=0.32\textwidth]{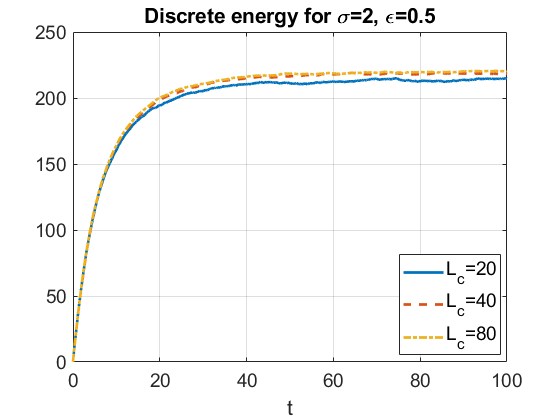}
\includegraphics[width=0.32\textwidth]{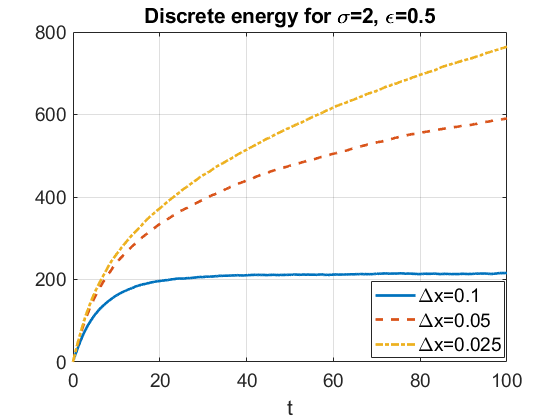} 
\includegraphics[width=0.32\textwidth]{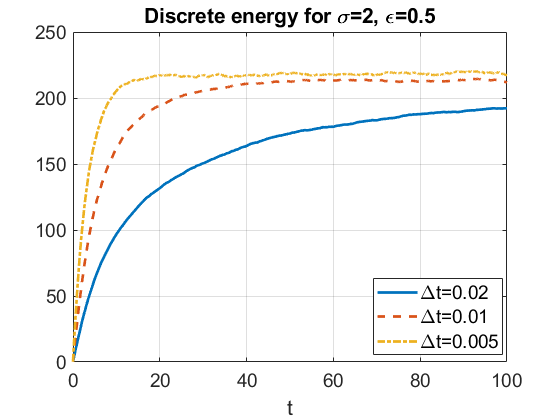}  
\caption{Time evolution of energy and its dependence on parameters $L_c$ (left), $\Delta x$ (middle) and $\Delta t$ (right) in Example 3 (Riesz kernel) 
in the $L^2$-critical case with $n=2$, $\beta=0.5$ and $\epsilon=0.5$.}
\label{F:dx dt Lc Ex3}
\end{figure}  

We next study how the energy is affected by the spatially homogeneous noise from Examples 3 and 4. First, we note that the discrete energy does not depend 
on the length of the computational domain $L_c$, see left subplots in Figures \ref{F:dx dt Lc Ex3} and  \ref{F:dx dt Lc Ex4}. 
In Example 3 (with the Riesz kernel singular at the origin), we observe a clear dependence on the spatial mesh size $\Delta x$: the smaller the size, 
the faster the growth of the energy is; see the middle subplot in Figure \ref{F:dx dt Lc Ex3}. In Example 4 (with a more regular kernel having an exponential decay) there is almost no dependence on $\Delta x$; see middle subplot in Figure \ref{F:dx dt Lc Ex4}. 
The right subplots in Figures \ref{F:dx dt Lc Ex3} and  \ref{F:dx dt Lc Ex4} show the dependence on the time step size $\Delta t$: in Figure \ref{F:dx dt Lc Ex3} 
(Riesz kernel) it has some influence on how fast the energy grows initially, however, eventually it starts leveling off and approaching a horizontal asymptote; 
in Figure \ref{F:dx dt Lc Ex4} (Example 4), where the kernel has an exponential decay,  it takes a longer time to reach the horizontal asymptote, especially for larger 
time steps (e.g. for $\Delta t=0.02$ in the right subplot in Figure \ref{F:dx dt Lc Ex4}). 
In these computations we tracked the energy on the time interval $(0,100)$ for 
comparison purposes, it is possible to obtain longer time tracking (see, for example, Figure \ref{F:Ex4 extend}, however, it does take longer computational time to track the energy growth, since it requires at least 100 trials to run in each particular value of a parameter to approximate the expected values).

\begin{figure}[ht]
\includegraphics[width=0.33\textwidth]{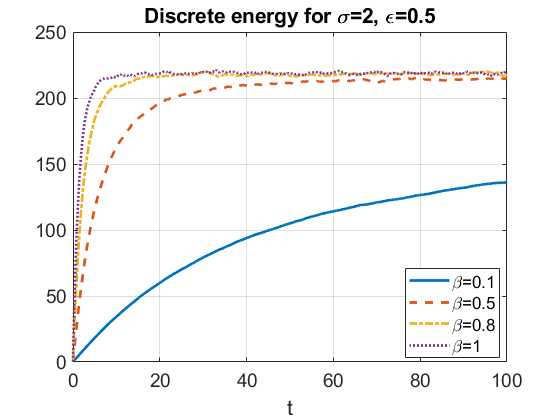} 
\includegraphics[width=0.33\textwidth]{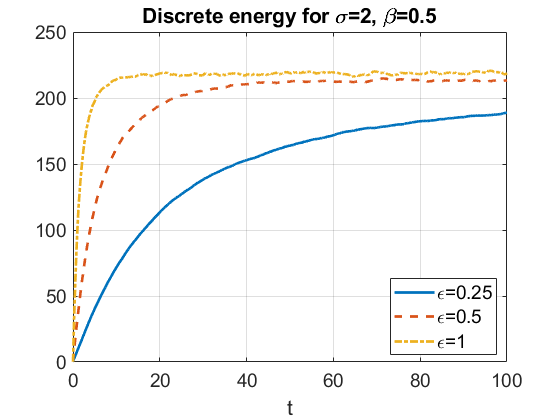}  
\includegraphics[width=0.33\textwidth]{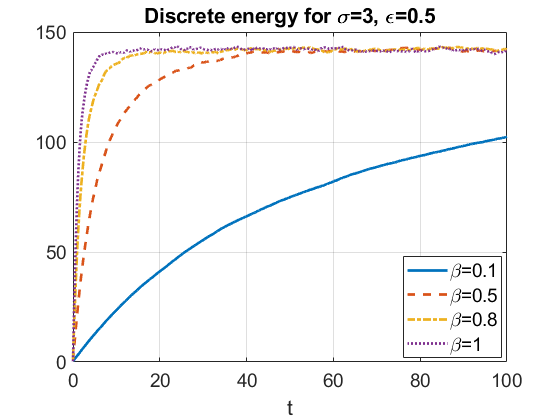} 
\includegraphics[width=0.33\textwidth]{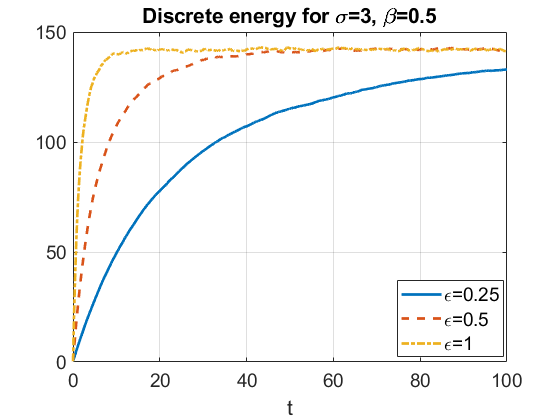}
\caption{The growth of energy for different values of $\beta$ (left) and $\epsilon$ (right) in Example 3 (Riesz kernel) with  $L_c=20$, $\Delta x=0.1$ and $\Delta t=0.01$. Comparison is given in both $L^2$-critical and supercritical cases
for the same $\epsilon=0.5$ on the left two plots, and for the same $\beta=0.5$ on the right two plots.}
\label{F:beta eps Ex3}
\end{figure} 
\begin{figure}[ht]
\includegraphics[width=0.32\textwidth]{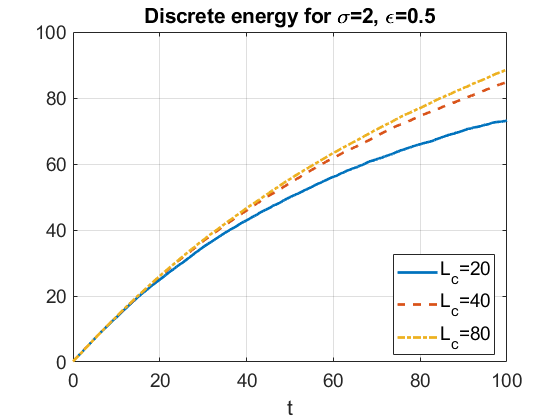}
\includegraphics[width=0.32\textwidth]{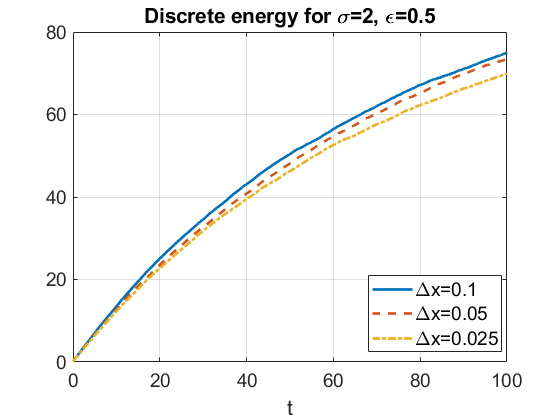} 
\includegraphics[width=0.32\textwidth]{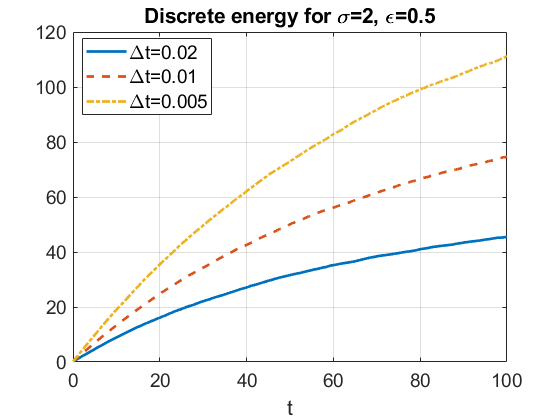}  
\caption{Time evolution of energy and its dependence on parameters $L_c$ (left), $\Delta x$ (middle) and $\Delta t$ (right) in Example 4 (exponential kernel)
in the $L^2$-critical case with $n=2$, $\beta=0.5$ and $\epsilon=0.5$.}
\label{F:dx dt Lc Ex4}
\end{figure}

We next track the influence of the noise strength $\epsilon$ and the correlation parameter $\beta$ on the energy growth in Examples 3 and 4. 
Figures  \ref{F:beta eps Ex3} show that the energy first increases, and then reaches the horizontal asymptote. The leveling off is clearly seen in Figure 
\ref{F:beta eps Ex3}, Example 3 (Riesz kernel). It does not seem to behave similarly in Example 4,
see Figure \ref{F:beta eps Ex4}, however, if we track the energy for longer times, for example, to time $t=200$ -- see Figure \ref{F:Ex4 extend}, 
the energy starts leveling off. Note that in this Example 4, the parameter $\beta$ can have values beyond $1$. Larger values of $\beta$ represent more irregular noise,
 and we note that in that case 
the energy approaches the horizontal asymptote faster; see Figure \ref{F:Ex4 extend}.

\begin{figure}[ht]
\includegraphics[width=0.33\textwidth]{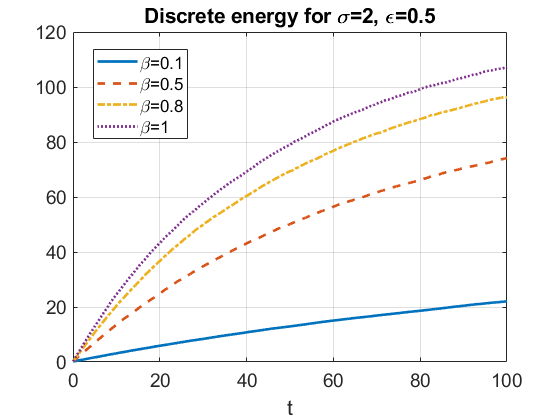} 
\includegraphics[width=0.33\textwidth]{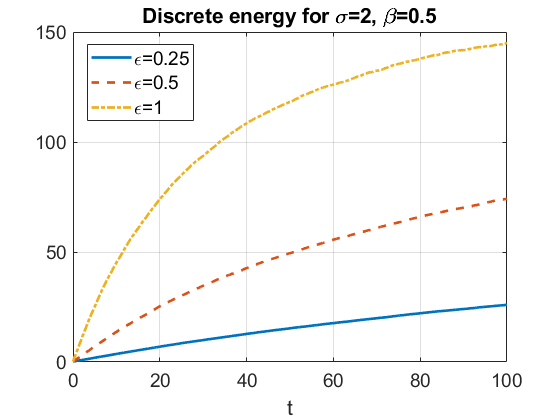} 
\includegraphics[width=0.33\textwidth]{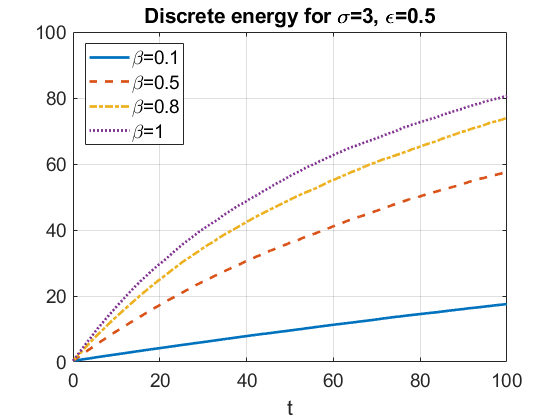} 
\includegraphics[width=0.33\textwidth]{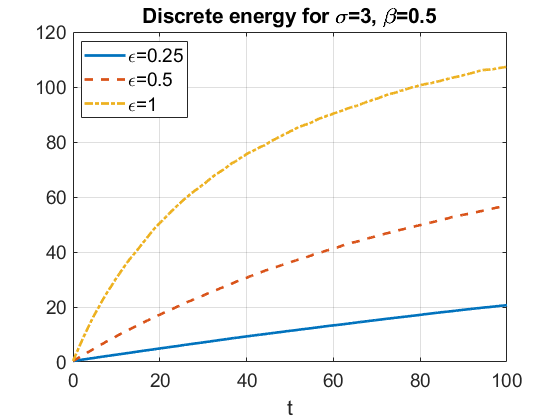}
\caption{The growth of energy for different values of $\beta$ (left) 
and $\epsilon$ (right) in Example 4 (exponential kernel) with  $L_c=20$, $\Delta x=0.1$ and $\Delta t=0.01$. Comparison is given in both $L^2$-critical and supercritical cases for the same $\epsilon=0.5$ on the left two plots, and for the same $\beta=0.5$ on the right two plots.}
\label{F:beta eps Ex4}
\end{figure} 
\begin{figure}[ht]
\includegraphics[width=0.40\textwidth]{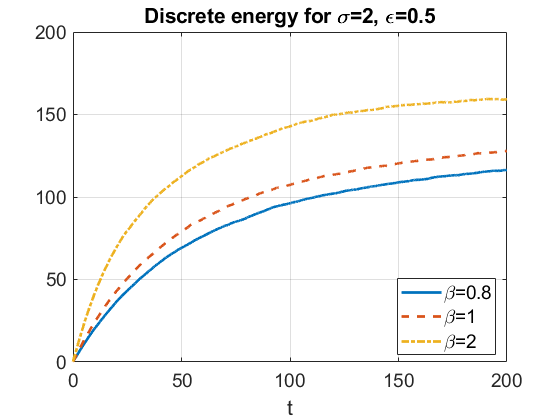} 
\caption{The growth of energy for different values of $\beta$ in Example 4 (exponential kernel) on a longer time interval, $0<t<200$.}
\label{F:Ex4 extend}
\end{figure}

To summarize, the stronger or more irregular the noise is (that is, the larger $\epsilon$ or $\beta$ is), the faster the convergence to the horizontal asymptote becomes.  
We also observe that regardless of the noise strength, the values of the discrete energy converge to the same horizontal asymptote. 

We point out that different types of noise in the limiting cases 
as $\beta = 1$ correspond to  the space-time white noise   (or a finite approximation of it) in Examples 1, 2, and 3 (but not 4), and note that in Figures \ref{Fig:beta eps Ex1}, \ref{F:beta eps Ex2}  and  \ref{F:beta eps Ex3} the energy curve levels at the value produced in the limiting case of $\beta=1$.

\section{Influence of noise on global behavior: blow-up probability}\label{Numerical results}
In this section, we investigate how a multiplicative perturbation  driven by a noise $W$ correlated in space and white in time (via our four examples) affects
the solutions' global behavior: whether it arrests the blow-up so that the solution exists on much longer time intervals, or instead, whether it ceases the global behavior and drives the evolution towards a blow-up. 

\subsection{Comments about mesh-refinement}
When checking a sufficiently long evolution of solutions considered in the previous two sections, a uniform space mesh was sufficient for the numerical simulations.  
However, when studying the blowup/scattering thresholds, one needs to be careful, since the uniform mesh may lead to an oscillatory solution at some amplitude, 
for which the true  solution actually would blow up. Such oscillations might be due to the effect of the 4th order derivative residue term from $u_{xxxx}$ 
when we approximate the discretization of $\partial_{xx}$ by Taylor expansion. 
(For more details on this we refer the reader to \cite{F2015}, \cite{DM2002a}, \cite{DM2002b}.)   
In order to avoid this issue, a mesh-refinement can be implemented (for example, as in   \cite{DM2002a}  or \cite{MRY2020}) to let the solution evolve in time more accurately in numerical simulations. 

As we refine the mesh, we face the recalculation of the covariance matrix. Furthermore, in Examples 3 and 4 (homogeneous noise), a convolution is involved 
leading to computation of a full covariance matrix and its Cholesky decomposition. This part consumes significant time, and  each mesh refinement, 
for example as we did in \cite{MRY2020}, would involve an extra recalculation of the covariance matrix, making the computational time prohibitive to obtain any 
useful results.  

Instead, we have a more efficient approach: instead of using a non-uniform mesh refinement, we start 
by setting {\it a priori} the central region to be refined enough to reach a height identified as blowup (for example, 5 times that of the initial data:
 $\|u(t)\|_{L^\infty}=5\|u_0\|_{L^\infty}$). 
Outside of the central region, we keep the previously used space mesh. 
Thus, by refining specific regions in our computational domain from the beginning,    the mesh refinement is no longer needed later in the computations. 
Hence, the Cholesky decomposition for the  
covariance matrix used in Examples 3 and 4 is done only once in the beginning, saving a large quantity of computational time. 
We use the computational interval with $L_c=10$ and set the initial space mesh as follows: we choose the central region to be $[-1,1]$ and set 
$\Delta x=0.1$ outside of it; inside the central region, that is,  for  $x \in [-1,1]$, we set $\Delta x=0.1/16$. The time mesh  we use is $\Delta t=0.05$. 
In this Section in our simulations, a solution is identified 
as {\it blow-up} if $\|u(t,\cdot)\|_{L^\infty} > 5\|u_0\|_{L^\infty}$, 
and a solution is identified as {\it scattering} if the time evolution did not blow up before the time $t=5$.

\subsection{Probability of blow-up}
We are now able to investigate the noise influence on the overall behavior of solutions to \eqref{E:NLS}. Computationally, the results of this part are the most 
challenging and time consuming. Computing the covariance matrices and the correlated noise via $\mathcal{A}\chi$ in Examples 3 and 4 takes a significant
 amount of computational time, as these quantities are full matrices, which need $O(N^2)$ operations. For other operations, such as matrix multiplication 
 when solving linear or non-linear systems in  \eqref{mass-energy}, \eqref{NS: crank-nicholson} or \eqref{NS:relaxation} schemes, we only need $O(N)$ 
 operations as they are sparse systems. In Examples 3 and 4, we use Nvidia RTX 2070 SUPER and Nvidia GTX1080ti GPUs to compute the covariance matrices, 
 since GPUs are much faster in matrix addition and multiplications. Nevertheless, currently it takes approximately $10$ hours to generate, for example, the right subplot 
 in Figure \ref{Ex3 blow-up percent} on one of our $18$ core Intel i9-7980xe workstations (when using one of the latest versions of Matlab 
with parallel computing command ``parfor"). Furthermore, we also used the HPC\footnote{High Performance Computing (HPC) 
resources at Florida International University.} to perform computations that we show in Figures \ref{Ex1 blow-up percent}-\ref{Ex4 blow-up percent}. 

For each of the four examples of spatially correlated noise considered, we track the  time evolution of solutions with initial data just slightly 
above the ground state $Q$.  In particular, in the $L^2$-critical case ($\sigma=2$) we take  
initial condition $u_0=1.05Q$, and in the $L^2$-supercritical case (with $\sigma =3$) we consider $u_0=1.01Q$. 
In the deterministic setting both of these initial conditions lead to a solution blowing up in finite time (and small perturbations of such data also lead to a blow-up in finite time with a similar dynamics). In the stochastic setting any data (even small) can blow-up with a positive probability (see \cite{dBD2005}), therefore, to track how the probability of blow-up changes, we have to take the initial conditions very close to $Q$.  

Recall that $\epsilon$ is the strength of the noise, which is an important parameter to track for understanding how the noise influences  the global behavior. 
We also investigate how the correlation parameter $\beta$ influences the blow-up (recall that in the first three examples,  $\beta \to 1$ means that the noise $W$ we use approaches the space-time  white noise).
We take $\beta \in [0,1]$ and subdivide this interval into 100 sub intervals
(that is,  we compute the time evolution of solutions with an increment of $0.01$  in  $\beta$).

We track the probability of blow-up as follows: 
in the $L^2$-critical case ($\sigma=2$) we average over 1000 
trials and in the $L^2$-supercritical case (we work with $\sigma=3$) we average over 3000 trials in order  to obtain a  smoother curve of probabilities;    
indeed, we notice that  the probability of blow-up turns out to be higher, and the random blow-up time  varies more. As we mentioned, we mark a numerical run as a 
blow-up solution  if the amplitude becomes higher than 5 times of the original amplitude, and we record a run as scattering if the time evolution did not blow-up 
within the  considered time interval $0<t\leq 5$. We typically consider values of the noise strength $\epsilon = 0.05, 0.1, 0.2$ and $0.5$,
 though in some instances we had to refine it (for example, in Figure \ref{Ex1 blow-up percent} to pin down the interval which affects the blow-up percentage. 
 (The values of the parameters  $L_c$, $\Delta t$ and $\Delta x$ are as described in the previous subsection.)

In Figures \ref{Ex1 blow-up percent} and \ref{Ex2 blow-up percent} we 
show the probability of blow-up for Example 1 (Gaussian-type decay) and Example 2 (polynomial decay). First, observe that as $\beta$ increases to 1, 
the probability of blow-up diminishes; it decreases more significantly in the $L^2$-critical case than in the $L^2$-supercritical case. In fact, in the $L^2$-critical case in Example 1 the noise strength with $\epsilon > 0.09$ seem to eliminate the blow-up completely as $\beta \to 1$, similar dependence is seen in Example 2. 

\begin{figure}[ht]
\includegraphics[width=0.48\textwidth]{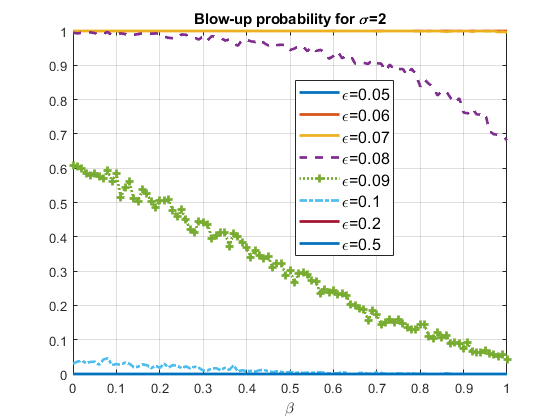}
\includegraphics[width=0.48\textwidth]{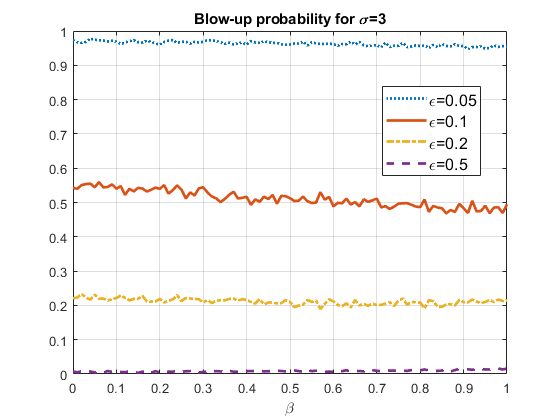} 
\caption{Blow-up probability for different noise strength $\epsilon$ and space correlation $\beta$ in Example 1 (Gaussian-type decay). 
Left: $L^2$-critical ($\sigma=2$) case. Right: $L^2$-supercritical ($\sigma =3$) case.}
\label{Ex1 blow-up percent}
\end{figure} 
\begin{figure}[ht]
\includegraphics[width=0.48\textwidth]{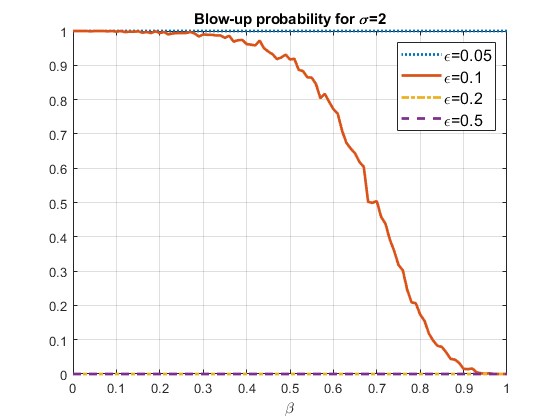}  
\includegraphics[width=0.48\textwidth]{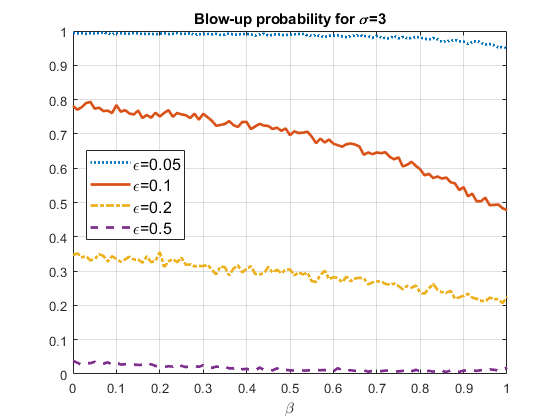} 
\caption{Blow-up probability for different $\epsilon$ and space correlation $\beta$ in Example 2 (polynomial decay)  with $n=4$. 
Left: $L^2$-critical ($\sigma=2$) case. Right: $L^2$-supercritical ($\sigma =3$) case.}
\label{Ex2 blow-up percent}
\end{figure} 

The larger noise strength $\epsilon$ tends to drive solutions away from blow-up into a scattering regime (for example, $\epsilon \geq 0.1$ in
 Figure \ref{Ex2 blow-up percent}), while very small values of $\epsilon$ let the time evolution keep the blow-up behavior (at least on the considered time interval), 
 for example, see curves for $\epsilon =0.05$ in both Figures \ref{Ex1 blow-up percent} and \ref{Ex2 blow-up percent} -- they are more easily identified in 
 the right subplots in the $L^2$-supercritical case, on the very top of the plot.

We next examine the global behavior in Example 3 of a homogeneous noise defined in terms of the Riesz kernel. 
The probability of blow-up is given in Figure \ref{Ex3 blow-up percent}. 
Note that, as we have a stronger spatial correlation when using  this Riesz kernel, the stronger noise tend to arrest blow-up with higher probability and 
force into the scattering  regime. For example, we see no blow-up behavior in solutions in the $L^2$-critical equation with $\epsilon=0.5$ starting with 
$\beta \geq 0.15$; see left plot in Figure \ref{Ex3 blow-up percent}.

\begin{figure}[ht]
\includegraphics[width=0.48\textwidth]{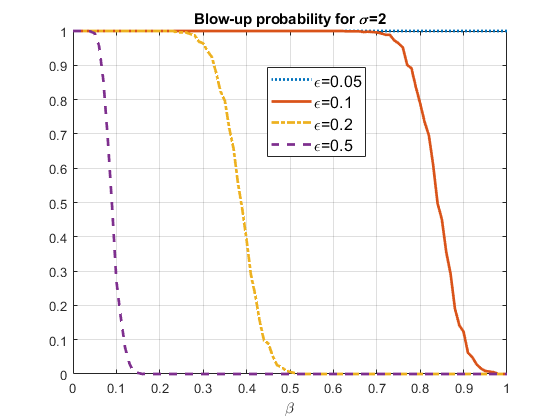}
\includegraphics[width=0.48\textwidth]{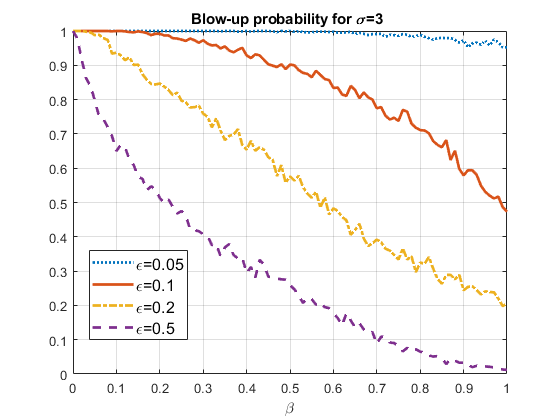}  
\caption{Blow-up probability for different $\epsilon$ and space correlation $\beta$ in Example 3 (Riesz kernel). Left: $L^2$-critical ($\sigma=2$) case. 
Right: $L^2$-supercritical ($\sigma =3$) case.}
\label{Ex3 blow-up percent}
\end{figure} 

\begin{figure}[ht]
\includegraphics[width=0.48\textwidth]{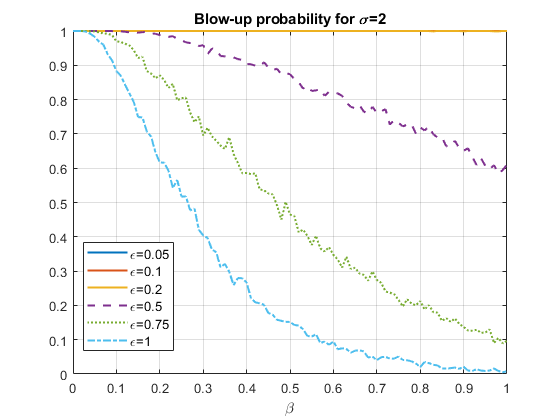}
\includegraphics[width=0.48\textwidth]{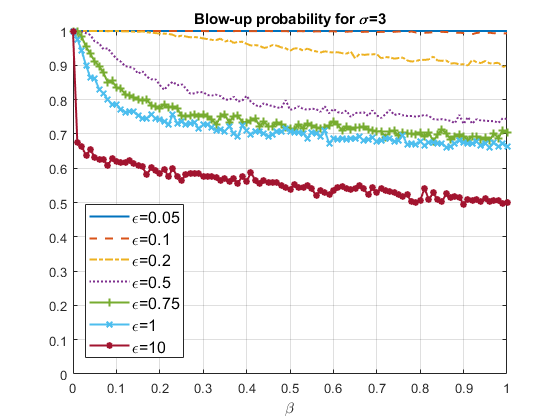}  
\caption{Blow-up probability for different $\epsilon$ and space correlation $\beta$ in Example 4 (exponential decay). Left: $L^2$-critical ($\sigma=2$) case. Right: $L^2$-supercritical ($\sigma =3$) case.} 
\label{Ex4 blow-up percent}
\end{figure}


Finally, we show the probability of blow-up in Example 4 (exponential kernel, $e^{-\beta \, |x|}$) in Figure \ref{Ex4 blow-up percent}. 
On the left it is the  $L^2$-critical case  and on the right it is the $L^2$-supercritical case ($\sigma=3$).
As with the Riesz kernel, we observe that a larger noise strength (such as $\epsilon = 0.5$) tends to arrest blow-up in an  increasing number 
of cases as $\beta$ increases (note that $\beta$ can go beyond $1$ in this example), and it can almost eliminate blow-up, at least in the $L^2$-critical case
 (see $\epsilon=1$ curve on the left subplot of Figure \ref{Ex4 blow-up percent}).  
Note that in the $L^2$-supercritical case in the Example 4, while the probability of scattering slightly increases with growing $\beta$ and with increasing $\epsilon$, 
the blow-up probability curve is not affected as dramatically as in the $L^2$-critical case. We were able to track the values of $\epsilon$ higher than $1$: an example of $\epsilon=10$ is shown on the right subplot of Figure \ref{Ex4 blow-up percent}. (In this case we did $N_t=2000$ trials; a visible initial jump is because for computational purposes when $\beta=0$, we take $\epsilon=0$.)
Thus, the blow-up probability curves in Figure \ref{Ex4 blow-up percent} (also in Figures \ref{Ex1 blow-up percent} and \ref{Ex2 blow-up percent}) are different from the $L^2$-critical case, which corroborates the result in \cite{dBD2005} that in the $L^2$-supercritical case {\it any} sufficiently smooth and sufficiently localized data blows-up in finite time with positive probability. 
We also observe that the stronger the  nonlinearity is ($\sigma=3$ vs. $\sigma=2$) the more resistance to scattering time evolution has. 

Summarizing, in all our examples of spatially-correlated noise, we observe that a larger noise strength $\epsilon$ and 
more concentrated  space-correlation (higher value of $\beta$)
help prevent or delay the blow-up, hence, forcing solutions to exist for longer time. We emphasize that the above 
simulations were done with the initial data that leads to blow-up in finite time in the deterministic case ($\epsilon=0$). 
\smallskip
 
Finally, we want to make a remark about a reverse phenomenon, i.e., when initial data, which in the deterministic case generate solutions existing globally in time 
(moreover, scattering), can in the stochastic case produce a time evolution, which blows up in finite, though random, time with positive probability. 
For example, in Example 3 (Riesz kernel) when taking $u_0=0.99Q$, $\epsilon=0.2$, and considering the $L^2$-supercritical case $\sigma=3$, we observed  1 or 2 blow-up trajectories out of 1000 runs for the values of $\beta=0.25, 0.5$, or $0.75$. While the probability is very low, it is positive; this is consistent with results proved in \cite[Thm 5.1]{dBD2005}.
We ran a similar experiment in the $L^2$-critical case and did not observe any blow-up trajectories in 2000 runs for a variety of values of $\beta$; this is consistent
 with \cite[Thm 2.7]{MR2020}. 
We conclude this section with mentioning that a similar positive probability of blow-up in finite time we observed in the case of space-time white noise in
\cite[end of Section 5]{MRY2020}.

\section{Effect of the noise on blow-up dynamics} \label{blow-up dynamics}

In this section we show how the finite-time blow-up dynamics (rates and profiles) might be affected by the spatially-correlated driving noise. 
To track the blow-up behavior we use the algorithm we introduced in \cite[Section 4]{MRY2020}. Note that the blow-up time is a random time $T=T(\omega)$.  To avoid reaching the blow-up time  $T(\omega)$, we use the non-uniform mesh in time, i.e., 
$\displaystyle \Delta t_m=\frac{\Delta t_0}{\|u(t_m,x)\|_{L^\infty}^{2\sigma}},$ 
where $t_0=0$ and $t_m=\sum_{l=0}^{m-1} \Delta t_l$ is the $m$th time step. 

We also use the spatial mesh-refinement. 
Unlike the previous section, where we can {\it a priori} preset the non-uniform mesh,
here, we need to keep refining the mesh as time evolves. Therefore, the interpolation for the  solution on the new grid points is needed. 
We apply our new mesh-refinement strategy with the mass-conservative interpolation\footnote{The new part in this interpolation is that the mass is preserved before and after the refinement of a spatial interval, see \cite[(4.8) and Figure 14]{MRY2020}.}, introduced in \cite[Section 4]{MRY2020} for the value of $u(t_m, \tilde{x}_j)$ at the new grid points $\tilde{x}_j$ at time $t_m$. This, however, results in the covariance matrices being recomputed and updated at each mesh refinement, slowing down the computations. 
We first check the accuracy of our approach by computing the difference of the mass at different times (for both discrete as in \eqref{E:M} and its approximation 
via the composite trapezoid rule, see (4.10) and (4.11) in \cite{MRY2020} for the definition) in the case of blow-up solutions; see Figure \ref{DM error blow-up}. We observe that our schemes preserve mass very accurately (at least $10^{-11}$ or more precisely) during the blow-up evolution. 
\begin{figure}[ht]
\includegraphics[width=0.45\textwidth]{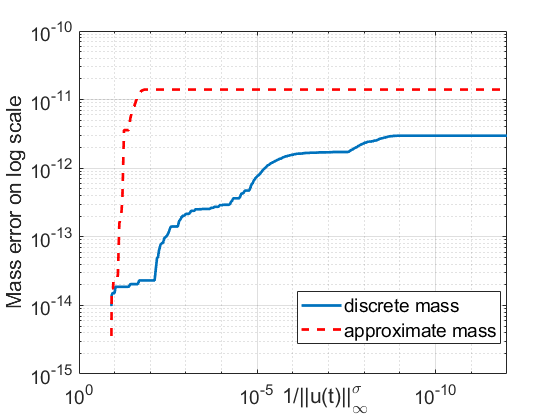}
\caption{Error in mass computation of a blow-up trajectory with $u_0=3 e^{-x^2}$ in Example 3 (Riesz kernel) with $\beta=0.5$ in the $L^2$-supercritical case 
($\sigma =3$).}  
\label{DM error blow-up}
\end{figure} 

A word of caution should be made about the refinements. It was already noted in \cite{DM2002b} that a refinement can affect the outcome of simulations of the global behavior quite severely (e.g., coarser mesh grids can allow the singularity to form, and the finer mesh grids can prevent or delay the blow-up), this also depends on the type of refinement used (in \cite{DM2002b} it was a classical linear interpolation, which does not preserve mass). In our implementation of mesh-refinement, we use a mass-conservative approach, which can be viewed as a next step in investigating blow-up in the stochastic setting. We show that it is sufficiently accurate and robust method to obtain information about the blow-up dynamics. It will be important to investigate further the formation of blow-up and the blow-up dynamics, and in particular, influence of the refinement onto the blow-up; in this work we initiate the study of blow-up dynamics for the white in time colored in space multiplicative stochastic perturbations, with some examples approaching the space-time white noise. 

In what follows we track blow-up dynamics in both $L^2$-critical and $L^2$-supercritical cases. In these simulations we average over $100$ runs. 
To make sure that the time evolution leads to a blow up behavior, we choose the initial condition $u_0 = A \, e^{-x^2}$ with $A \geq 3$ (in both critical and supercritical cases). We set $\Delta t_0=0.002$ and the initial uniform grid mesh size $\Delta x=0.05$ on $x \in [-L_c,L_c]$ with $L_c=5$ (as the blow-up is a local phenomenon, 
a larger value of $L_c$ is unnecessary). We stop our simulations when $\|u\|_{L^\infty}^{\sigma}$ reaches $10^{10}$, or equivalently, $L(t) \sim 10^{-10}$. 
For a review on blow-up dynamics, we refer the reader to \cite[Introduction]{MRY2020}, \cite{YRZ2018} (for the $L^2$-critical case), \cite{YRZ2019} 
(for the $L^2$-supercritical case), or see monographs \cite{SS1999}, \cite{F2015} and references therein.

\begin{figure}[ht]
\includegraphics[width=0.32\textwidth]{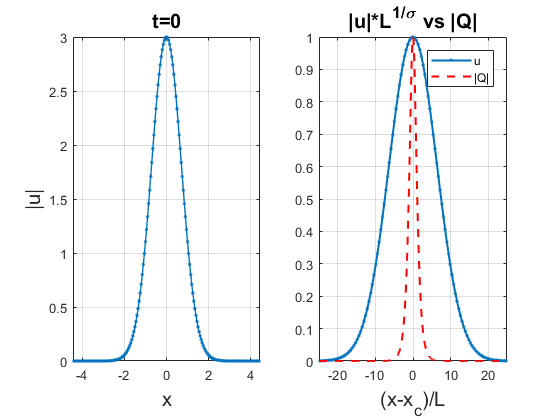} 
\includegraphics[width=0.32\textwidth]{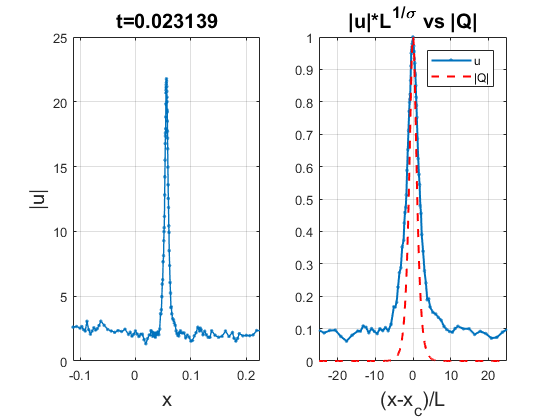}
\includegraphics[width=0.32\textwidth]{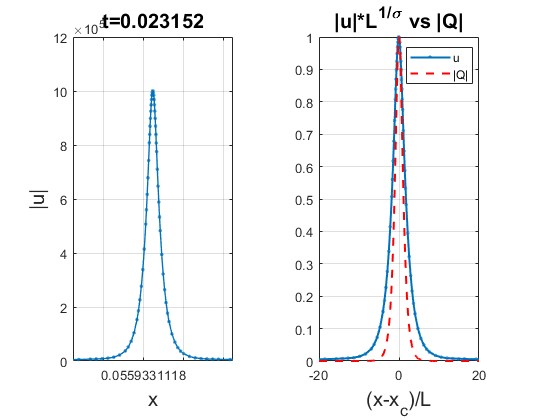} 
\includegraphics[width=0.32\textwidth]{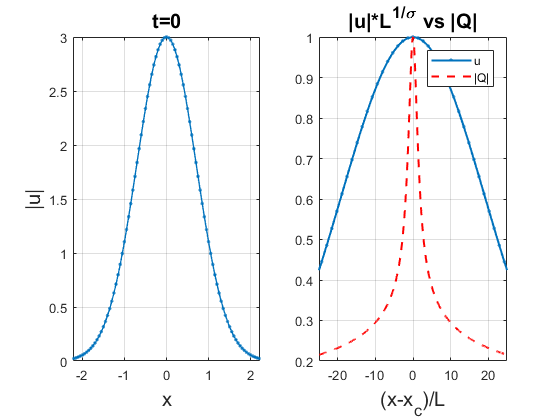} 
\includegraphics[width=0.32\textwidth]{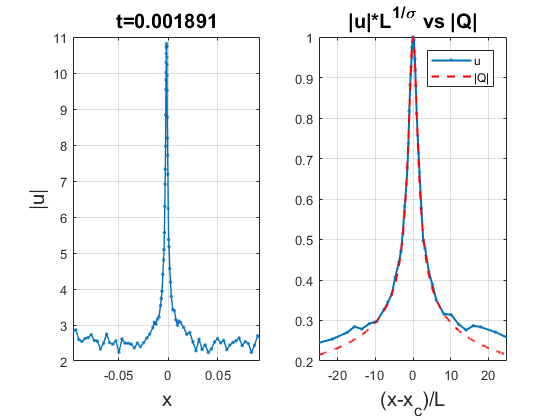}
\includegraphics[width=0.32\textwidth]{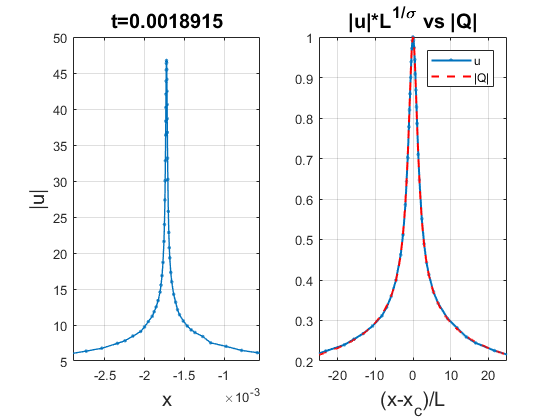} 
\caption{Formation of blow-up in Example 2 (polynomial decay) with $n=4$, $\beta=0.5$ and $\epsilon=0.1$: snapshots of time evolution for $u_0=3\, e^{-x^2}$
 (given in pairs of actual and rescaled solution) at different times. Each pair of graphs shows in solid blue the actual solution $|u|$ and its rescaled version 
 $L^{1/\sigma} |u|$, comparing it to the normalized ground state $Q$ in dashed red.  
Top row: $L^2$-critical ($\sigma =2$) case (blow-up  smooths  out and converges slowly to the ground state $Q$). 
Bottom row: $L^2$-supercritical ($\sigma =3$) case (blow-up profile becomes smooth and converges faster to the profile $Q_{1,0}$).} 
\label{blowup profile EX2}
\end{figure} 
\begin{figure}[ht]
\includegraphics[width=0.32\textwidth]{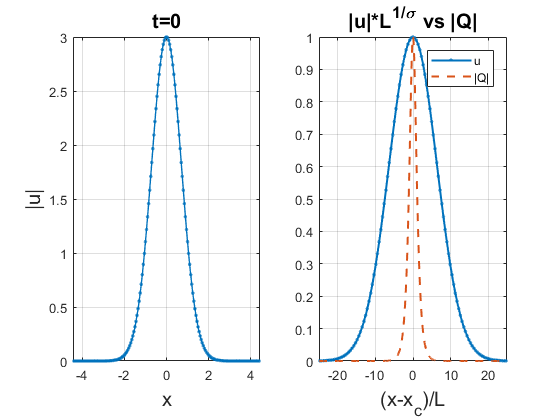} 
\includegraphics[width=0.32\textwidth]{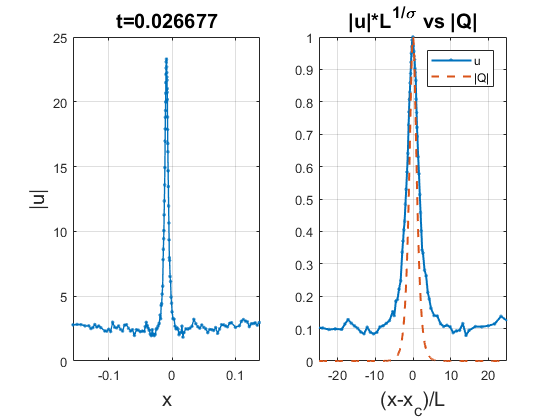}
\includegraphics[width=0.32\textwidth]{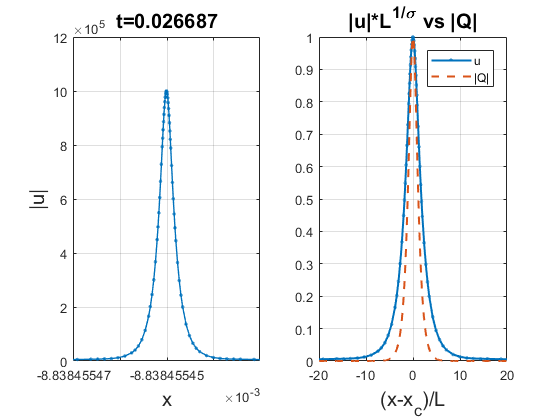} 
\includegraphics[width=0.32\textwidth]{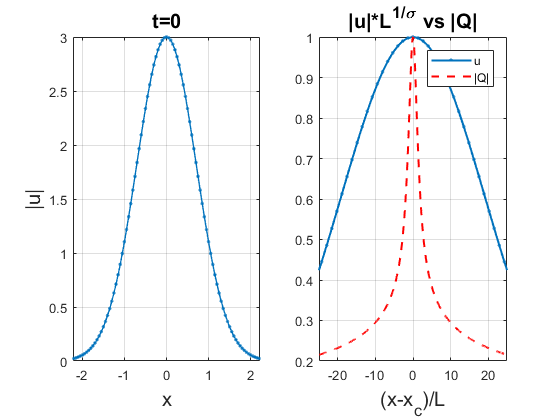} 
\includegraphics[width=0.32\textwidth]{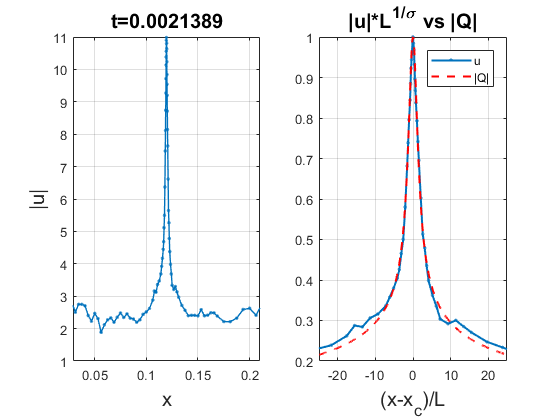}
\includegraphics[width=0.32\textwidth]{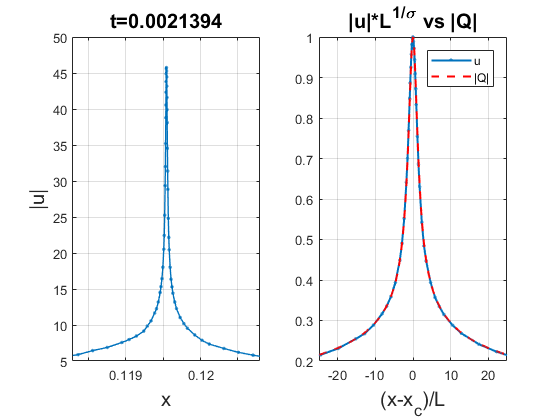} 
\caption{Formation of blow-up in Example 4 (exponential kernel) with  $\beta=0.5$ and $\epsilon=0.1$, for other details, see Figure \ref{blowup profile EX2}. }
\label{blowup profile EX4}
\end{figure} 

We start with tracking the blow-up profiles. Figure \ref{blowup profile EX2} shows snapshots of  blow-up solutions at different times for 
Example 2 (polynomial decay). 
The top row shows the case of the $L^2$-critical blow-up;  one can see that the solution  smooths  and converges slowly to the ground state $Q$. 
(The reason for slow convergence is the same as in the deterministic case: the shown regime is still far from the high focusing level needed to observe the 
convergence.)
The bottom row shows the blow-up in the $L^2$-supercritical case ($\sigma =3$). Observe that in this case the solution  smooths 
out and converges to the  (rescaled) profile solution $Q_{1,0}$ fast;  see the right bottom plot in Figure \ref{blowup profile EX2}. We also show the convergence of the solution in a homogeneous noise Example 4 (exponential decay) in Figure \ref{blowup profile EX4}, observing a similar convergence behavior. The other two examples are in Appendix B in Figures \ref{blowup profile EX1} and \ref{blowup behavior}.

To determine blow-up rates, we check the time dependence of $L(t)$. 
We take $L(t)=\|\nabla u(t)\|_{L^2}^{-(1-s)}$ (note that in the deterministic case it is typical to take $L(t)=\|u(t)\|_{L^\infty}^{-\sigma}$);  however, here, 
we define $L(t)$ via the $L^2$-norm of the gradient, since it gives a more stable computation for the parameter $a$ below; both definitions are equivalent for $s<1$, see \cite{F2015}). 
In the  left subplots of Figure \ref{blowup behavior EX2} we show the logarithmic dependence of $L(t)$ on $T-t$. 
Note that in both critical and supercritical cases, the  slope is  0.5, that is, solutions blow up with a rate $L(t) \sim \sqrt{T-t}$, possibly with some correction terms. 

\begin{figure}[ht]
\includegraphics[width=0.32\textwidth]{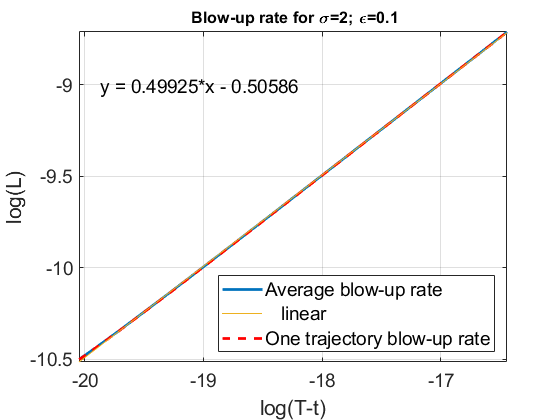}     \includegraphics[width=0.32\textwidth]{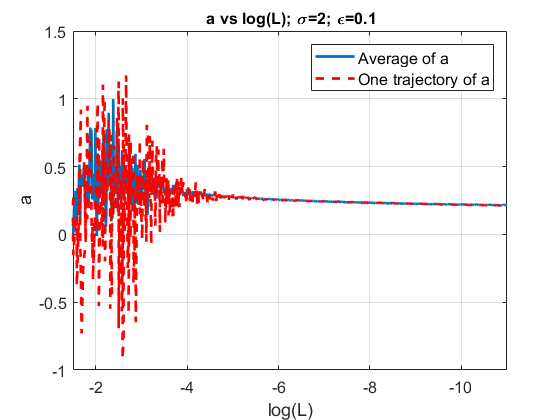}
\includegraphics[width=0.32\textwidth]{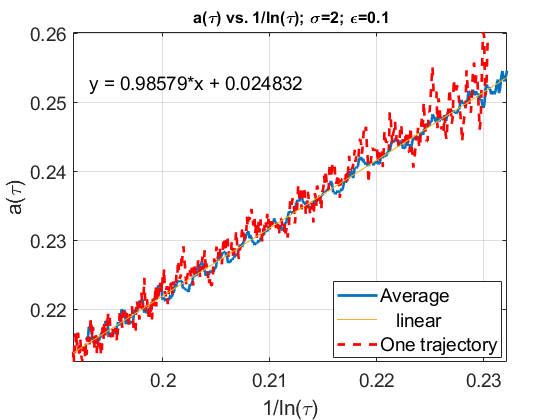} 
\includegraphics[width=0.32\textwidth]{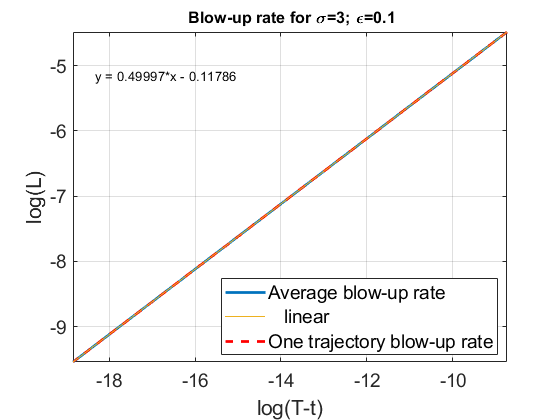} 
\includegraphics[width=0.32\textwidth]{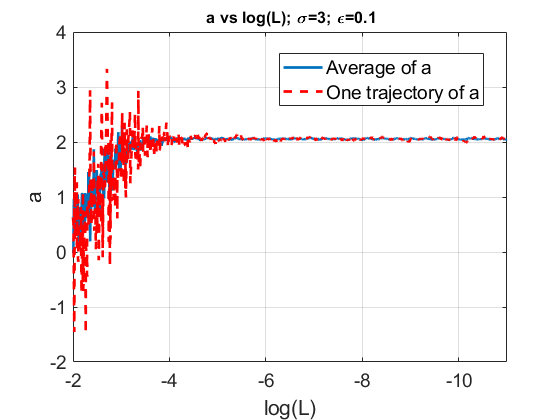}
\includegraphics[width=0.32\textwidth]{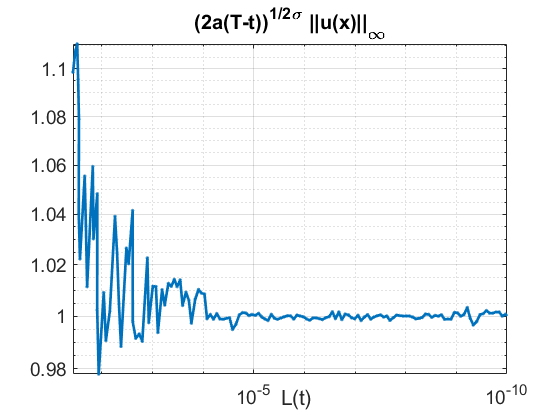} 
\caption{Blow-up rate tracking in Example 2 (polynomial decay) with $n=4$, $\beta=0.5$ and $\epsilon = 0.1$. Top row: $L^2$-critical ($\sigma =2$) case.
 Bottom row: $L^2$-supercritical case. Left: logarithmic dependence of $\log L(t)$ vs. $\log (T-t)$ (note in both cases the slope is $0.50$). 
 Middle: $a(\tau_m)$ vs. $\log L(\tau_m)$ (an extremely slow decay to zero in the top plot and rather fast leveling at a constant level in the bottom plot). 
 Right top: dependence $a(\tau)$ vs $1/ \ln (\tau)$ to confirm the logarithmic correction. Right bottom: fast convergence to a constant $1$ of the quantity
 $\|u\|_{L^\infty} \left( 2a(T-t)\right)^{\frac{1}{2\sigma} }$.} 
\label{blowup behavior EX2}
\end{figure} 
To investigate the correction terms we study the convergence of the parameter $a(t)$ as $t \to T(\omega)$, or equivalently, behavior of
 $a(\tau)$ as $\tau \to \infty$ in the rescaled time $\tau=\int_0^t \frac{1}{L^2(s)}ds$, or $\frac{d\tau}{dt} = \frac1{L^2(t)}$ (then, $t \rightarrow T$ is equivalent to $\tau \rightarrow \infty$). 
As discussed in the introduction, we set $a(t)= -L_t \, L$. A direct calculation (see also \cite{MRY2020}, \cite{SS1999}) yields
$$
a(t) = - \frac{2}{\alpha} \frac{1}{(\| \nabla u(t)\|_{L^2}^{2})^{\frac{2}{\alpha}+1}} \int |u(t)|^{2\sigma}\Im(u_{xx}(t)\,\bar{u}(t)) \, dx, \qquad \alpha=1+\frac{2}{\sigma}.
$$

In the discrete version, we take $\Delta \tau =\Delta t_0$ with $\tau_m = m \cdot \Delta t_0$ as a rescaled time. Then at the $m$th step we obtain values 
$L(\tau _m)$, $u(\tau_m)$, and $a(\tau_m)$. 
We track the behavior of $a(\tau)$ vs. $\log L(\tau)$, which is shown in the middle subplots of Figures \ref{blowup behavior EX2} and \ref{blowup behavior EX4} (see also Figures \ref{blowup behavior EX1} and \ref{blowup behavior}): the red dashed curve shows the behavior for one trajectory of $a$ and the blue solid line shows an averaged over 100 runs behavior. Observe that in the $L^2$-supercritical case (bottom middle plot), $a(\tau)$ converges to a constant after 4 orders of magnitude of $L$, while in the $L^2$-critical case (top middle plot) $a(\tau)$ decays very slowly (to zero). This is similar to the deterministic case. We also track the dependence of $a(\tau)$ on $1/ \ln \tau $ (as in the deterministic case) to show that the correction to the rate in the $L^2$-critical case is slower than  any polynomial power correction. This gives a confirmation that the correction is of a logarithmic order. Our conjecture is that in the SNLS equation,  the correction in the $L^2$-critical case is a double log correction \eqref{E:loglog}, similar to the deterministic case, though it is a highly nontrivial task to show it (as in the deterministic case) and requires further studies. Nevertheless, the above findings give partial confirmation to the blow-up dynamics stated in Conjecture \ref{C:1}. 
 
In the $L^2$-supercritical case, since the convergence of $a(t)$ to a constant is very fast, solving the ODE $a(t)=-L_t L$ with $L(T)=0$, gives 
$L(t)=\sqrt{2a(T-t)}$. 
The bottom right subplot of Figure \ref{blowup behavior} confirms this,  justifying the rate in Conjecture \ref{C:2}.

\begin{figure}[ht]
\includegraphics[width=0.32\textwidth]{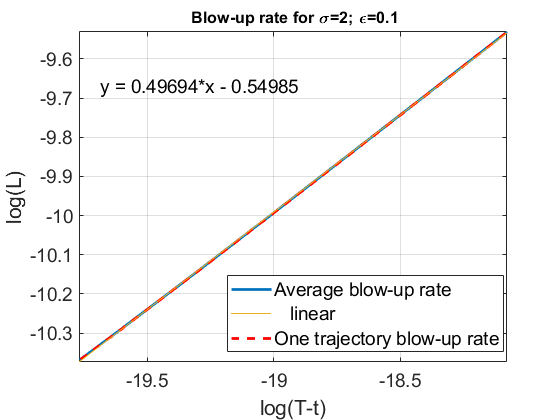}     \includegraphics[width=0.32\textwidth]{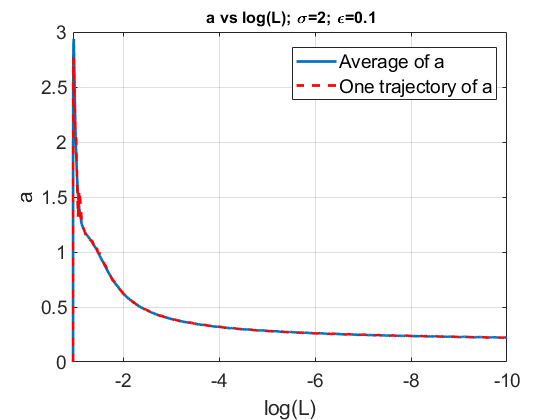}
\includegraphics[width=0.32\textwidth]{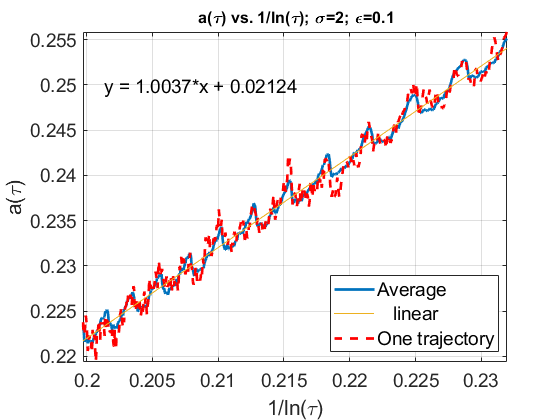} 
\includegraphics[width=0.32\textwidth]{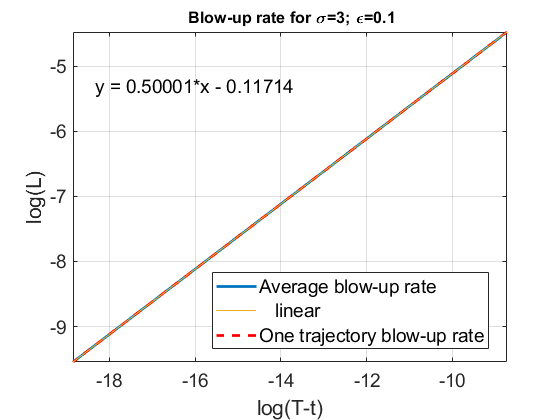} 
\includegraphics[width=0.32\textwidth]{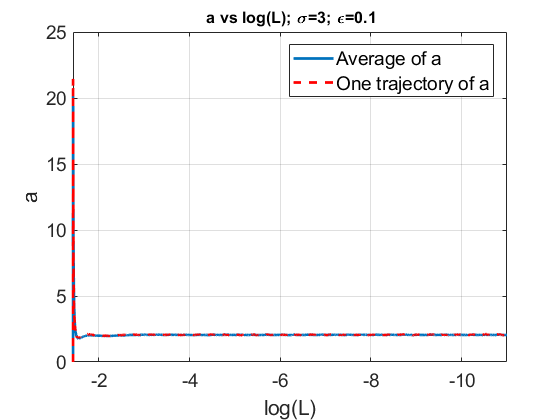}
\includegraphics[width=0.32\textwidth]{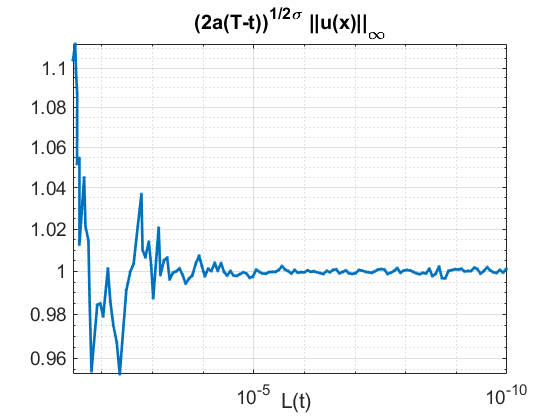} 
\caption{Blow-up rate tracking in Example 4 (exponential kernel) with $\beta=0.5$ and $\epsilon = 0.1$. Top row: $L^2$-critical ($\sigma =2$) case.
 Bottom row: $L^2$-supercritical case. Left: logarithmic dependence of $\log L(t)$ vs. $\log (T-t)$ (note in both cases the slope is $0.50$). 
 Middle: $a(\tau_m)$ vs. $\log L(\tau_m)$ (decay to zero in the top plot and an almost immediate leveling at a constant level in the bottom plot). 
 Right top: dependence $a(\tau)$ vs $1/ \ln (\tau)$ to confirm the logarithmic correction. Right bottom: fast convergence to a constant 1 of the quantity
 $\|u\|_{L^\infty} \left( 2a(T-t)\right)^{\frac{1}{2\sigma} }$.} 
\label{blowup behavior EX4}
\end{figure}

In \cite{MRY2020} we observed that the noise affects the location of the blow-up center, and shifts it away from the origin (or from the original peak location), 
making it a random variable distributed normally. In this work we also check the location of blow-up centers for different runs;  see 
Figures \ref{EX1-location}-\ref{EX4-location} for the distribution curves (we did $1000$ runs, except for the right graph in Figure \ref{EX4-location} with $N_t=2000$ runs) and Tables \ref{Tab:1}-\ref{Tab:2} 
for the mean and variance  in  each of our four examples in the $L^2$-critical and supercritical cases ($\sigma =3$) with different correlation parameters $\beta$. 
Our findings confirm the normal distribution of the curves (conditionally to the existence of blow-up), and that the shifting is more prominent in the $L^2$-supercritical case (larger variance); the variance also grows with the correlation parameter $\beta$ in the supercritical case. The mean of the distribution varies but remains quite small; we think that with a larger number of trials it would converge to zero. 

We conclude that 
the spatially-correlated noise has little effect on the blow-up dynamics, similar to our findings in \cite{MRY2020} for the approximation of the space-time white noise. 
In particular, the driving noise considered in this paper has almost no effect on the blow-up profiles and rates, but shifts the location of the blow-up center.

\section{Conclusion}\label{conclusion}
In this  paper we investigate how solutions behave in the 1D focusing SNLS subject to a multiplicative stochastic perturbation driven by a space-correlated Wiener process. 
We consider four different examples of space-correlation, where the noise is either driven by $\mathcal{Q}$-Brownian motions, thus, with a trace-class (diagonal) 
covariance operator, or by a homogeneous Wiener process, where the covariance matrix is no longer diagonal and has longer range effects. Due to the Stratonovich 
integral, the mass is conserved in this stochastic setting;  however, the energy (or Hamiltonian) changes in time. We observe that in our examples the energy grows 
first and then levels off to a horizontal asymptote whose value is close to a corresponding one in the case of the space-time white noise (to be precise, the 
approximation of it). 
We then investigate the effect of the spatially-correlated noise onto the probability of blow-up. We note that a larger strength of the stochastic perturbation
 and a greater concentration close to the origin for  spatially 
correlated driving noises tend to decrease the probability of blow-up in both $L^2$-critical and supercritical cases, though bearing more influence in the critical case.
 In the supercritical case we also observe that the spatially-correlated noise can drive the evolution towards the blow-up for the initial data that in the deterministic 
 setting would generate solutions existing globally and scattering (to linear solutions). This is in agreement with results proved in \cite{dBD2005} for a regular driving noise, and with   our previous  numerical findings in \cite{MRY2020} for the
 SNLS equation driven by an approximation of the space-time white noise.
Finally, we study the blow-up dynamics, and confirm that the spatially-correlated noise as we consider in this work has almost no influence on profiles or rates of the blow-up, and only affects the location of the blow-up center. This is  similar to our findings for the space-time white noise in \cite{MRY2020}. 
Once the evolution is driven into the blow-up regime, the dynamics except for the blow-up center location is the same as in the deterministic case.

\section*{Appendix A}\label{S:A}
In this appendix we show the distribution of the locations of blow-up center $x_c$ as its being influenced by the noise. 
We provide Figures \ref{EX1-location}-\ref{EX4-location} for each of our four examples when we run $N_t=1000$ trials in most of them except for the right part in Figure \ref{EX4-location}, where we did $N_t=3000$ trials (note how much more accurate the convergence to the normal distribution is, however, this takes significantly larger computational efforts.) In Tables \ref{Tab:1}-\ref{Tab:2} 
we record the mean $\mu_{x_c}$ and variance $\sigma^2_{x_c}$ of the normal distribution that we obtain for the location random variable $x_c$. 
Observe that the variance noticeably increases in Examples 3 and 4 in the $L^2$-critical case as $\beta$ increases (similar increase is happening in these Examples in the $L^2$-supercritcal case, see Table \ref{Tab:2}).

\begin{figure}[ht]
\includegraphics[width=0.45\textwidth]{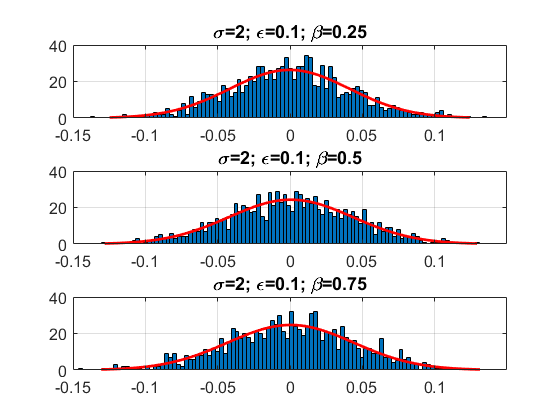} \includegraphics[width=0.45\textwidth]{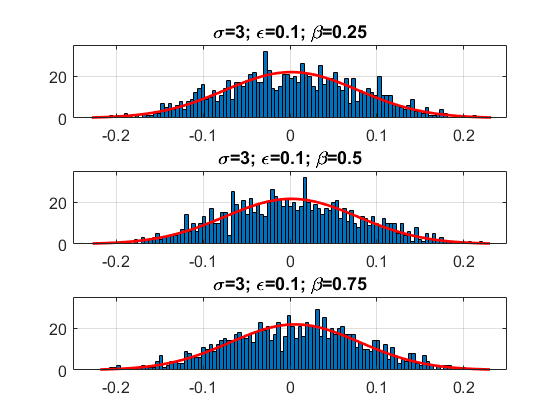}
\caption{Location distribution of the blow-up center $x_c$ in Example 1 for different $\beta$ from $1000$ runs. Here, $u_0=3 e^{-x^2}$ and $\epsilon=0.1$.  
Left: $\sigma=2$. Right: $\sigma=3$.} 
\label{EX1-location}
\end{figure} 

\begin{figure}[ht]
\includegraphics[width=0.45\textwidth]{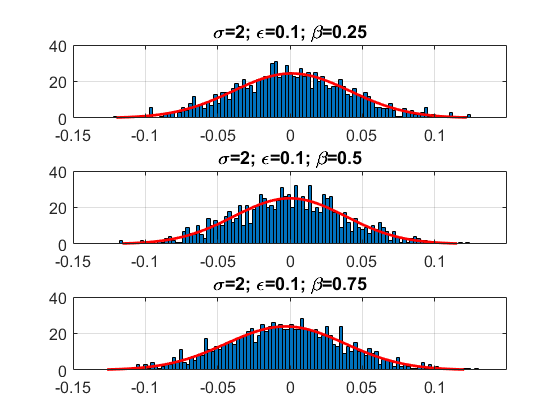} \includegraphics[width=0.45\textwidth]{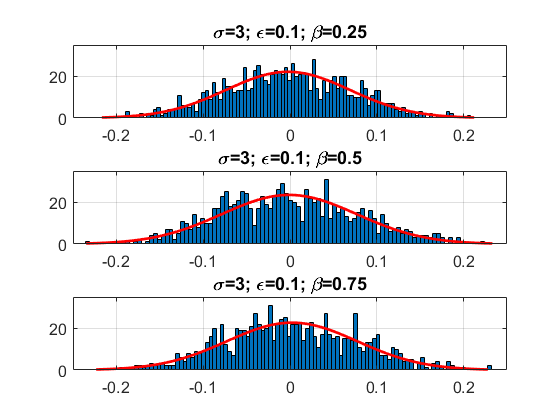}
\caption{Location distribution of the blow-up center $x_c$ in Example 2 (polynomial decay)  with $n=2$ for different $\beta$ from $1000$ runs. 
Here, $u_0=3 e^{-x^2}$ and $\epsilon=0.1$.  
Left: $\sigma=2$. Right: $\sigma=3$.} 
\label{EX2-location}
\end{figure}

\begin{table}[!htb]
\parbox{.48\linewidth}{
\centering
\quad
\begin{tabular}{|c|c|c|c|}
\hline
 $\sigma$   & $\beta$ & $\mu_{x_c}$ & $\sigma^2_{x_c}$   \\
 \hline
 $2$& $0.25$  & $-2.4e-4$ & $0.0013$   \\
 \hline
 $2$ & $0.5$ &$-5.7e-4$& $0.0014$    \\
 \hline
 $2$ & $0.75$& $-4.1e-4$ &$0.0014$   \\
 \hline
  $3$& $0.25$  & $0.0017$ & $0.0024$   \\
 \hline
 $3$ & $0.5$ &$0.0014$& $0.0024$    \\
 \hline
 $3$ & $0.75$& $0.0059$ &$0.0024$   \\
 \hline
\end{tabular}
\linebreak
}
\parbox{.48\linewidth}{
\qquad
\begin{tabular}{|c|c|c|c|}
\hline
 $\sigma$   & $\beta$ & $\mu_{x_c}$ & $\sigma^2_{x_c}$  \\
 \hline
 $2$& $0.25$  & $0.0011$ & $0.0013$   \\
 \hline
 $2$ & $0.5$ &$-1.3e-4$& $0.0012$    \\
 \hline
 $2$ & $0.75$& $-0.003$ &$0.0013$   \\
 \hline
  $3$& $0.25$  & $-0.0023$ & $0.0023$   \\
 \hline
 $3$ & $0.5$ &$-9e-4$& $0.0025$    \\
 \hline
 $3$ & $0.75$& $0.0023$ &$0.0024$   \\
 \hline
\end{tabular}
\linebreak
}
\caption{Mean $\mu_{x_c}$ and variance $\sigma^2_{x_c}$ of the location of blow-up center random variable $x_c$. Left: Example 1 (Gaussian decay) shown in Figure \ref{EX1-location}. 
Right: Example 2 (polynomial decay, $n=2$) shown in Figure \ref{EX2-location}.} 
\label{Tab:1}
\end{table}

\begin{figure}[ht]
\includegraphics[width=0.45\textwidth]{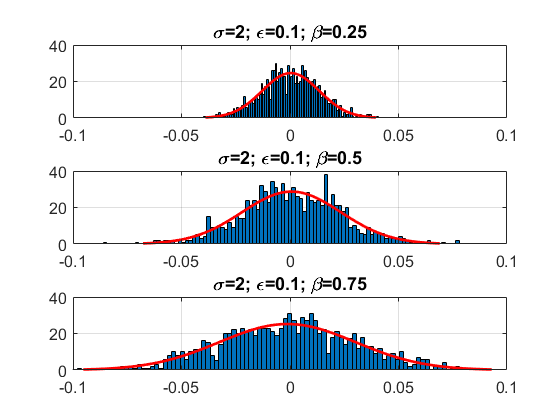} \includegraphics[width=0.45\textwidth]{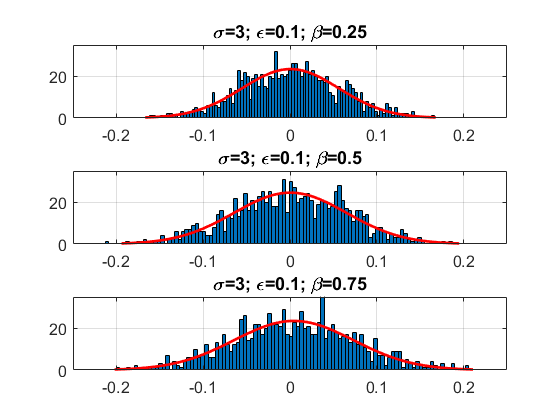}
\caption{Location distribution of the blow-up center $x_c$ in Example 3 (Riesz kernel) for different $\beta$ from $1000$ runs. Here, 
$u_0=3 e^{-x^2}$ and $\epsilon=0.1$.  
Left: $\sigma=2$. Right: $\sigma=3$.} 
\label{EX3-location}
\end{figure} 

\begin{figure}[ht]
\includegraphics[width=0.45\textwidth]{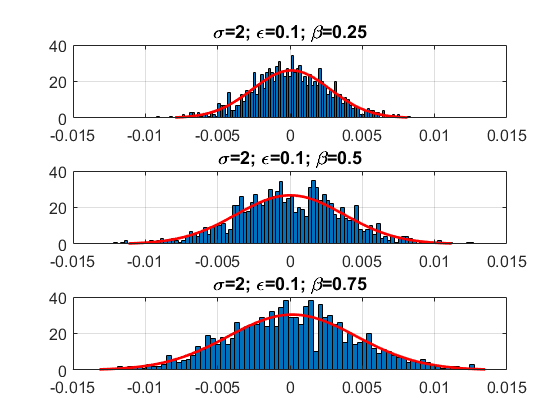} \includegraphics[width=0.45\textwidth]{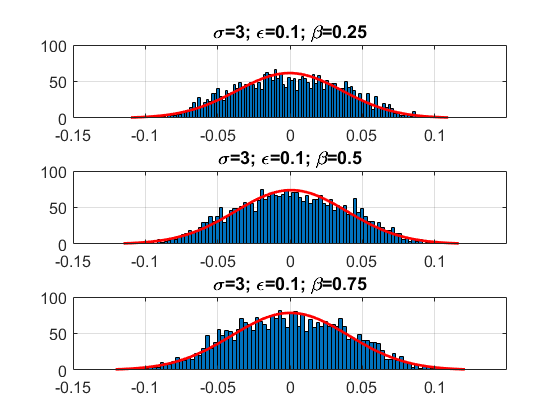}
\caption{Location distribution of the blow-up center $x_c$ in Example 4 (exponential kernel) for different $\beta$. Here, $u_0=3 e^{-x^2}$ and $\epsilon=0.1$.  
 Left: $\sigma=2$ with $1000$ runs. Right: $\sigma=3$ with $3000$ runs (compare this much smoother convergence to a normal distribution with Figure \ref{EX3-location} with 1000 runs). 
} 
\label{EX4-location}
\end{figure} 

\begin{table}[!htb]
\parbox{.48\linewidth}{
\centering
\quad
\begin{tabular}{|c|c|c|c|}
\hline
 $\sigma$   & $\beta$ & $\mu_{x_c}$ & $\sigma^2_{x_c}$   \\
 \hline
 $2$& $0.25$  & $3.3e-4$ & $4.1e-4$   \\
 \hline
 $2$ & $0.5$ &$6.6e-4$& $7.2e-4$    \\
 \hline
 $2$ & $0.75$& $-0.0012$ &$9.9e-4$   \\
 \hline
  $3$& $0.25$  & $5.3e-4$ & $0.0018$   \\
 \hline
 $3$ & $0.5$ &$3.5e-4$& $0.0020$    \\
 \hline
 $3$ & $0.75$& $0.0044$ &$0.0022$   \\
 \hline
\end{tabular}
\linebreak
}
\parbox{.48\linewidth}{
\qquad
\begin{tabular}{|c|c|c|c|}
\hline
 $\sigma$   & $\beta$ & $\mu_{x_c}$ & $\sigma^2_{x_c}$   \\
 \hline
 $2$& $0.25$  & $8.8e-5$ & $8.5e-5$   \\
 \hline
 $2$ & $0.5$ &$5.4e-6$& $1.2e-4$    \\
 \hline
 $2$ & $0.75$& $1.6e-4$ &$1.4e-4$   \\
 \hline
  $3$& $0.25$  & $-3.3e-4$ & $6.7e-4$   \\
 \hline
 $3$ & $0.5$ &$7.9e-4$& $7.0e-4$    \\
 \hline
 $3$ & $0.75$& $2.1e-4$ &$7.4e-4$   \\
 \hline
\end{tabular}
\linebreak
}
\caption{Mean $\mu_{x_c}$ and variance $\sigma^2_{x_c}$ of the location of blow-up center random variable $x_c$. Left: Example 3 (Riesz kernel) shown in Figure \ref{EX3-location}. Right: Example 4 (exponential kernel) shown in Figure \ref{EX4-location}.}
\label{Tab:2}
\end{table}

\clearpage



\section*{Appendix B}

Here we show figures of blow-up dynamics (convergence of profiles and rates) in Examples 1 and 3.


\begin{figure}[ht]
\includegraphics[width=0.32\textwidth]{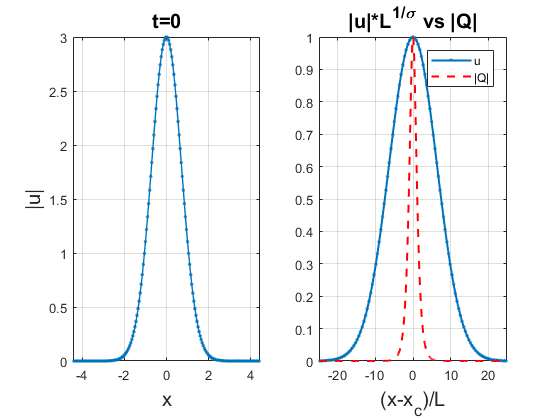} 
\includegraphics[width=0.32\textwidth]{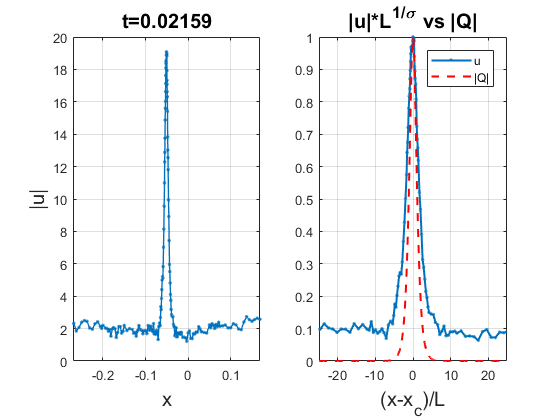}
\includegraphics[width=0.32\textwidth]{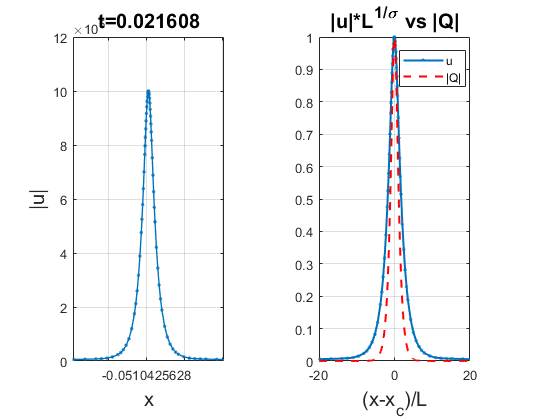} 
\includegraphics[width=0.32\textwidth]{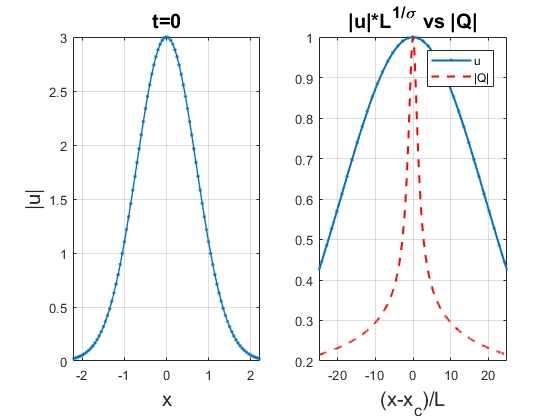} 
\includegraphics[width=0.32\textwidth]{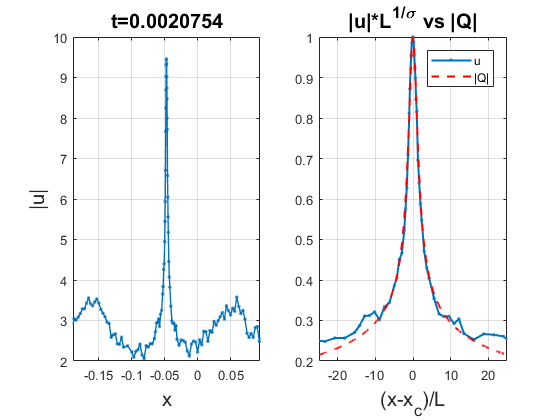}
\includegraphics[width=0.32\textwidth]{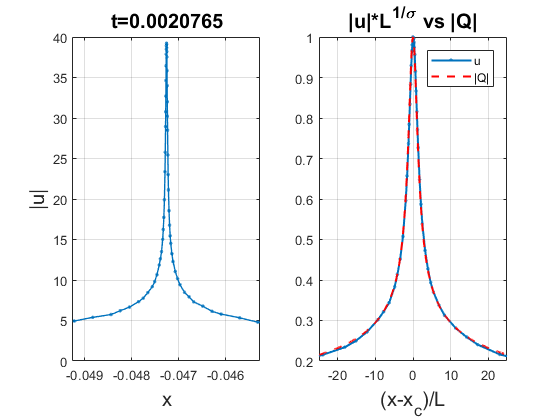} 
\caption{Formation of blow-up in Example 1 (exponential decay) with  $\beta=0.5$ and $\epsilon=0.1$: snapshots of time evolution for $u_0=3\, e^{-x^2}$
 (given in pairs of actual and rescaled solution) at different times. Each pair of graphs shows in solid blue the actual solution $|u|$ and its rescaled version 
 $L^{1/\sigma} |u|$, comparing it to the normalized ground state $Q$ in dashed red.  
Top row: $L^2$-critical ($\sigma =2$) case (blow-up  smooths  out and converges slowly to the ground state $Q$). 
Bottom row: $L^2$-supercritical ($\sigma =3$) case (blow-up profile becomes smooth and converges faster to the profile $Q_{1,0}$).} 
\label{blowup profile EX1}
\includegraphics[width=0.32\textwidth]{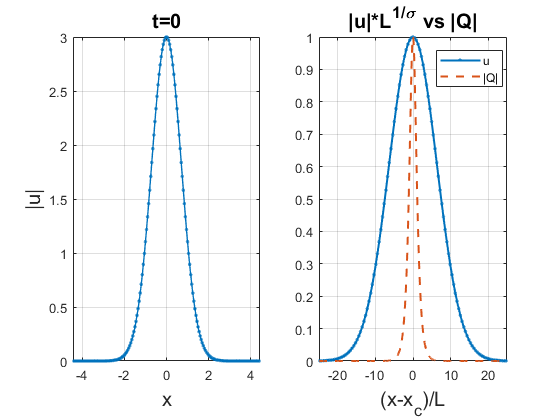} 
\includegraphics[width=0.32\textwidth]{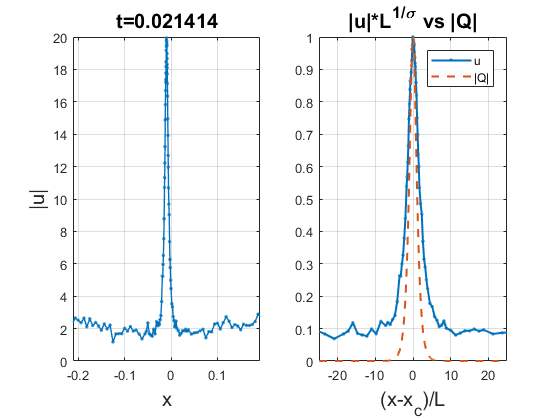}
\includegraphics[width=0.32\textwidth]{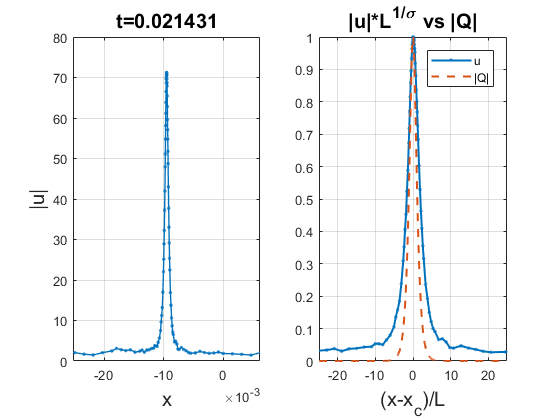} 
\includegraphics[width=0.32\textwidth]{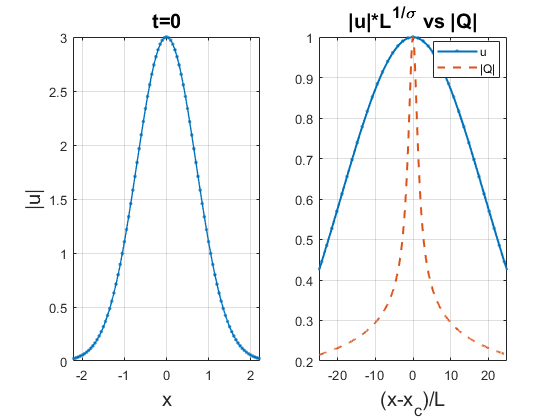} 
\includegraphics[width=0.32\textwidth]{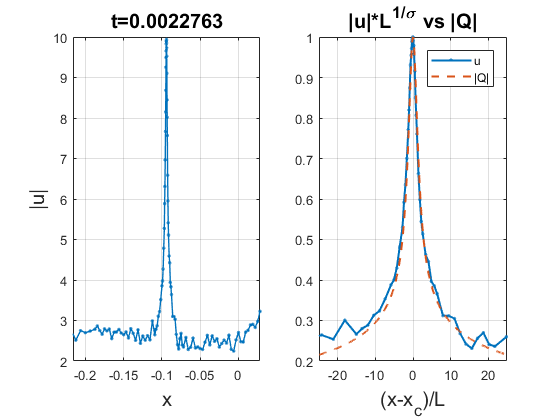}
\includegraphics[width=0.32\textwidth]{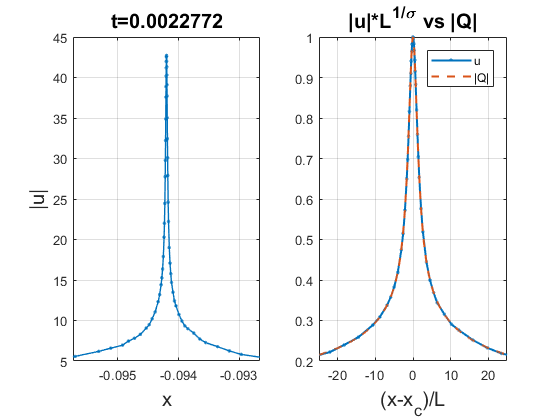} 
\caption{Formation of blow-up in Example 3 (Riesz kernel) with  $\beta=0.5$ and $\epsilon=0.1$: snapshots of time evolution for $u_0=3\, e^{-x^2}$. For other details, see Figure \ref{blowup profile EX1}. }
\label{blowup profile}
\end{figure} 
\begin{figure}[ht]
\includegraphics[width=0.32\textwidth]{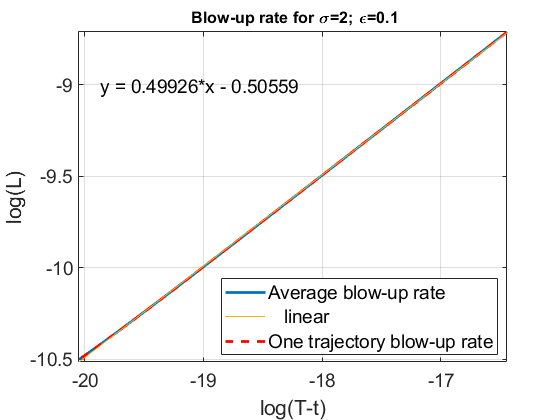}     \includegraphics[width=0.32\textwidth]{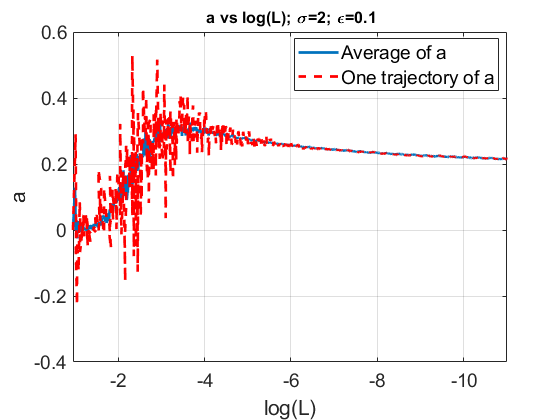}
\includegraphics[width=0.32\textwidth]{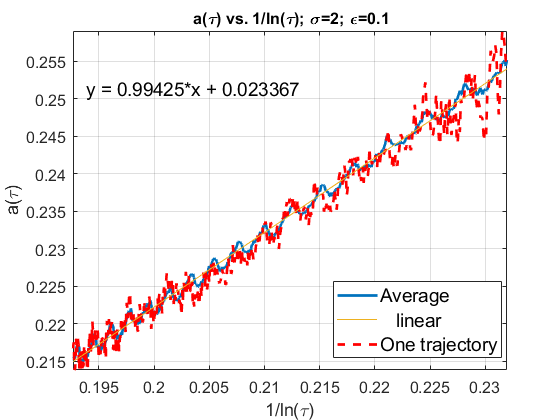} 
\includegraphics[width=0.32\textwidth]{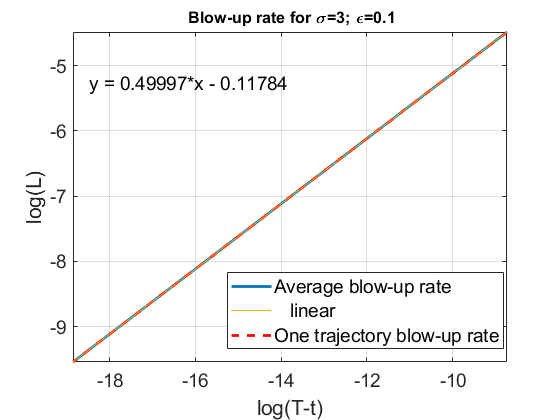} 
\includegraphics[width=0.32\textwidth]{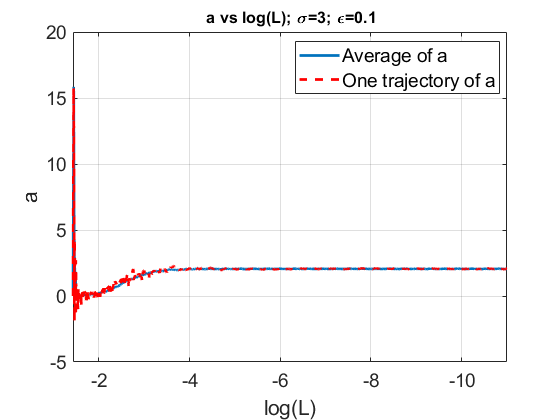}
\includegraphics[width=0.32\textwidth]{SNLS_1d7p_rate_EX1.png} 
\caption{Blow-up rate tracking in Example 1 (exponential decay) with $\beta=0.5$ and $\epsilon = 0.1$. Top row: $L^2$-critical ($\sigma =2$) case.
 Bottom row: $L^2$-supercritical case. Left: logarithmic dependence of $\log L(t)$ vs. $\log (T-t)$ (note in both cases the slope is $0.50$). 
 Middle: $a(\tau_m)$ vs. $\log L(\tau_m)$ (an extremely slow decay to zero in the top plot and rather fast leveling at a constant level in the bottom plot). 
 Right top: dependence $a(\tau)$ vs $1/ \ln (\tau)$ to confirm the logarithmic correction. Right bottom: fast convergence to a constant of the quantity
 $\|u\|_{L^\infty} \left( 2a(T-t)\right)^{\frac{1}{2\sigma} }$.} 
\label{blowup behavior EX1}
\end{figure} 
\begin{figure}[ht]
\includegraphics[width=0.32\textwidth]{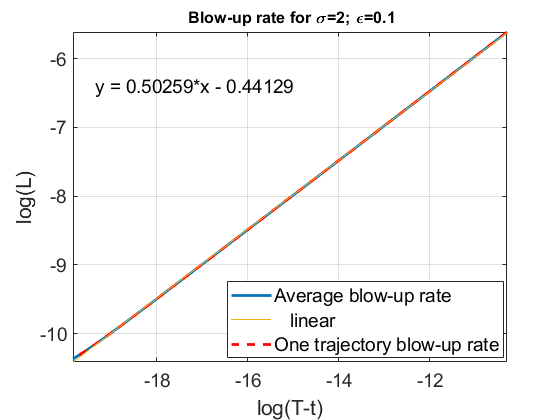}     \includegraphics[width=0.32\textwidth]{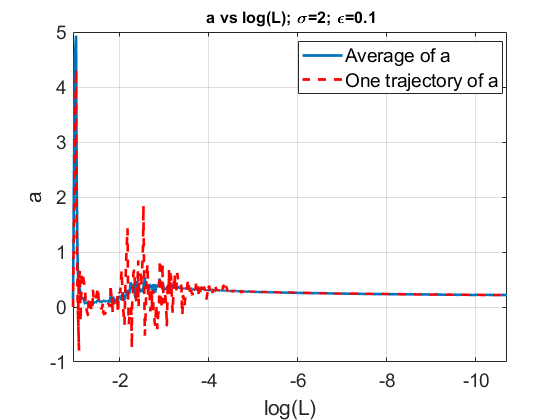}
\includegraphics[width=0.32\textwidth]{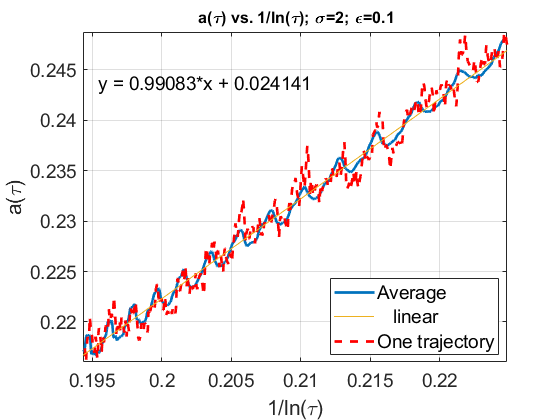} 
\includegraphics[width=0.32\textwidth]{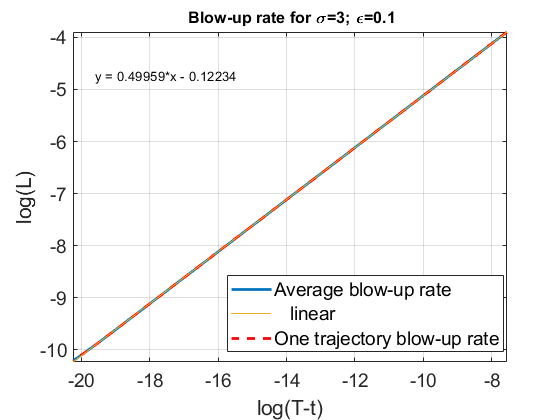} 
\includegraphics[width=0.32\textwidth]{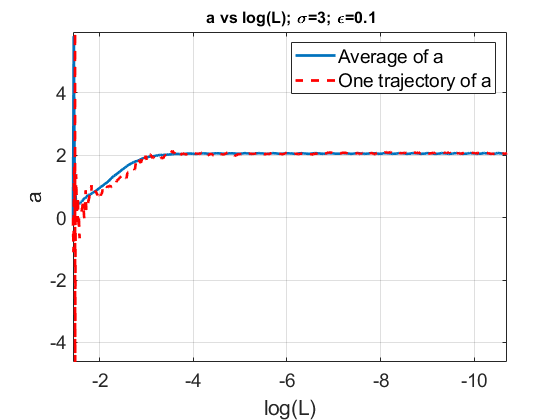}
\includegraphics[width=0.32\textwidth]{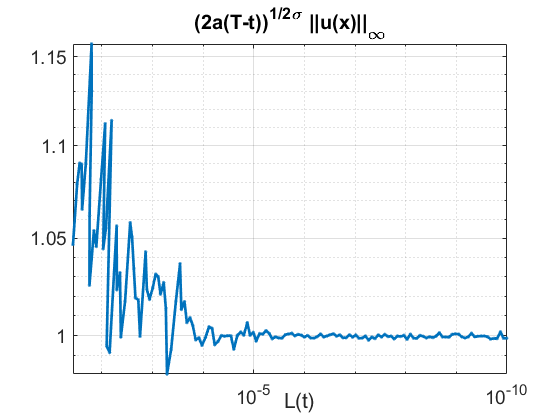} 
\caption{Blow-up rate tracking in Example 3 (Riesz kernel) with $\beta=0.5$ and $\epsilon = 0.1$. For details, see Figure \ref{blowup behavior EX1}. }
\label{blowup behavior}
\end{figure}

\clearpage

\bibliography{SNLS_bib} 
\bibliographystyle{abbrv}

\end{document}